\documentclass{article}
\usepackage[T1]{fontenc}
\usepackage{amsthm}
\usepackage{amsmath}
\usepackage{amssymb}
\usepackage{amsfonts}
\usepackage{eucal}
\usepackage[all]{xy}
\usepackage{hyperref}

\usepackage{mathrsfs}

\newtheorem{prop}{Proposition}[subsection]

\newtheorem{coro}[prop]{Corollary}

\newtheorem{lem}[prop]{Lemma}
\newtheorem*{lemm*}{Lemma}

\theoremstyle{definition}

\newtheorem{empt}[prop]{}
\newtheorem{dfn}[prop]{Definition}
\newtheorem{rem}[prop]{Remark}
\newtheorem{ntn}[prop]{Notation}

\newtheorem*{rem*}{Remark}

\theoremstyle{thm}
\newtheorem{thm}[prop]{Theorem}
\newtheorem*{thm*}{Theorem}
\newtheorem*{lem*}{Lemma}

\newtheorem*{cor*}{Corollary}
\newtheorem*{prop*}{Proposition}

\theoremstyle{dfn}
\newtheorem*{dfn*}{Definition}

\numberwithin{equation}{prop}

\newcommand{\riso}{ \overset{\sim}{\longrightarrow}\, }

\newcommand{\Spec}{\mathrm{Spec}\,}
\newcommand{\Spf}{\mathrm{Spf}\,}

\newcommand{\gr}{\mathrm{gr}}
\newcommand{\gp}{\mathrm{gp}}

\newcommand{\FF}{{\mathcal{F}}}
\newcommand{\B}{{\mathcal{B}}}

\newcommand{\E}{{\mathcal{E}}}
\newcommand{\G}{{\mathcal{G}}}
\renewcommand{\H}{{\mathcal{H}}}

\newcommand{\D}{{\mathcal{D}}}
\newcommand{\I}{{\mathcal{I}}}
\newcommand{\PP}{{\mathcal{P}}}

\renewcommand{\O}{{\mathcal{O}}}

\newcommand{\V}{\mathcal{V}}
\renewcommand{\S}{\mathcal{S}}

\newcommand{\Y}{\mathcal{Y}}
\newcommand{\ZZ}{\mathcal{Z}}
\newcommand{\X}{\mathfrak{X}}

\newcommand{\U}{\mathfrak{U}}

\newcommand{\A}{\mathbb{A}}

\newcommand{\F}{\mathbb{F}}

\newcommand{\DD}{\mathbb{D}}
\renewcommand{\L}{\mathbb{L}}
\newcommand{\R}{\mathbb{R}}
\newcommand{\Q}{\mathbb{Q}}
\newcommand{\Z}{\mathbb{Z}}
\newcommand{\N}{\mathbb{N}}

\newcommand{\hdag}{  \phantom{}{^{\dag} }    }

\def\debrom{
\makeatletter
\renewcommand{\theenumi}{(\roman{enumi})}
\renewcommand{\labelenumi}{\theenumi}
\makeatother\begin{enumerate}}
\def\finrom{\end{enumerate}}

\begin{document}

\title{Logarithmic $p$-bases and arithmetical differential modules}
\author{Daniel Caro and David Vauclair}
\date{}

\maketitle

\begin{abstract}

We introduce the notion of log $p$-smoothness which weakens that of log-smoothness and that of having locally $p$-bases.
We extend Berthelot's construction of arithmetic $\D$-modules and some properties in this context.

\end{abstract}

\tableofcontents

\bigskip

\section*{Introduction}
Berthelot's theory of arithmetic $\D$-modules in the context of varieties over a perfect field $k$ of characteristic $p>0$
gives a $p$-adic cohomology satisfying properties analogous to that in Grothendieck's theory of $l$-adic étale cohomology. 
For instance, we have the  stability under six operations of the overholonomicity with Frobenius structures (see \cite{caro-Tsuzuki})
and a theory of weights (see \cite{Abe-Caro-weights}) which are respectively the analogue of Grothendieck's stability of constructible $l$-adic sheaves
and of  Deligne's theory of weights of constructible $l$-adic complexes.
However, we also need a theory of arithmetic $\D$-modules in a wider context. 
For instance, 
nearby cycles or vanishing cycles 
require to work over certain schemes which are not varieties over $k$ (e.g. the spectrum of a henselian ring).
In relation to the $p$-adic local monodromy theorem
(see \cite{andreHA}, \cite{Kedlaya-localmonodromy}, \cite{mebkhout-monodr-padique}),
Crew studied the case of $k [[t]]/k$
(e.g. he proved the holonomicity of $F$-isocrystals on the bounded Robba ring in \cite{crew-arith-D-mod-curve}).
More recently, D. Pigeon in \cite{these_Pigeon} extended the construction of Berthelot's sheaf of differential operators
in the context of schemes having locally $p$-bases. 
Roughly speaking, the point is to extend in the theory the notion of étale morphisms by that of relative perfect morphisms.
For instance, since $k [[t]]/k[t]$ is relatively perfect, Pigeon's context also ``englobes'' Crew's case $k [[t]]/k$.
In fact, to be more precise, the objects in Crew's construction 
live over $\Spf \V [[t]]$ endowed with the $(p,t)$-adic topology whereas in Pigeon's context (and in our work here)
we only consider the $p$-adic topology. In other words, objects are living over the topological space of $\Spec k [[t]]$ in 
Pigeon or our context whereas in Crew's one they are living
over that of $\Spec k$. Taking global sections, one might consider some comparison theorems between both contexts ; 
but this is not the purpose of the paper.
On the other hand, it would also be useful (e.g. to be able to use Kedlaya's semistable reduction theorem (\cite{kedlaya-semistableIV}))
to work with log schemes
as in C. Montagnon's thesis
(see \cite{these_montagnon}).
This is the goal of  the present paper where we
introduce the notion of log $p$-smoothness that generalizes at the same time the notions of having $p$-bases locally and of log smoothness.
We extend in this work Berthelot's construction of arithmetic $\D$-modules in this context which generalizes Berthelot, 
Montagnon or Pigeon's one.
Over Laurent series fields, it would be interesting to compare our constructions with the theory of rigid cohomology as developed by
C. Lazda and A. P\'al
 (see \cite{Lazda-Pal-book}), which would be a stimulating continuation of the present paper.
 Moreover, we might hope to check some finiteness properties by using the framework of arithmetic $\D$-modules.

Let us describe the content of the paper.
In the first chapter, we introduce the notions of log $p$-étaleness (of level $m$), 
log $p$-basis (of level $m$), log $p$-smoothness and log relative perfectness.
These notions are some generalizations of the notions of (log-)étaleness, (log) basis, (log) smoothness and relative perfectness
see (\ref{let-lpet}, \ref{smooth=>psmooth}).
We also propose a notion of (finite) log $p$-basis which is suited in the context of $\D$-modules when we look at local descriptions. 
This latter notion is related with Tsuji's notion of $p$-basis defined in \cite[1.4]{Tsuji-nearby-log-smooth} (see remarque \ref{Tsuji-havingpbases}).
In the first chapter, we also retrieve  the usual local description
of $m$-PD-envelopes (see \cite{Be1} or \cite{these_Pigeon}) in the context of log $p$-smoothness (of level $m$).
In the second chapter we extend Berthelot's constructions of arithmetic $\D$-modules to log $p$-smooth fine log schemes.
Namely, we define and study $\D ^{(m)} _{X/S}$ (the sheaf of differential operators of level $m$)
and we put a canonical right  $\D ^{(m)} _{X/S}$-module structure on the dualizing sheaf $\omega _{X/S}$.
A standard elegant proof of this fact needs Grothendieck's extraordinary pullback as defined in \cite{HaRD}. 
Since we do not know a suitable generalization of Grothendieck's extraordinary pullback in our context,
we construct the canonical right  $\D ^{(m)} _{X/S}$-module structure
on $\omega _{X/S}$ by glueing local computations. 
In  the third chapter we move to log $p$-adic formal schemes by taking limits
as  in \cite{Be1}.
We conclude by defining extraordinary pull-backs, pushforwards and duals similarly to Berthelot in his theory of arithmetic $\D$-modules.

\subsection*{Acknowledgment}
The first author acknowledges Ambrus P\'al and Christopher Lazda for
an invitation to a conference in London where their talks were related to this paper.
The first author was supported by the IUF.

\subsection*{Convention, notation of the paper}
Let $\V$ be a complete discrete valued ring of mixed characteristic $(0,p)$, $\pi$ be a uniformizer,
$K$ its field of fractions,
$k$ its residue field which is supposed to be perfect.
If $X \to Y$ is a morphism of log schemes, we denote by
$\underline{f}\colon \underline{X} \to  \underline{Y}$ the underlying morphism of schemes.
A scheme means a log schemes endowed with the trivial log structure. 
By convention, a fine log scheme $X$ is noetherian means that $\underline{X}$ is noetherian as a scheme. 
A fs log scheme means a fine saturated log scheme.
The formal scheme $\Spf \, \V$ is meant for the $p$-adic topology.
A formal $\V$-scheme means a $p$-adic formal scheme over $\Spf \, \V$.
When we say ``etale locally'' this means we use Definition
\cite[IV.6.3]{sga3-1} of the etale topology (see also proposition \cite[IV.6.3.1(iv)]{sga3-1} for an alternative definition).
In the category of fine log schemes, 
replacing the morphisms of schemes by the strict morphisms of fine log schemes,
we get the similar notion of étale locality (see the remark \ref{convention}.3).

Unless otherwise stated
fiber products of fine log schemes (resp. fine formal log schemes) are always computed in the category of
fine log schemes (resp. fine formal log schemes). 
The sheaves of monoids are by convention sheaves for the étale topology.  Notice that since we are working with fine log schemes, 
from \cite[II.1.2.11]{Ogus-Logbook}, the reader who prefers to work with sheaves of monoids for the Zarisky topology can switch to it without any problem.
The remark \ref{rem-cohorfine?} should convince the reader why it is too pathological to work with  coherent log structures instead of fine log structures.

\section{Log $p$-étaleness, relative log perfection, $m$-PD envelope}
Let $i$ be an integer and
 $S$ be a fine log scheme over the scheme $\Spec (\Z / p ^{i+1}\Z)$.

\subsection{Formal log etaleness, $n$th infinitesimal neighborhood}
\label{subsection11}
In this subsection \ref{subsection11}, the reader might notice that we could consider the case where
$S =\Spec (\Z / p ^{i+1}\Z)$ without loss of generality. 
But, we will  need in the subsection \ref{subsection12} the case where  $S$ is any fine log scheme over $\Spec (\Z / p ^{i+1}\Z)$.
In order to keep notation as homogenous as possible, $S$ will remain any fine log scheme over $\Spec (\Z / p ^{i+1}\Z)$.
\begin{ntn}
If $X$ is a log scheme over $\Spec (\Z / p ^{i+1}\Z)$,
then for any integer $0\le k\le i$ we put
$X _{k} := X \times _{\Z / p ^{i+1}\Z} \Z/p^{k+1}\Z$.
If $f \colon X \to Y$ is a morphism of $\Spec (\Z / p ^{i+1}\Z)$-log schemes,
then for any integer $0\le k\le i$
we set 
$f _k \colon X _k \to Y _k$ the morphism canonically induced by $f$.

Let $\I$ be a quasi-coherent sheaf (for the Zariski topology) on a log scheme $X$. 
The preasheaf which associates $\I ( U) ^{(n)}$ (the subideal of $\I ( U)$ generated by $n$th powers of 
elements of $\I ( U)$)
 to an affine open set $U$ of $X$
is a quasi-coherent sheaf. 
We denote it by $\I ^{(n)}$.
Similarly, using \cite[A.1.5.(ii)]{Be2},  when $\I$ is endowed with an $m$-PD structure, 
we get a quasi-coherent sheaf 
$\I ^{\{ n\} _{(m)}}$ (which depends on the $m$-PD structure of $\I$) such that 
for an affine open set $U$ of $X$, 
$\I ^{\{ n\} _{(m)}} (U ) = 
\I (U)  ^{\{ n\} _{(m)}}$.

If $X$ is a fine log scheme over $\F _p$,
we denote by $F _X \colon X \to X$ the absolute Frobenius of $X$ as defined by Kato in
\cite[4.7]{Kato-logFontaine-Illusie}.

\end{ntn}

\begin{dfn}
\label{dfn-imm}
Let
$u \colon Z\to X$ be a morphism of log-schemes.
\begin{enumerate}
\item We say that $u$ is an immersion (resp. closed immersion) if
$\underline{u}$ is an immersion (resp. a closed immersion) of schemes  and if
$u ^{*} M _X  \to M _Z$ is surjective (here $u^*$ means the pullback of log structures \cite[1.4]{Kato-logFontaine-Illusie}).
An immersion (resp. a closed immersion) is exact if and only if 
$u ^{*} M _X  \to M _Z$ is an isomorphism.

\item We say that $u$ is an open immersion if $u$ is an exact immersion such that 
$\underline{u}$ is an open immersion of schemes.

\item Let $n$ be an integer.
A ``log thickening of order $(n)$'' (resp. ``log thickening of order $n$'') is an {\it exact} closed immersion
$u \colon U \hookrightarrow T$ such that $\I ^{(n)}=0$ (resp. such that $\I ^{n+1}= 0$),
where $\I$ is the ideal associated with the closed immersion $\underline{u}$.
The convention of the respective case is that of \cite[IV.2.1.1]{Ogus-Logbook} and is convenient when we are dealing with
$n$-infinitesimal neighborhood. 

\item Let $ a\in \N$. A ``$(p)$-nilpotent log thickening of order $a$'' is a log thickening of order $(p ^{a+1})$.
A ``$(p)$-nilpotent log thickening" is a $(p)$-nilpotent log thickening of order $a$ for some $a\in \N$ large enough.
\end{enumerate}

An $S$-immersion (resp. $S$-log thickening) is an immersion (resp. log thickening) which is an $S$-morphism.
  \end{dfn}

\begin{rem}
\label{rem-immersion}
\begin{enumerate}
\item If $u \colon Z \to X$ and $f \colon X \to Y$ are two $S$-morphisms of log schemes  such that
$f \circ u$ is an $S$-immersion, then so is $u$ (use \cite[5.3.13]{EGAI}).

\item If $u \colon Z \to X$ and $f \colon X \to Y$ are two $S$-morphisms of log schemes  such that
$f \circ u$ is a closed $S$-immersion and $\underline{f}$ is separated, then 
$u$ is a closed immersion (use \cite[5.4.4]{EGAI}).

\item \label{rem-immersion2} We  can decompose a (strict) $S$-immersion $u$ into
$u = u _1 \circ u _2$, where $u _1$ is an open $S$-immersion and $u _2$ is a (strict) closed $S$-immersion.

\item Let $u \colon U \hookrightarrow T$ be an $S$-log thickening of order $(p ^{a})$ for some integer $a$.
Since $p$ is nilpotent in $\O _T$,
by applying finitely many times the functor $\I \mapsto \I ^{(p)}$ to the ideal defining $\underline{u}$
we obtain the zero ideal, which justifies the definition of ``$(p)$-nilpotent log thickening''.
This also implies that $u$ is the composition of several $S$-log thickenings of order $(p) $.

\end{enumerate}

\end{rem}

\begin{dfn}
\label{dfnC}
\begin{enumerate}

\item
We denote by $\mathscr{C}$ the category
whose objects are $S$-immersions of fine log-schemes
and whose morphisms
$u'\to u$
are commutative diagrams of the form
\begin{equation}
\label{morpCpre}
\xymatrix{
{X'}
\ar[r] ^-{f}
&
{X}
\\
{Z'}
\ar@{^{(}->}[u] ^-{u'}
\ar[r]
&
{Z.}
\ar@{^{(}->}[u] ^-{u}}
\end{equation}
We say that
$u' \to u $ is strict (resp. 
flat, resp. cartesian)
if $f$ is strict (resp. $\underline{f}$ is flat, resp. 
the square \ref{morpCpre} is cartesian).

\item Let $n \in \N$.
We denote by $\mathscr{C} _{(n)}$ (resp. $\mathscr{C} _{n}$, resp. $\mathscr{T}hick _{(p)}$) the full subcategory
of $\mathscr{C}$ whose objects are
$S$-log thickenings of order $(n)$ (resp. $S$-log thickening of order $n$, resp. $(p)$-nilpotent $S$-log thickenings).

\item Let $u $ be an object of $\mathscr{C}$.
 An ``log thickening of order $(n)$ (resp. of order $n$) induced by $u$'' is
 an object $u'$ of $\mathscr{C} _{(n)}$ (resp. $\mathscr{C} _{n}$)
endowed with a morphism
$u' \to u$ of $\mathscr{C}$ 
satisfying the following universal property:
for any object $u''$ of $\mathscr{C} _{(n)}$ (resp. $\mathscr{C} _{n}$) endowed with a morphism
$f \colon u''\to u$ of $\mathscr{C} $
there exists a unique
morphism $u''\to u'$
of $\mathscr{C} _{(n)}$
(resp. $\mathscr{C} _{n}$)
whose
composition 
with $u' \to u$ is $f$.
The unicity up to canonical isomorphism
of the log thickening of order $(n)$ (resp. of order $n$) induced by $u$ is obvious. 
We will denote by 
$P ^{(n) }(u)$ (resp. $P ^{n} (u)$) the log thickening of order $(n)$ (resp. of order $n$) induced by $u$. 
We also say that $P ^{n} (u)$ is the $n$th infinitesimal neighbourhood of $u$ (see \cite[5.8]{Kato-logFontaine-Illusie}).
The existence is checked below (see \ref{PDenvelope} and 
\ref{PDenvelopebis}).

\item We denote by $\mathscr{C} ^{\mathrm{sat}} $  (resp. $\mathscr{C} _{(n)}^{\mathrm{sat}}$, resp. $\mathscr{C} _{n}^{\mathrm{sat}}$, resp. $\mathscr{T}hick _{(p)}^{\mathrm{sat}}$)
the full subcategory of
$\mathscr{C} $  (resp. $\mathscr{C} _{(n)}$, resp. $\mathscr{C} _{n}$, resp. $\mathscr{T}hick _{(p)}$)
whose objects are also morphisms of fs log-schemes.

\end{enumerate}

\end{dfn}

\begin{rem}
\label{fibprodC}
\begin{enumerate}
\item If $u' \to u$ is a strict cartesian morphism of $\mathscr{C}$
with $u \in \mathscr{C} _{(n)}$ (resp. $u\in \mathscr{C} _{n}$, resp. $u\in \mathscr{T}hick _{(p)}$),
then $u' \in \mathscr{C} _{(n)}$ (resp. $u' \in\mathscr{C} _{n}$, resp. $u' \in\mathscr{T}hick _{(p)}$).
Indeed, the corresponding square of the form \ref{morpCpre} of $u' \to u$ remains cartesian after applying the forgetful functor
from the category of fine log schemes to the category of schemes (to check this fact, we need a priori the strictness of $u'\to u$).
\item The category $\mathscr{C}$ has fibered products.
More precisely, let $u \colon Z \hookrightarrow X$,
$u '\colon Z' \hookrightarrow X'$,
$u ''\colon Z'' \hookrightarrow X''$ be some objects of
$\mathscr{C}$ ;
let  $u ' \to u$ and $u'' \to u$ be two morphisms of
$\mathscr{C}$. Then $u ' \times _{u} u ''$ is the immersion
$Z ' \times _{Z} Z '' \hookrightarrow     X ' \times _{X } X''$.
If $u' \to u$ is moreover cartesian, then so is 
the projection $u ' \times _{u} u '' \to u''$.

\end{enumerate}

\end{rem}

In order to be precise, let us clarify the standard definitions.
\begin{dfn}
\label{dfn-etale}
Let $f\colon X \to Y$ be an $S$-morphism of fine log schemes.
\begin{enumerate}

\item \label{1}  We say that $f$ is ``fine formally log étale'' (resp. ``fine formally log unramified'') if it satisfies the following property:
for any commutative diagram of fine log schemes of the form
\begin{equation}
\label{dfn-petale-squarepre}
\xymatrix{
{U}
\ar[r] ^-{u _0}
\ar@{^{(}->}[d] ^-{\iota}
&
{X}
\ar[d] ^-{f}
\\
{T}
\ar[r] ^-{v}
&
{Y}
}
\end{equation}
such that $\iota$ is an object of $\mathscr{C} _{1}$,
there exists a unique morphism (resp. there exists at most one morphism)
$u \colon T \to X$ such that $u\circ \iota = u _0$ and $f \circ u = v$.

\item Replacing $\mathscr{C} _1$ 
by $\mathscr{C} ^{\mathrm{sat}} _{1}$
we get the notion of
``fs formally log étale'' morphism
and of 
``fs formally log unramified'') morphism.

\item We say that $f$ is ``log étale'' if $f$ is fine formally log étale and if $\underline{f}$ is locally of finite presentation. 

\item We say that $f$ is ``étale'' if $f$ is log étale and strict (which is equivalent to saying that 
$\underline{f}$ is étale and $f$ is strict). 
\end{enumerate}
\end{dfn}

\begin{rem}

\begin{enumerate}
\label{convention}
\item The definitions appearing in \ref{dfn-etale} (or \ref{dfn-petale}) do not depend on the choice of 
 the fine log scheme $S$ over $\Z / p ^{i+1}\Z$.
 More precisely, let $f\colon X \to Y$ be an $S$-morphism of fine log schemes.
 Then we can consider $f$ as a $\Spec (\Z / p ^{i+1}\Z)$-morphism of fine log schemes. 
 We notice that the properties satisfied by $f$ can be checked
 equivalently when $S$ is equal to  $\Spec (\Z / p ^{i+1}\Z)$
 (i.e. we replace $S$ by $\Spec (\Z / p ^{i+1}\Z)$ 
 in the corresponding categories defined in \ref{dfnC}).

\item Let $f\colon X \to Y$ be an $S$-morphism of fine log schemes.
The notion of etaleness of Kato appearing in \cite[3.3]{Kato-logFontaine-Illusie}
is what we have defined as "log etaleness". We distingue by definition
``log etale'' and ``etale'' morphisms in order to avoid confusion when we say for instance
``etale locally''.

\item There exists in the literature a notion of etale morphism of log schemes with coherent log structures 
(see in  Ogus's book at \cite[IV.3.1.1]{Ogus-Logbook}). 
This notion is compatible with Kato's notion of etale morphism of fine log schemes.
Indeed, both notions have the same characterization when we focus on morphisms of fine log schemes 
(see respectively Theorem  \cite[3.5]{Kato-logFontaine-Illusie} and Theorem  \cite[IV.3.3.1]{Ogus-Logbook}). 

To avoid confusion with the etale notion in the classical sense, we will call such a morphism 
a ``log etale'' morphism of log schemes with coherent log structures (instead of ``etale morphism'').

Moreover, let 
 $f\colon X \to Y$ be a morphism of log schemes with coherent log structures.
From \cite[IV.3.1.12]{Ogus-Logbook}, 
$X ^{\mathrm{int}} \to X$ and $Y ^{\mathrm{int}} \to Y$
are log etale
(see \cite[II.2.4.5.2]{Ogus-Logbook} concerning the functor $X \mapsto X ^{\mathrm{int}}$). 
Hence, using Remark \cite[IV.3.1.2]{Ogus-Logbook},
we check that $f$ is log etale if and only if $f ^{\mathrm{int}}$ is log etale.

\end{enumerate}

\end{rem}

\begin{empt}
\label{ex-cl-imm}
We recall in the paragraph how we can exactify an immersion. 
Let $u \colon Z \hookrightarrow X$ be an $S$-immersion of fine log-schemes.
Let $\overline{z}$ be a geometric point of $Z$.
Using the proof of 
\cite[4.10.1]{Kato-logFontaine-Illusie}
and Proposition \cite[IV.18.1.1]{EGAIV4}, we check that there exists a commutative diagram of the form
\begin{equation}
\notag
\xymatrix{
{\widetilde{X}}
\ar[r] ^-{f}
&
{X'}
\ar[r] ^-{g}
\ar@{}[dr]|{\square}
&
{X}
\\
&
{Z'}
\ar@{_{(}->}[u] ^-{u'}
\ar@{^{(}->}[lu] ^-{v'}
\ar[r] ^-{h}
&
{Z}
\ar@{^{(}->}[u] ^-{u}}
\end{equation}
such that the square is cartesian, $f$ is log étale,
$\underline{f}$ is affine,
$g$ is étale,
$v'$ is an exact closed $S$-immersion
and $h$ is an étale neighborhood of
$\overline{z}$ in $Z$.

\end{empt}

\begin{lem}\label{prebasecgt-flat-env}
Let $u' \to u$ be a strict cartesian morphism of $\mathscr{C}$.
Suppose
that the $P ^{(n) }(u)$ (resp. $P ^{n} (u)$), 
the log thickening of order $(n)$ (resp. of order $n$) induced by $u$,
exists. 
Then the log thickening of order $(n)$ (resp. of order $n$) induced by $u'$ exists and we have 
$P ^{(n) }(u') = P ^{(n) }(u)  \times _{u} u'$
(resp. $P ^{n}(u') = P ^{n}(u)  \times _{u} u'$).
\end{lem}

\begin{proof}
Using the remarks of \ref{fibprodC}, 
since $u' \to u$ is strict and cartesian, then so is
the projection 
$P ^{(n) }(u)  \times _{u} u' \to P ^{(n) }(u) $
and then  
$P ^{(n) }(u)  \times _{u} u' \in \mathscr{C} _{(n)}$.
Hence, we check easily that 
$P ^{(n) }(u)  \times _{u} u'$ endowed with the projection
$P ^{(n) }(u)  \times _{u} u'\to u'$ satisfies the corresponding universal property
of $P ^{(n) }(u')$.
We check similarly the respective case. 
\end{proof}

\begin{lem}
\label{prePDenv-lrp}
Let $n\in \N$, $f \colon X \to Y$ be a fine formally log étale morphism of fine log $S$-schemes,
$u \colon Z \hookrightarrow X$ and $v \colon Z \hookrightarrow Y$ be two $S$-immersions of fine log schemes
such that $v= f\circ u$.
If $P ^{n }(u)$ exists, then 
$P ^{n}(v)$ exists and we have 
$P ^{n}(u)=P ^{n}(v)$.
\end{lem}

\begin{proof}
Abstract nonsense.
\end{proof}

\begin{prop}
\label{PDenvelope}
 For any integer $n$,
 the inclusion functor $For  _n \colon \mathscr{C} _n \to \mathscr{C}$
 has a right adjoint functor which we will denote by
$P ^{n }
\colon \mathscr{C} \to \mathscr{C}  _n $.
Let  $u\colon Z \hookrightarrow X$ be an object of
$\mathscr{C}$.
Then $Z$ is also the source of $P ^{n} (u) $.
Moreover,  denoting abusively by $P ^{n} (u)$ the target of the arrow
$P ^{n} (u)$, 
the underlying morphism of schemes of
$P ^{n} (u) \to X$ is affine.
We denote by $\PP ^{n} (u)$ the quasi-coherent $\O _X$-algebra so that
$\underline{P ^{n} (u)} = \mathrm{Spec} (\PP ^{n} (u))$.
If $X$ is noetherian, then so is $P ^{n} (u)$.
\end{prop}

\begin{proof}
The construction of $P ^{n }$ is given in \cite[5.8]{Kato-logFontaine-Illusie}.
Since Proposition \ref{PDenvelope} is slightly more precise than the existence of $P ^{n }$, 
for the reader convenience, let us give a detailled proof. 
Let $u\colon Z \hookrightarrow X$ be an $S$-immersion of fine log-schemes.
Using \ref{prebasecgt-flat-env}, the existence of $P ^{n} (u) $ (and then the whole proposition)
is étale local on $X$ (i.e. following our convention, this is local for the Zariski topology and 
we can proceed by 
descent of a finite covering with étale quasi-compact morphisms). 
Hence, by \ref{ex-cl-imm}, we may thus assume that
there exists a commutative diagram of the form
\begin{equation}
\notag
\xymatrix{
{\widetilde{X}}
\ar[r] ^-{f}
&
{X}
\\
&
{Z}
\ar@{_{(}->}[u] ^-{u}
\ar@{^{(}->}[lu] ^-{\widetilde{u}}
}
\end{equation}
such that $f$ is log étale, $\underline{f}$ is affine and $\widetilde{u}$ is an exact closed $S$-immersion.
Let $\I$ be the ideal defining $\widetilde{u}$.
Let $P ^{n}\hookrightarrow \widetilde{X}$
be the exact closed immersion which is induced by 
$\I ^{n+1}$.
Using 
\ref{prePDenv-lrp}, we check that 
$P ^{n}(u)$
is the exact closed immersion 
$Z \hookrightarrow  P ^{n}$.
When $X$ is noetherian, then so are $\widetilde{X}$
and $P ^n$.

\end{proof}

\begin{prop}
\label{formally log unramified}
Let $f\colon X \to Y$ be an $S$-morphism of fine log schemes and $\Delta _{X/Y}\colon X \hookrightarrow X \times _{Y}X$ (as always the product
is taken in the category of fine log schemes) be the
diagonal $S$-immersion.
The following assertions are equivalent : 
\begin{enumerate}
\item the morphism $P ^{1} (\Delta _{X/Y})$ is an isomorphism ;
\item the morphism $f$ is fine formally log unramified ; 
\item the morphism $f$ is formally log unramified (this notion is defined at \cite[IV.3.1.1]{Ogus-Logbook}).
\end{enumerate}
\end{prop}

\begin{proof}
Following \cite[IV.3.1.3]{Ogus-Logbook}, the last two assertions are equivalent. 
Moreover, by definition, if $f$ is formally log unramified then $f$ is fine formally log unramified.
Using \cite[5.8]{Kato-logFontaine-Illusie} and 
with Notation \ref{Omega1pre}, the property
$\Omega _{X/Y}=0$ is equivalent to say that $P ^{1} (\Delta _{X/Y})$ is an isomorphism. 
Copying the proof of ``if $f$ is formally log unramified then $\Omega _{X/Y}=0$''
of \cite[IV.3.1.3]{Ogus-Logbook} we check in the same way that 
if $f$ is fine formally log unramified then $\Omega _{X/Y}=0$ (indeed, since $X$ fine, 
then the log scheme $T:=X \oplus \Omega _{X/Y}$ is fine because its log structure 
is $M _X \oplus \Omega _{X/Y}$ : see \cite[IV.2.1.6]{Ogus-Logbook}).

\end{proof}

\subsection{Log $p$-étaleness}
\begin{dfn}
\label{dfn-petale}
Let $f\colon X \to Y$ be an $S$-morphism of fine log schemes.
\begin{enumerate}

\item \label{1}  We say that $f$ is ``log $p$-étale'' (resp. ``log $p$-unramified'') if it satisfies the following property:
for any commutative diagram of fine log schemes of the form
\begin{equation}
\label{dfn-petale-square}
\xymatrix{
{U}
\ar[r] ^-{u _0}
\ar@{^{(}->}[d] ^-{\iota}
&
{X}
\ar[d] ^-{f}
\\
{T}
\ar[r] ^-{v}
&
{Y}
}
\end{equation}
such that $\iota$ is an object of $\mathscr{C} _{(p)}$,
there exists a unique morphism (resp. there exists at most one morphism)
$u \colon T \to X$ such that $u\circ \iota = u _0$ and $f \circ u = v$.

\item Replacing $\mathscr{C} _{(p)}$ by 
$\mathscr{C} ^{\mathrm{sat}} _{(p)}$
we get the notion
``fs log $p$-étale'' 
and of 
``fs log $p$-unramified''. 

\item Replacing
``fine log $S$-schemes'' by ``$S$-schemes'' in the definition \ref{dfn-petale}.\ref{1},
we get the notion of ``$p$-étale'' (resp. ``$p$-unramified'') morphism of schemes.
\end{enumerate}
\end{dfn}

\begin{rem}
\label{rem-petale-Thick}
With the last remark of \ref{rem-immersion} in mind, we can replace
$\mathscr{C} _{(p)}$ by $\mathscr{T}hick _{(p)}$
(resp. $\mathscr{C} _{(p)} ^{\mathrm{sat}}$ by $\mathscr{T}hick _{(p)} ^{\mathrm{sat}}$) in the definition of log $p$-étale or log $p$-unramified
(resp. fs log $p$-étale or fs log $p$-unramified).

\end{rem}

\begin{lem}
\label{lem-lrp}
Let $f\colon X \to Y$ be an $S $-morphism of fine log-schemes.
Then $f$ is fs log $p$-étale 
(resp. fs formally log etale, resp. fs log $p$-unramified, resp. fs formally log unramified) 
if and only if so is $f ^{\mathrm{sat}}$.
\end{lem}

\begin{proof}
This is checked by using the fact that 
the functor $X \mapsto X ^{\mathrm{sat}}$ is a right adjoint of the inclusion functor from the category of fs log schemes to 
the category of fine log schemes
(see \cite[II.2.4.5.2]{Ogus-Logbook}).
\end{proof}

\begin{lem}
\label{logetale-etale}
Let $f \colon X\to Y$ be a strict $S$-morphism of fine log schemes,
$\underline{f}\colon \underline{X} \to \underline{Y} $,
$f ^{\mathrm{sat}}\colon X^{\mathrm{sat}} \to Y^{\mathrm{sat}}$
$\underline{f}^{\mathrm{sat}}\colon \underline{X} ^{\mathrm{sat}}\to \underline{Y} ^{\mathrm{sat}}$ be the induced morphisms
(see \cite[II.2.4.5.2]{Ogus-Logbook} concerning the functor $X \mapsto X ^{\mathrm{sat}}$).

\begin{enumerate}
\item The morphism $f$ is log $p$-étale if and only if $\underline{f}$ is $p$-étale.
\item  The morphism $f$ is fs log $p$-étale 
if and only if $\underline{f} ^\mathrm{sat}$ is $p$-étale.

\end{enumerate}

\end{lem}

\begin{proof}
If $f$ is log $p$-étale then this is straightforward that  $\underline{f}$ is $p$-étale
(from a diagram of the form \ref{dfn-petale-square} in the category of schemes, 
use the diagram of fine log schemes with strict morphisms which is induced by base change with $Y \to \underline{Y}$).
To check the converse, using the fact that any morphism 
$u \colon T \to Y$ of fine log schemes factorizes uniquely of the form
$T \to T' \to Y$ where $T' \to Y$ is strict and $\underline{T}=\underline{T}'$, we reduce to check the universal property of 
lop $p$-étaleness in the case where the morphisms of the diagram \ref{dfn-petale-square} are strict, which is clear. 
The second part of the Lemma is proved similarly.
\end{proof}

\begin{lem}
\label{rel-parf-stab2pre}
 Let $f\colon X \to Y$ and $g \colon Y '\to Y$ be two $S$-morphisms of fine log-schemes.
Set $X ':= X \times _{Y} Y'$ in the category of fine log schemes and $f' \colon X ' \to Y '$ the projection.
If $f$ is log $p$-étale (resp. fine formally log étale, resp. log étale, resp. fine formally log unramified,  resp. fs log $p$-étale, resp. fs formally log etale,
resp. fs formally log unramified),
then so is $f'$.
\end{lem}

\begin{proof}
Abstract nonsense and standard.
\end{proof}

\begin{lem}
\label{p-etale-modp}
Let $f\colon X \to Y$ be an $S$-morphism of fine log-schemes
and
 $f _0\colon X _{0} \to Y _0$ be the induced $S _0$-morphism.
The morphism $f $ is log $p$-étale (fs log $p$-étale)
if and only if $f$ is fine formally log etale (resp. fs formally log étale) and $f _0$ is
log $p$-étale (resp. fs log $p$-étale).
Similarly replacing everywhere "étale" by "unramified".
\end{lem}

\begin{proof}
If $f $ is log $p$-étale 
then by definition $f$ is fine formally log etale and by using \ref{rel-parf-stab2pre}
$f _0$ is
log $p$-étale.
Conversely, suppose that 
$f$ is fine formally log etale and $f _0$ is
log $p$-étale.
Let 
\begin{equation}
\label{diagpetale-modpdiag1}
\xymatrix{
{U}
\ar[r] ^-{u }
\ar@{^{(}->}[d] ^-{\iota}
&
{X}
\ar[d] ^-{f}
\\
{T}
\ar[r] ^-{w}
&
{Y}
}
\end{equation}
be a commutative diagram of fine log schemes
such that $\iota$ is an object of $\mathscr{C} _{(p)}$.
Since $f _0$ is log $p$-étale, 
there exists a unique morphism 
$\upsilon _0 \colon T _0 \to X _0$ such that $\upsilon _0\circ \iota _0 = u _0 $ and $f _0 \circ \upsilon _0 = w_0 $
(recall $f _0$, $w _0$, $\iota _0$, $u _0$
mean the reduction modulo $p$).
Let $\alpha _T \colon T _0 \hookrightarrow T$, 
$\alpha _X \colon X _0 \hookrightarrow X$, 
and
$\alpha _U \colon U _0 \hookrightarrow U$
be the canonical nilpotent exact closed immersions.
Since $\alpha _T$ is a nilpotent exact closed immersion, 
since $f$ is in particular fine formally log etale then 
there exists a unique morphism 
$\upsilon \colon T  \to X$ such that 
$\upsilon \circ \alpha _T = \alpha _X \circ \upsilon _0$
and $f \circ \upsilon  = w $.
Since $\alpha _U \colon U _0 \hookrightarrow U$ is a nilpotent exact closed immersion, 
since $f$ is in particular fine formally unramified, 
since $(\upsilon \circ \iota )\circ \alpha _U = u\circ \alpha _U $,
and
$f\circ (\upsilon \circ \iota) = f\circ u$,
then 
$\upsilon \circ \iota = u  $. 
Hence, the morphism
$\upsilon \colon T  \to X$ is such that 
$f \circ \upsilon  = w $
and $\upsilon \circ \iota = u  $. 
To check the uniqueness of such morphism $\upsilon \colon T  \to X$,
since $f _0$ is log $p$-unramified then 
$\upsilon _0$, the reduction of such $\upsilon$, is unique.
Since $f$ is  in particular fine formally unramified,
then such $\upsilon$ is unique. 
The unramified or fs cases are checked similarly. 

\end{proof}

\begin{lem}
\label{p-etale-stab1}
Let $f\colon X \to Y$ and $g\colon  Y \to Z$ be two $S$-morphisms of fine log schemes.
The morphisms $f$ and $g$ are log $p$-étale
(resp. fine formally log etale, resp. log étale, 
resp. fs log $p$-étale, resp. fs formally log etale)
if and only if so are $g\circ f$ and $g$. 
\end{lem}

\begin{proof}
Abstract nonsense and standard. 
\end{proof}

\begin{lem}
\label{p-etale-stab2bispre}
Let $f\colon X \to Y$ and $g \colon X '\to X$ be two $S$-morphisms of fine log-schemes such that $g$ is étale, quasi-compact and surjective.
The morphism $f$ is 
log $p$-étale 
(resp. fine formally log etale, resp. log étale, 
resp. fs log $p$-étale, resp. fs formally log etale)
 if and only so is $f \circ g$.
\end{lem}

\begin{proof}
First, let us prove the non respective case. 
From \ref{p-etale-stab1}, since an étale morphism is log $p$-étale, 
we check that 
if $f$ is log $p$-étale  then so is $f \circ g$.
Conversely, suppose $f \circ g$ is log $p$-étale.
Let 
\begin{equation}
\label{diagpetale-modpdiag2}
\xymatrix{
{U}
\ar[r] ^-{u _0}
\ar@{^{(}->}[d] ^-{\iota}
&
{X}
\ar[d] ^-{f}
\\
{T}
\ar[r] ^-{v}
&
{Y}
}
\end{equation}
be a commutative diagram of fine log schemes 
such that $\iota$ is an object of $\mathscr{C} _{(p)}$.
Put $U ' := U \times _{X} X'$ and $u ' _0 \colon U ' \to X'$ the morphism induced from $u _0$ by base change by $g$.
Since the projection $g' \colon U' \to U$ is étale, 
using Theorem \cite[IV.18.1.2]{EGAIV4}, 
there exists a unique (up to isomorphisms) étale morphism $h\colon T' \to T$
such that we have an isomorphism of the form
$U'\riso T' \times _{T} U$. 
Let $\iota ' \colon U ' \hookrightarrow T'$ 
be the projection. 
Since $f \circ g$ is log $p$-étale, there exists 
a unique morphism 
$u ' \colon T ' \to X'$ such that $u' \circ \iota' = u '_0$ and $f\circ g \circ u '=  v \circ h$.
Set $T'' := T ' \times _{T} T'$,
$U'' := U ' \times _{U} U'$.
Let $p _1\colon T'' \to T'$,
and $p _2\colon T'' \to T'$
(resp. $p _1\colon U'' \to U'$,
and $p _2\colon U'' \to U'$)
be respectively  the left and right projections.
Let $\iota '' := \iota ' \times \iota '\colon U '' \hookrightarrow T''$.  
Since 
$f \circ g \circ (u  ' \circ p _1)= f \circ g \circ  (u  ' \circ p _2)$,
$(u  ' \circ p _1) \circ \iota ''= (u  ' \circ p _2) \circ \iota ''$,
and since 
$f \circ g$ is log $p$-ramified, we
get 
$u  ' \circ p _1=   u  ' \circ p _2$
and then 
$( g \circ u  ') \circ p _1=  (g \circ  u  ') \circ p _2$.
Since
$T ' \to T$ is a strict epimorphism, this yields that
there exists a unique morphism $u\colon T \to X$ such that $g \circ u'= u \circ h$.
Since 
$u\circ \iota\circ g'= u _0\circ g'$, 
since $g'$ is étale and surjective,
then 
$u\circ \iota= u _0$.
Since 
$f \circ u \circ h= v \circ h$,
since $h$ is étale and surjective,
then 
$f \circ u = v$.
We conclude that 
$f$ is log $p$-étale.
The respective cases are checked similarly.

\end{proof}

\begin{lem}
\label{p-etale-stab2bis}
Let $f\colon X \to Y$ and $g \colon Y '\to Y$ be two $S$-morphisms of fine log-schemes such that $g$ is étale, quasi-compact and surjective.
Set $X ':= X \times _{Y} Y'$ in the category of fine log schemes and $f' \colon X ' \to Y '$ the projection.
The morphism $f$ is 
 log $p$-étale 
(resp. fine formally log etale, resp. log étale, 
resp. fs log $p$-étale, resp. fs formally log etale)
 if and only so is $f'$.
\end{lem}

\begin{proof}
Since the respective cases are similar, let us only prove the non respective case. 
From \ref{rel-parf-stab2pre}, 
if $f$ is log $p$-étale then so is $f'$.
Conversally, suppose that $f'$ is log $p$-étale. 
Let 
\begin{equation}
\label{diagpetale-modpdiag3}
\xymatrix{
{U}
\ar[r] ^-{u _0 }
\ar@{^{(}->}[d] ^-{\iota}
&
{X}
\ar[d] ^-{f}
\\
{T}
\ar[r] ^-{v}
&
{Y}
}
\end{equation}
be a commutative diagram of fine log schemes 
such that $\iota$ is an object of $\mathscr{C} _{(p)}$.
Put $T':= T \times _Y Y'$, 
$U':= U \times _Y Y'$, and let 
$\iota ' \colon U ' \hookrightarrow T'$,
$u' _0 \colon U' \to X'$, 
$v' \colon T' \to Y'$ be the morphism induced 
respectively from $\iota$, $u _0$, $v$ by base change by $g$.
Since $f'$ is log $p$-étale, there exists 
a unique morphism
$w ' \colon T ' \to X'$ such that $w' \circ \iota' = u '_0$ and $f ' \circ w '= v'$.
Set $Y'' := Y ' \times _{Y} Y'$, 
$T'':= T \times _Y Y''$, 
$U'':= U \times _Y Y''$,
and let 
$\iota '' \colon U ' \hookrightarrow T'$,
$u'' _0 \colon U'' \to X''$, 
$v'' \colon T'' \to Y''$, 
$f'' \colon X'' \to Y''$ be the morphism induced 
respectively from $\iota$, $u _0$, $v$, $f$ by base change by $Y'' \to Y$.
Let $p _1\colon Y'' \to Y'$,
$p _2\colon Y'' \to Y'$
(resp. $p _1\colon X'' \to X'$,
$p _2\colon X'' \to X'$, 
resp. 
$p _1\colon T'' \to T'$,
$p _2\colon T'' \to T'$)
be the left and right projections. 
Since $f''$ is log $p$-étale, there exists 
a unique morphism
$w '' \colon T '' \to X''$ such that $w'' \circ \iota'' = u ''_0$ and $f '' \circ w ''= v''$.
Since $f'$ is log $p$-étale,
we check 
$w' \circ p _1= p _1 \circ w''$,
$w' \circ p _2= p _2 \circ w''$
(e.g. use a cube with  \ref{diagpetale-modpdiag3} with primes and double primes as respectively the front and back squares).  
Hence, 
putting $u':=g \circ  w'\colon T' \to X$,
we get $u  ' \circ p _1= u  ' \circ p _2$.
Since $T ' \to T$ is a strict epimorphism, we conclude that there exists 
a unique morphism 
$u  \colon T  \to X$ such that 
$u '$ is the composition of $u$ with $T'\to T$.
Since $U ' \to U$ (resp. $T'\to T$) is étale and surjective then we check $u \circ \iota = u _0$
(resp. $f \circ u = v$).
Hence, $f$ is log $p$-étale.
\end{proof}

\begin{lem}
\label{lem-let-lpet}
Let $u\colon Z \hookrightarrow X$ be an object of 
$\mathscr{T}hick _{(p)}$
and $\I$ be the ideal defining the closed immersion $\underline{u}$.
Then $1+\I $ is a subgroup of $\O ^* _X$. 
Let $n$ be an integer prime to $p$.
The homomorphism
 $1 + \I \to 1 + \I$ of groups 
 defined by $x \mapsto x ^{n}$ is an isomorphism.
 Moreover, if $\underline X$ is affine then for any $q >0$, we have the vanishing
 $H ^{q} (X, 1 + \I)=0$.
\end{lem}

\begin{proof}
For $N$ large enough, $\I ^{(p ^N)}=0$ and $p ^N=0$ in $\O _S$. 
This yields that $1+\I $ is a subgroup of $\O ^* _X$. 
Since $(n, p^N)=1$, we obtain that 
the homorphism $1 + \I \to 1 + \I$ 
 given by $x \mapsto x ^{n}$ is an isomorphism.
Using \cite[VI.5.1]{sga4-1}, to check the last statement, we can suppose that $\I$ is nilpotent.
Hence, by devissage, we can reduce to the case where $\I ^{2}=0$. Then $(1+\I, \times)$ can be identified with $(\I, +)$ as a group.
Since $\I$ is quasi-coherent, we are done.

\end{proof}

\begin{prop}
\label{let-lpet}
Let $f\colon X \to Y$ be a log étale $S$-morphism of fine log-schemes.
Then $f$ is log $p$-étale.
\end{prop}

\begin{proof}
Following \ref{p-etale-stab2bispre} and \ref{p-etale-stab2bis}
the log $p$-étaleness is étale local in both $X$ and $Y$.
Hence, 
using \cite[Theorem 3.5]{Kato-logFontaine-Illusie},
we reduce to the case where
$X=A_P$, $Y=A_Q$ and 
where there exists a chart of $f$ subordinate to a morphism $\phi \colon Q\to P$ of fine monoids (see the definition \cite[II.2.1.7]{Ogus-Logbook})
 such that the kernel and cokernel of $\phi ^\mathrm{gp}$ is finite of order prime to $p$ (i.e. is invertible in $\Z /p ^{i+1}\Z$).
Let $\iota \colon U \hookrightarrow T$ be an object of $\mathscr{C} _{(p)}$. A morphism $\iota \to f$ 
can be thought of as commutative diagram
\begin{equation}
\label{let-lpet-diag1}
\notag
\xymatrix{
{Q}
\ar[r] ^-{\theta}
\ar[d] ^-{h}
&
{P}
\ar[d] ^-{g}
\ar@{.>}[ld] ^-{\widetilde{g}}
\\
{\Gamma (T, M_T)}
\ar[r] ^-{i}
&
{\Gamma (U, M _U).}
}
\end{equation}
We need to check the existence and unicity of 
a map 
$\widetilde{g}
\colon 
P \to \Gamma (T, M_T)$
such that 
$i\circ \widetilde{g}=g$
and 
$\widetilde{g} \circ \theta =h$.
Since this is locally étale, 
we can suppose $T$ affine. 
Following 
\cite[IV.2.1.2.4]{Ogus-Logbook},
the natural map 
$M _T \to M _T ^{\mathrm{gr}} \times _{M _U} M _U ^{\mathrm{gr}}$
is an isomorphism.
Moreover, the morphism
$ M_T \hookrightarrow M_T ^{\mathrm{gr}}$ is injective.
Hence, 
we reduce to check there exists a unique 
morphism $\widetilde{g}
\colon 
P ^{\mathrm{gr}} \to \Gamma (T, M_T ^{\mathrm{gr}})$
making commutative the diagram
\begin{equation}
\label{let-lpet-diag2}
\xymatrix{
{Q ^{\mathrm{gr}}}
\ar[r] ^-{\theta ^{\mathrm{gr}}}
\ar[d] ^-{h'}
&
{P ^{\mathrm{gr}}}
\ar[d] ^-{g'}
\ar@{.>}[ld] ^-{\widetilde{g'}}
\\
{\Gamma (T, M_T ^{\mathrm{gr}})}
\ar[r] ^-{i'}
&
{\Gamma (U, M _U ^{\mathrm{gr}}),}
}
\end{equation}
where $h'$, $g'$ are the morphism canonically  induced from $h$, $g$ 
and where $i'$ is induced from the map $M_T ^{\mathrm{gr}} \to M_U ^{\mathrm{gr}}$.
By \cite[IV.2.1.2.2]{Ogus-Logbook}, 
we have
$1+\I = \ker (M _T ^{\mathrm{gr}} \to M _U ^{\mathrm{gr}} )$, 
where 
$\I$ is the ideal defining the closed immersion $\underline{\iota}$.
Using \ref{lem-let-lpet}, since $T$ is affine, we get the exact sequence 
\begin{equation}
\label{let-lpet-ES1}
1 \to \Gamma ( T, 1+ \I) \to 
\Gamma ( T, M _T ^{\mathrm{gr}} )
\to 
\Gamma ( U, M _U ^{\mathrm{gr}} )
\to 1.
\end{equation}
First, let us check the unicity.
Let $\widetilde{g} _1,\widetilde{g} _2
\colon 
P ^{\mathrm{gr}} \to \Gamma (T, M_T ^{\mathrm{gr}})$
be two morphisms making the diagram \ref{let-lpet-diag2} commutative. 
From the exactness \ref{let-lpet-ES1}, 
we get $\widetilde{g} _1\widetilde{g} _2 ^{-1}
\colon 
P ^{\mathrm{gr}} 
\to 
\Gamma ( T, 1+ \I) $.
Since 
$(\widetilde{g} _1\widetilde{g} _2 ^{-1}) \circ \theta ^{\mathrm{gr}}=1$, 
the morphism has a factorization by a morphism of the form 
$\mathrm{coker} (\theta ^{\mathrm{gr}}) \to 
\Gamma ( T, 1+ \I) $. Since 
$\mathrm{coker} (\theta ^{\mathrm{gr}})$ is finite of order prime to $p$, 
since the homomorphism $\Gamma (X, 1 +\I) \to \Gamma (X, 1 +\I)$ 
given by $x \mapsto x ^n$ is a bijection for any integer $n$ prime to $p$ (see Lemma \ref{lem-let-lpet})
then  we check that
such morphism 
$\mathrm{coker} (\theta ^{\mathrm{gr}}) \to 
\Gamma ( T, 1+ \I) $ is $1$, i.e. 
$\widetilde{g} _1 =\widetilde{g} _2 $.

Now, let us prove the existence. 
It is sufficient to copy  in the proof of \cite[IV.3.1.8]{Ogus-Logbook} 
the part corresponding to the implication $3.1.8.(1)\Rightarrow 3.1.8.(2)$. 
For the convenience of the reader, 
let us clarify it. 
Put $E:=  
\Gamma ( T, M _T ^{\mathrm{gr}} )
\times  _{\Gamma ( U, M _U ^{\mathrm{gr}} )}
P ^{\mathrm{gr}}$ (the morphisms used to define the fiber product are those appearing in the diagram \ref{let-lpet-diag2}).
Taking the pullback of  \ref{let-lpet-ES1}  by $g'$, we get the exact sequence 
$1 \to \Gamma ( T, 1+ \I) \to 
E
\overset{\pi}{\longrightarrow} 
P ^{\mathrm{gr}}
\to 1$,
where the first map is given by 
$x \mapsto (x,1)$ and where the map $\pi$
is the projection $(x, y) \to y$.
We put 
$\phi : = ( h', \theta ^{\mathrm{gr}})
\colon Q ^{\mathrm{gr}} \to E $. 
We get the commutative diagram 
\begin{equation}
\notag
\xymatrix{
{1}
\ar[r] ^-{} 
& 
{\ker  (\theta ^{\mathrm{gr}})} 
\ar[r] ^-{} \ar[d] ^-{} 
&
{Q ^{\mathrm{gr}}}
\ar[r] ^-{} \ar[d] ^-{\phi} \ar[rd] ^-{\theta ^{\mathrm{gr}}} 
&
{Q ^{\mathrm{gr}}/\ker  (\theta ^{\mathrm{gr}}}
\ar[r] ^-{} \ar[d] ^-{} 
&
{1}
\\ 
{1} 
\ar[r] ^-{} 
& 
{\Gamma ( T, 1+ \I) } 
\ar[r] ^-{} \ar@{=}[d] ^-{} 
& 
{E}
\ar[r] ^-{\pi} \ar[d] ^-{} 
&
{P ^{\mathrm{gr}}}
\ar[r] ^-{} \ar[d] ^-{} 
&
{1} 
\\
{1} 
\ar[r] ^-{} 
& 
{\Gamma ( T, 1+ \I) } 
\ar[r] ^-{} 
&
{\mathrm{coker} (\phi)}
\ar[r] ^-{} 
&
{\mathrm{coker} (\theta ^{\mathrm{gr}})}
\ar[r] ^-{} 
&
{1} 
} 
\end{equation}
whose top and middle rows are exact. 
Since the homomorphism $\Gamma (X, 1 +\I) \to \Gamma (X, 1 +\I)$ 
given by $x \mapsto x ^n$ is a bijection for any integer $n$ prime to $p$,
since $\ker  (\theta ^{\mathrm{gr}})$ is finite of order prime to $p$, then the morphism 
$\ker  (\theta ^{\mathrm{gr}})
\to \Gamma ( T, 1+ \I) $ is equal to $1$.
Hence, using the snake Lemma to the top and middle rows, 
we check that the bottom row is also exact. 
Since the homomorphism $\Gamma (X, 1 +\I) \to \Gamma (X, 1 +\I)$ 
given by $x \mapsto x ^n$ is a bijection for $n$ equal to the cardinal of $\mathrm{coker}  (\theta ^{\mathrm{gr}})$,
then $\mathrm{Ext} ^1 (\mathrm{coker}  (\theta ^{\mathrm{gr}}),  \Gamma (X, 1 +\I))=1$.
Hence, the bottom  row splits (and even uniquely). 
Let $\tau \colon P ^{\mathrm{gr}}
\to \mathrm{coker} (\phi)$ be the composition 
of $P ^{\mathrm{gr}} \to \mathrm{coker} (\theta ^{\mathrm{gr}})$ with a section 
$\mathrm{coker} (\theta ^{\mathrm{gr}})\to \mathrm{coker} (\phi)$ of the surjection of the bottom surjective map.

We remark that the set of morphisms $\widetilde{g}
\colon 
P ^{\mathrm{gr}} \to \Gamma (T, M_T ^{\mathrm{gr}})$
making commutative the diagram
\ref{let-lpet-diag2}
is equipotent with 
the set of sections
$\sigma \colon P ^{\mathrm{gr}} \to E$ of $\pi$
such that $\sigma \circ \theta ^{\mathrm{gr}} =\phi $
(the bijection is given by
$\widetilde{g} \mapsto (\widetilde{g}, id)$).
Since the middle and bottow rows are exact, 
we get that the square on the bottom right is exact. 
Hence, we get the morphism 
$\sigma := (\tau, \pi) \colon P ^{\mathrm{gr}} \to E$. 
We check that this morphism is a section of $\pi$ 
such that 
$\sigma \circ \theta ^{\mathrm{gr}} =\phi $.
\end{proof}

\begin{lem}
\label{prePDenv-lrpbis}
Let $n\in \N$, $f \colon X \to Y$ be a log $p$-étale morphism of fine log $S$-schemes,
$u \colon Z \hookrightarrow X$ and $v \colon Z \hookrightarrow Y$ be two $S$-immersions of fine log schemes
such that $v= f\circ u$.
If $P ^{(p ^n) }(u)$ exists, then 
$P ^{(p ^n) }(v)$ exists and we have 
$P ^{(p ^n) }(u)=P ^{(p ^n) }(v)$.
\end{lem}

\begin{proof}
Abstract nonsense.
\end{proof}

\begin{prop}
\label{PDenvelopebis}
 For any integer $n$,
 the canonical functor 
$\mathscr{C} _{(p ^n)} \to \mathscr{C}$
 has a right adjoint functor which we will denote by
$P ^{(p ^n)} \colon \mathscr{C} \to \mathscr{C}  _{(p ^n)} $.
Let  $u\colon Z \hookrightarrow X$ be an object of
$\mathscr{C}$.
Then $Z$ is also the source of $P ^{(p ^n)}(u)$. 
Moreover,  denoting abusively by $P ^{(p ^n)} (u)$ the target of the arrow
$P ^{(p ^n)}(u)$, 
the underlying morphism of schemes of
$P ^{(p ^n)} (u) \to X$ is affine.
When $X$ is noetherian, then so is $P ^{(p ^n)} (u)$.
\end{prop}

\begin{proof}
The proof is similar to that of \ref{PDenvelope}:
Let $u\colon Z \hookrightarrow X$ be an $S$-immersion of fine log-schemes.
Using \ref{prebasecgt-flat-env}, the existence of $P ^{(p ^n)} (u) $ (and then the proposition)
is étale local on $X$ (i.e. following our convention, this is local for the Zariski topology and 
we can proceed by 
descent of a finite covering with étale quasi-compact morphisms). 
Hence, by \ref{ex-cl-imm}, we may thus assume that
there exists a commutative diagram of the form
\begin{equation}
\notag
\xymatrix{
{\widetilde{X}}
\ar[r] ^-{f}
&
{X}
\\
&
{Z}
\ar@{_{(}->}[u] ^-{u}
\ar@{^{(}->}[lu] ^-{\widetilde{u}}
}
\end{equation}
such that $f$ is log étale, $\underline{f}$ is affine and $\widetilde{u}$ is an exact closed $S$-immersion.
Let $\I$ be the ideal defining $\widetilde{u}$.
Let $P ^{(p ^n)}\hookrightarrow \widetilde{X}$
be the exact closed immersion which is induced by 
$\I ^{(p ^n)}$.
Using \ref{let-lpet} and \ref{prePDenv-lrpbis}, we check that 
$P ^{(p ^n)}(u)$
is the exact closed immersion 
$Z \hookrightarrow  P ^{(p ^n)}$.
\end{proof}

\subsection{Log relatively perfectness}

The following definition \ref{dfn-lrpFp} will be extended in \ref{dfn-lrp}.
\begin{dfn}
\label{dfn-lrpFp}
Let $f\colon X \to Y$ be an $S _0$-morphism of fine log-schemes.
We say that $f $ is ``fine log relatively perfect''
(resp. ``fs log relatively perfect'' ) if the diagram on the left (resp. on the right)
\begin{equation}
\label{cart-logrp}
\xymatrix{
{X }
\ar[r] ^-{F _{X }}
\ar[d] ^-{f }
&
{X }
\ar[d] ^-{f }
\\
{Y }
\ar[r] ^-{F _{Y }}
&
{Y ,}
}
\xymatrix{
{X ^{\mathrm{sat}}}
\ar[r] ^-{F _{X ^{\mathrm{sat}}}}
\ar[d] ^-{f ^{\mathrm{sat}}}
&
{X ^{\mathrm{sat}}}
\ar[d] ^-{f ^{\mathrm{sat}}}
\\
{Y ^{\mathrm{sat}}}
\ar[r] ^-{F _{Y ^{\mathrm{sat}}}}
&
{Y ^{\mathrm{sat}}}
}
\end{equation}
is cartesian in the category of fine log-schemes (resp. fs log-schemes).
This definition does not depend on the choice of 
 the fine log scheme $S$ over $\Z / p ^{i+1}\Z$.
We remark that $f$ is fs log relatively perfect if and only if so is $f ^{\mathrm{sat}}$.

\end{dfn}

\begin{rem}
\label{varpi}
Let $\iota \colon U \hookrightarrow T$ be a log $S _0$-thickening of order $(p)$.
Let $T \times _{ F _T, T , \iota}U$ be the base change of $F _T$ by $\iota$.
Let 
$p _1 \colon T \times _{ F _T, T , \iota}U \to T$,
and $p _2 \colon T \times _{ F _T, T , \iota}U \to U$
be the projections.   
Since $\iota$ is of order $(p)$ then $p _1$ is an isomorphism.  
Put $\varpi _\iota := p _2 \circ p _1 ^{-1}\colon T \to U$.
We remark that the morphism $\varpi _\iota$ is the unique morphism $T\to U$ making commutative the diagram
\begin{equation}
\label{varpidiag}
\xymatrix{
{U}
\ar[r] ^-{F _U}
\ar@{^{(}->}[d] _-{\iota}
&
{U}
\ar@{^{(}->}[d] ^-{\iota}
\\
{T}
\ar[r] ^-{F _T}
\ar[ur] ^-{\varpi _\iota}
&
{T.}
}
\end{equation}
\end{rem}

\begin{lem}
\label{rlp-fle}
Let $f\colon X \to Y$ be an  $S _0$-morphism of fine log-schemes.
If $f$ is fine log relatively perfect (resp. fs log relatively perfect)
then $f$ is log $p$-étale
(resp. fs log $p$-étale).
\end{lem}

\begin{proof}
Let us check the fine version.
Let
\begin{equation}
\xymatrix{
{U}
\ar[r] ^-{u _0}
\ar@{^{(}->}[d] ^-{\iota}
&
{X}
\ar[d] ^-{f}
\\
{T}
\ar[r] ^-{v}
&
{Y}
}
\end{equation}
be a commutative diagram of fs $S _0$-log schemes
such that $i$ is a log $S _0$-thickening of order $(p)$.
First, let us check the unicity.
Let $u \colon T \to X$ be a morphism
such that $u\circ \iota = u _0$ and $f \circ u = v$.
With the notation of the remark \ref{varpi},
we get
$F  _X \circ u = u \circ F  _T = u \circ \iota \circ \varpi _\iota = u _0 \circ  \varpi _\iota$.
Since we have also $f \circ u = v$, we obtain the uniqueness of $u$ from the cartesianity of \ref{cart-logrp}.
Now, let us check the existence.
We have
$F  _Y \circ v = v \circ F  _T = v \circ \iota \circ \varpi _\iota = f \circ u _0 \circ \varpi _\iota$.
 Hence, via the cartesianity of \ref{cart-logrp}, we get the $Y$-morphism
$u = ( v , u _0 \circ  \varpi _\iota )\colon T \to X= Y \times _{F  _Y, Y,f}X$.
By definition, $f \circ u = v$. Moreover,
$u \circ \iota = ( v \circ \iota, u _0 \circ  \varpi _\iota \circ \iota )
=
( f \circ u _0, u _0 \circ  F  _U)
=
( f \circ u _0, F  _X \circ  u _0 )
=
u _0\colon T' \to X= Y \times _{F  _Y, Y,f}X$.

Suppose now that $f$ is fs log relatively perfect.
Using the lemma \ref{lem-lrp}, we reduce to check that
 $f ^{\mathrm{sat}}$ is fs log $p$-étale.
 Since $f ^{\mathrm{sat}}$ is fs log relatively perfect, 
we can proceed in the same way by replacing ``fine'' by ``fs''.
\end{proof}

\begin{dfn}
\label{dfn-lrp}
Let $f\colon X \to Y$ be an $S$-morphism of fine log-schemes.
We say that $f $ is ``fine log relatively perfect''
(resp. ``fs log relatively perfect'') if $f$ is fine formally log etale (resp. fs formally log etale) and if $f _0$ is
fine log relatively perfect (resp. fs log relatively perfect).
This definition does not depend on the choice of 
 the fine log scheme $S$ over $\Z / p ^{i+1}\Z$.
\end{dfn}

\begin{rem}
From lemma \ref{rlp-fle}, this definition of log relative perfectness of \ref{dfn-lrp} agrees that of \ref{dfn-lrpFp} when $i =0$, i.e. $S =S _0$.
\end{rem}

\begin{lem}
\label{rem-lrplogetale}
Let $f\colon X \to Y$ be an $S $-morphism of fine log-schemes.
Then $f$ is fs log relatively perfect if and only if so is $f ^{\mathrm{sat}}$.
\end{lem}

\begin{proof}
From \ref{lem-lrp}, $f$ is fs formally log étale  if and only if so is $f ^{\mathrm{sat}}$.
Moreover, since $(f ^{\mathrm{sat}}) _0 = (f _0 )^{\mathrm{sat}}$, then 
$f _0$ is fs log relatively perfect if and only if so is $(f ^{\mathrm{sat}}) _0$.
\end{proof}

\begin{prop}
\label{rlp-fle-prop}
A fine (resp. fs) log relatively perfect morphism
is  log $p$-étale (resp. fs log $p$-étale).
\end{prop}

\begin{proof}
This is a consequence of  \ref{p-etale-modp} and \ref{rlp-fle}.
\end{proof}

\begin{lem}
\label{rel-parf-stab2}
 Let $f\colon X \to Y$ and $g \colon Y '\to Y$ be two $S$-morphisms of fine log-schemes.
Set $X ':= X \times _{Y} Y'$ in the category of fine log schemes and $f' \colon X ' \to Y '$ the projection.
If $f$ is fine log relatively perfect (resp. fs log relatively perfect),
then so is $f'$.
\end{lem}

\begin{proof}
Since the fs part is similar, 
let us only consider the fine part.
Using \ref{rel-parf-stab2pre}, we reduce to the case where $i=0$.
Since the outline of both diagrams
\begin{equation}
\label{cart-logrp2}
\xymatrix{
{X'}
\ar[r] ^-{F _{X'}}
\ar[d] ^-{f'}
&
{X'}
\ar[r] ^-{}
\ar[d] ^-{f'}
\ar@{}[rd] ^-{}|\square
&
{X}
\ar[d] ^-{f}
\\
{Y'}
\ar[r] ^-{F _{Y'}}
&
{Y'}
\ar[r] ^-{}
&
{Y},
}
\;
\
\xymatrix{
{X'}
\ar[r] ^-{}
\ar[d] ^-{f'}
\ar@{}[rd] ^-{}|\square
&
{X}
\ar[r] ^-{F _X}
\ar[d] ^-{f}
\ar@{}[rd] ^-{}|\square
&
{X}
\ar[d] ^-{f}
\\
{Y'}
\ar[r] ^-{}
&
{Y}
\ar[r] ^-{F _Y}
&
{Y}
}
\end{equation}
are the same, we conclude.
\end{proof}

\begin{lem}
\label{strict-lrp}
Let $f\colon X \to Y$ be a strict $S$-morphism of fine log schemes
(resp. fs log schemes).
Then $f$ is fine log relatively perfect (resp. fs log relatively perfect)
if and only if
$\underline{f}$ is relatively perfect as defined by Kato in
\cite[1.1]{Kato-explicity-recip91}.

\end{lem}

\begin{proof}
Since $f$ is strict then $f$ is the base change of $\underline{f}$ by $Y \to \underline{Y}$ in the category of fine log schemes (resp. fs log schemes).
This yields that if $\underline{f}$ is relatively perfect then 
$f$  is fine log relatively perfect (resp. fs log relatively perfect).
Conversely, suppose that 
the morphism $f$ is fine (resp. fs) formally log étale.
Then, similarly to the proof of 
\ref{logetale-etale}, we check that $\underline{f}$  is formally étale.
Hence, we reduce to the case $i=0$.
Moreover, since $f$ is strict
then
$\underline{Y \times _{ F _Y, Y , f} X}
=
\underline{Y} \times _{ F _{\underline{Y}}, \underline{Y} , \underline{f}} \underline{X}$. 
Since $f$ is fine log relatively perfect (resp. fs log relatively perfect), we get
$Y=Y \times _{ F _Y, Y , f} X$. This yields
$\underline{Y}=\underline{Y} \times _{ F _{\underline{Y}}, \underline{Y} , \underline{f}} \underline{X}$ and we are done.
\end{proof}

\begin{lem}
\label{rel-parf-stab1}
Let $f\colon X \to Y$ and $g\colon  Y \to Z$ be two $S$-morphisms of fine log schemes.
The morphisms $f$ and $g$ are 
fine log relatively perfect
(resp. fs log relatively perfect)
if and only if so are $g\circ f$ and $g$. 
\end{lem}

\begin{proof}
Using \ref{p-etale-stab1}, we reduce to the case where $i=0$. 
Then, this is abstract nonsense.
\end{proof}

\begin{lem}
\label{et-flrp}
Let $f\colon X \to Y$ be an étale $S$-morphism of fine log-schemes.
Then $f$ is fine log relatively perfect.
\end{lem}

\begin{proof}
Using 
\ref{strict-lrp}, we reduce to check that
an étale morphism of schemes is relatively perfect
as defined by Kato in
\cite[1.1]{Kato-explicity-recip91}, which is well known.
\end{proof}

\begin{lem}
\label{rel-parf-stab2bispre}
Let $f\colon X \to Y$ and $g \colon X '\to X$ be two $S$-morphisms of fine log-schemes such that $g$ is étale, quasi-compact and surjective.
The morphism $f$ is 
fine log relatively perfect 
(resp. fs log relatively perfect) if and only so is $f \circ g$.
\end{lem}

\begin{proof}
From \ref{rel-parf-stab1} and \ref{et-flrp}, if $f$ is fine log relatively perfect then so is $f \circ g$.
Conversely, suppose $f \circ g$ is fine log relatively perfect.
Using \ref{p-etale-stab2bispre}, we reduce to the case where $i=0$. 
We have to check that the morphism 
$(F _X, f) \colon X \to X \times _{f, Y ,F _Y}Y$ 
is an isomorphism. Since is Zariski local, we can suppose $X$ affine.
Since $f \circ g$ is fine log relatively perfect, 
then $(F _{X'}, f\circ g) \colon X '\to X '\times _{f\circ g, Y ,F _Y}Y$, 
is an isomorphism.
We notice that  $(F _{X'}, f\circ g)$, 
is the base change of $(F _X, f)$ by $g\times id \colon X '\times _{f\circ g, Y ,F _Y}Y \to X \times _{f, Y ,F _Y}Y$
Since $g\times id$ is etale, quasi-compact and surjective, then 
using \cite[IV.2.7.1.(viii)]{EGAIV2}, 
we can conclude.
The fs log relatively perfect case is checked similarly.
\end{proof}

\begin{lem}
\label{rel-parf-stab2bis}
Let $f\colon X \to Y$ and $g \colon Y '\to Y$ be two $S$-morphisms of fine log-schemes such that $g$ is étale, quasi-compact and surjective.
Set $X ':= X \times _{Y} Y'$ in the category of fine log schemes and $f' \colon X ' \to Y '$ the projection.
The morphism $f$ is 
fine log relatively perfect 
(resp. fs log relatively perfect) if and only so is $f'$.
\end{lem}

\begin{proof}
From \ref{rel-parf-stab2},
 $f$ is fine log relatively perfect then so is $f'$.
Let us check converse: suppose $f'$ is fine log relatively perfect. 
Using \ref{p-etale-stab2bis}, we reduce to the case where $i=0$. 
We have to check that the morphism 
$(F _X, f) \colon X \to X \times _{f, Y ,F _Y}Y$ 
is an isomorphism.
Since $f'$ is fine log relatively perfect then 
the base change of $(F _X, f)$ by $id \times g\colon 
X '\times _{f', Y ',F _{Y'}}Y '
=
X \times _{f, Y ,F _Y\circ g}Y ' 
\to 
X \times _{f, Y ,F _Y}Y
$ is an isomorphism (to check the equality, recall from \ref{et-flrp} that $g$ is relatively perfect).  
Hence, using \cite[IV.2.7.1.(viii)]{EGAIV2}, 
we check that $(F _X, f)$ 
is an isomorphism.
The fs log relatively perfect case is checked similarly.
\end{proof}

\begin{ntn}
\label{ntnAp}
Let $P$ be a monoid.
We denote by $A _P:= (\Spec (\Z / p ^{i+1}\Z[ P ]), M _P)$
 the log formal scheme
whose underlying scheme is $\Spec (\Z / p ^{i+1}\Z [ P ])$ 
and whose log structure is the log structure associated with the pre-log structure
induced canonically by $P \to \Z / p ^{i+1}\Z [ P ]$.

Beware that the notation $A _P:= (\Spec (\Z [ P ]), M _P)$ seems more common in the literature, 
but since we are working over $\Z / p ^{i+1}\Z$ this is much more convenient for us
to put $A _P:= (\Spec (\Z / p ^{i+1}\Z[ P ]), M _P)$.

\end{ntn}

\begin{prop}
\label{let-lrp}
Let $f\colon X \to Y$ be a log étale $S$-morphism of fine log-schemes.
Then $f$ is fs log relatively perfect.
\end{prop}

\begin{proof}
Since these notions do not depend on $S$, 
we can suppose 
$S= \Spec (\Z / p ^{i+1}\Z)$.
From \cite[IV.3.1.12]{Ogus-Logbook}, 
$ X  ^{\mathrm{sat}} \to X $
and 
$Y  ^{\mathrm{sat}} \to Y$
are log etale
(see \cite[II.2.4.5.2]{Ogus-Logbook} concerning the functor $X \mapsto X ^{\mathrm{sat}}$). 
Hence, using Remark \cite[IV.3.1.2]{Ogus-Logbook},
$f^{\mathrm{sat}} \colon X ^{\mathrm{sat}} \to Y^{\mathrm{sat}}$ is also log étale.
From Lemmas \ref{lem-lrp} and \ref{rem-lrplogetale},
we can suppose that $f = f^{\mathrm{sat}}$.
Since $f$ is fs formally log étale, then we are reduced to the case $i=0$. Next we observe that the case where $f$ is strict 
is already known (see \ref{strict-lrp} and \ref{et-flrp}).
Following 
\ref{rel-parf-stab2bispre}, and \ref{rel-parf-stab2bis},
the fs relative perfectness of $f$  is \'etale local on both $X$ and $Y$.
Hence, using \cite[II.2.2.12]{Ogus-Logbook}
and  \cite[Theorem 3.5]{Kato-logFontaine-Illusie},
we reduce to the case 
 where $X=A_P$, $Y=A_Q$ (see Notation \ref{ntnAp}) and $f$ is induced by a morphism $\phi \colon Q\to P$ of $fs$ monoids
 such that the kernel and cokernel of $\phi ^\mathrm{gp}$ is finite of order prime to $p$. 
It remains to prove that the left square below is cartesian in the category of fs log schemes:

\begin{equation}
\label{let-lrp-diag1}
\xymatrix{A_P  
\ar[d] ^-{f}
\ar[r]^{F_{A _P}}
&
A_P 
\ar[d]^-{f}
\\
A _Q
\ar[r]^{F_{A _Q}}
&
A_Q ,
}
\hspace{3cm} 
\xymatrix{
P
&
\ar[l]_{p} 
P
\\
Q
\ar[u] ^-{\phi}
&
\ar[l]_p Q.
\ar[u]  ^-{\phi}
}
\end{equation}
Since the functor $P \mapsto A _P $ from the category of fine satured monoids
to the category of fs log-schemes over $\Spec (\Z / p ^{i+1}\Z)$
transforms cocartesian squares into cartesian squares, 
it is thus sufficient to check that the right square of \ref{let-lrp-diag1} is cocartesian in the category of fine satured monoids.
Let us check that $P$ satisfies the universal property of the pushout.
Let $\alpha \colon P \to O$ and $\beta \colon Q \to O$ be two morphisms of
 fine satured monoids such that $\alpha \circ \phi = \beta \circ p$.
Since $O$ is saturated, if $z$ is an element  of $O ^{\mathrm{gp}}$ such that
$p z$ is in $O$ then $z $ is an element of $O$.
Hence, we reduce to check the universal property in the case where
$P = P ^{\mathrm{gp}}$,
$Q = Q ^{\mathrm{gp}}$ and $O = O ^{\mathrm{gp}}$.
Let $x _0\in P$.
There exist $y \in Q$ and $x\in P$ such than
$x _0 = \phi ( y) + p x$.
(Indeed,
let $n$ be an integer prime to $p$ such that  $n \cdot \mathrm{coker} (\phi)=0$.
There exists $y _0 \in Q$ such that $n x _0= \phi (y _0)$.
Using Bezout lemma, we get the desired result.)
We define a morphism $h \colon P \to O$ satisfying $\alpha = h \circ p$ and $\beta = h \circ \phi$ by putting $h  (x _0):= \beta ( y) + \alpha (x)$.
The morphism $h$ is well defined. Indeed, if $x _0 = \phi ( y) + p x = \phi ( y') + p x'$,
then $\phi ( y -y') = p( x'-x)$ and then we can suppose $y'=0$ and $x=0$, i.e. $x _0= \phi (y) = p x'$.
Since $p$ is an isomorphism on $\mathrm{coker} (\phi)$, there exists $z'\in Q$ such that $x' = \phi (z')$.
Since $p$ is an isomorphism on $\ker (\phi)$, there exists $z'' \in \ker (\phi)$ such that $ y  = p (z' +z'')$.
We compute $\beta ( y ) = \beta ( p (z' +z'') ) = \alpha \circ \phi (z' +z'')=  \alpha \circ \phi (z')= \alpha (x')$.
The unicity of such $h$ is also clear.

\end{proof}

\begin{rem}
\label{counterexample}
The ``fine'' version of proposition \ref{let-lrp} is wrong,
i.e. a log étale morphism is not necessarily fine log relatively perfect.
Indeed, we have the following counter-example.
Let $n\ge 2$ such that $p\nmid n$. 
In the category of fine (resp fine saturated) monoids the inductive limit of the diagram
$\xymatrix{\N&\ar[l] _p\ar[r]^n \N &\N}$
is the submonoid of $\N$ generated by $p$ and $n$ (resp. the associated saturated monoid, i.e. $\N$ itself).
Let $f:A_{\N}\to A_\N$ denote the morphism induced by $n:\N\to \N$.
Hence, the morphism $f$ is log étale but not fine relatively perfect.
\end{rem}

\begin{empt}
The following diagram summarizes the relations between our definitions:
\begin{equation}
\label{links}
\xymatrix{
{\text{log étale}}
\ar@/^1cm/@{=>}[rr] _-{\ref{let-lpet}}
\ar@{=>}[dr] _-{\ref{let-lrp}}
\ar@{}[r] 
|-{\ref{counterexample}}
&
{\text{fine log relatively perfect}}
\ar@{=>}[d] ^-{}
\ar@{=>}[r] ^-{\ref{rlp-fle-prop}}
&
{\text{log $p$-étale}}
\ar@{=>}[d] ^-{}
\ar@{=>}[r] ^-{}
&
{\text{fine formally log étale}}
\ar@{=>}[d] ^-{}
\\
&
{\text{fs log relatively perfect}}
\ar@{=>}[r] ^-{\ref{rlp-fle-prop}}
&
{\text{fs log $p$-étale}}
\ar@{=>}[r] ^-{}
&
{\text{fs formally log étale}.}
}
\end{equation}

From now on we will work with fine (not necessarily saturated) log schemes. 
For our purpose the most relevant notion among those of the above diagram turns out to be log $p$-etaleness
(indeed, the notion of fine formally log etaleness seems too general because,
to get a satisfying theory of $\D$-modules, we need the local description of 
\ref{mPDenv-local-desc}, whose proof uses \ref{mPDenv-lrp}).
The reader which is only interested in the category of fs log schemes can replace log $p$-etaleness by fs log $p$-etaleness in the sequel.
\end{empt}

\subsection{Log $p$-étaleness of level $m$, $m$-PD-envelopes, $n$th infinitesimal neighborhood of level $m$}
\label{subsection12}
Let $(I _S,J _S, \gamma)$ be a quasi-coherent $m$-PD-ideal of $\O _S$.
Let us fix some definitions.

\begin{dfn}
\label{dfnCm}
\begin{enumerate}
\item Let $\mathscr{C} ^{(m)} _{\gamma}$ (resp. $\mathscr{C} _{\gamma, n} ^{(m)}$) be the category whose objects are pairs
$(u, \delta)$ where $u$ is an exact closed $S$-immersion
$Z \hookrightarrow X$ of fine log-schemes
and $\delta$ is an $m$-PD-structure on the ideal $\I$ of $\O _X$ defining 
$u$
which is compatible (see definition \cite[1.3.2.(ii)]{Be1}) with $\gamma$ (resp. and such that $\I ^{\{n+1\} _{(m)}}=0$),
where $\I ^{\{n+1\} _{(m)}}$ is defined in the appendix of \cite{Be2}) ; whose morphisms
$(u', \delta ') \to (u, \delta)$ 
are commutative diagrams of the form
\begin{equation}
\label{morpC^{(m)}}
\xymatrix{
{X'}
\ar[r] ^-{f}
&
{X}
\\
{Z'}
\ar@{^{(}->}[u] ^-{u'}
\ar[r]
&
{Z}
\ar@{^{(}->}[u] ^-{u}}
\end{equation}
such that $f$ is an $m$-PD-morphism with respect to the $m$-PD-structures $\delta$ and $\delta'$
(i.e., denoting by $\I'$ the sheaf of ideals of $\O _{X'}$ defining 
$u'$,
for any affine open sets $U'$ of $X'$ and $U$ of $X$ such that 
$f(U') \subset U$, the morphism $f$ induces the $m$-PD-morphism 
$(\O _{X} (U) , \I (U), \delta) \to (\O _{X'} (U') , \I' (U'), \delta')$).
Beware that theses categories depends on $S$ and 
also on the quasi-coherent $m$-PD-ideal $(I _S,J _S, \gamma)$. 
The objects of $\mathscr{C} ^{(m)} _{\gamma}$ (resp. $\mathscr{C} _{\gamma, n} ^{(m)}$)
are called $m$-PD-$S$-immersions compatible with $\gamma$ 
(resp. $m$-PD-$S$-immersions of order $n$ compatible with $\gamma$). 
We remark that we have the inclusions
$\mathscr{C} ^{(m)} _{\gamma} 
\subset 
\mathscr{C} ^{(m')} _{\gamma} $
for any integer $m' \geq m$ (recall an $m$-PD-structure is also an $m'$-PD-structure).

We say that a morphism 
$(u', \delta ') \to (u, \delta)$ 
of $\mathscr{C} ^{(m)} _{\gamma}$ (resp. $\mathscr{C} _{\gamma, n} ^{(m)}$) is strict 
(resp. flat, resp. cartesian)
if $f$ is strict
(resp. $\underline{f}$ is flat, resp. 
the square \ref{morpC^{(m)}} is cartesian).

\item Let $u $ be an object of $\mathscr{C}$ (see the notation \ref{dfnC}).
 An ``$m$-PD-envelope compatible with $\gamma$ of $u$'' is
 an object $(u', \delta ')$ of $\mathscr{C} ^{(m)} _{\gamma}$
endowed with a morphism
$u' \to u$ in $\mathscr{C}$ satisfying the following universal property:
for any object $(u'', \delta '')$ of $\mathscr{C} ^{(m)} _{\gamma}$ endowed with a morphism
$f \colon u''\to u$ of $\mathscr{C} $
there exists a unique
morphism $(u'', \delta '') \to (u', \delta ')$
of $\mathscr{C} ^{(m)} _{\gamma}$
whose
composition 
with $u' \to u$ is $f$.
The unicity up to canonical isomorphism
of the $m$-PD-envelope compatible
with $\gamma$ of $u$ is obvious. We will denote by 
$P _{(m), \gamma} (u)$ the $m$-PD-envelope compatible
with $\gamma$ of $u$. By abuse of notation we also denote by 
$P _{(m), \gamma} (u)$ the underlying exact closed immersion or its target. 
The existence is checked below (see \ref{mPDenvelope}).

\item Since $p$ is nilpotent in $S$, we get   the forgetful functor
$For ^{(m)}\colon \mathscr{C} ^{(m)} _{\gamma} \to \mathscr{T}hick _{(p)}$
 (resp. $For _n ^{(m)}\colon \mathscr{C} ^{(m)} _{\gamma,n} \to \mathscr{T}hick _{(p)}$)
 given by $(u,\delta) \mapsto u$.
 We denote by 
$\mathscr{C'} ^{(m)} _{\gamma}$ (resp. $\mathscr{C'} ^{(m)} _{\gamma,n}$)
the image of $For ^{(m)}$ 
(resp. $For _n ^{(m)}$).
\end{enumerate}

\end{dfn}

\begin{ntn}
\label{gammaempty}
In this paragraph, suppose $J _S = p \O _S$.
Then, there is a unique PD-structure on $J _S$ which we will denote by $\gamma _{\emptyset}$. 
Let $u\colon Z \hookrightarrow X$ be an exact closed $S$-immersion of fine log-schemes
 and $\delta$ be an $m$-PD-structure on the ideal $\I$ of $\O _X$ defining $u$.
It follows from Lemma \cite[1.2.4]{Be1} that the $m$-PD-structure $\delta$ of $\I$ is always compatible with $\gamma _{\emptyset}$. 
Hence, in the description of $\mathscr{C} ^{(m)} _{\gamma _{\emptyset}}$ (resp. $\mathscr{C} _{\gamma _{\emptyset}, n} ^{(m)}$)
we can remove ``compatible with $\gamma _{\emptyset}$'' without changing the respective categories. 
For this reason,  
we put $\mathscr{C} ^{(m)} := \mathscr{C} ^{(m)} _{\gamma _{\emptyset}}$ (resp. $\mathscr{C} _{n} ^{(m)}
:= \mathscr{C} _{\gamma _{\emptyset}, n} ^{(m)}$).
But, recall  these categories depend on $S$ even if this is not written in the notation.
Finally, for any quasi-coherent $m$-PD-ideal $(I _S,J _S, \gamma)$ of $\O _S$, we have the inclusions 
\begin{equation}
\label{gammaempty-incl}
\mathscr{C} ^{(m)} _{\gamma}\subset \mathscr{C} ^{(m)},
\text{ and }
 \mathscr{C} _{\gamma, n} ^{(m)} \subset  \mathscr{C} _{n} ^{(m)}.
\end{equation}
\end{ntn}

\begin{dfn}
\label{dfn-petale(m)}
Let $f\colon X \to Y$ be an $S$-morphism of fine log schemes.
\begin{enumerate}

\item \label{1(m)}  We say that $f$ is ``formally log étale of level $m$ compatible with $\gamma$'' 
(resp. ``formally log unramified of level $m$ compatible with $\gamma$'') if it satisfies the following property:
for any commutative diagram of fine log schemes of the form
\begin{equation}
\label{dfn-petale-square(m)}
\xymatrix{
{U}
\ar[r] ^-{u _0}
\ar@{^{(}->}[d] ^-{\iota}
&
{X}
\ar[d] ^-{f}
\\
{T}
\ar[r] ^-{v}
&
{Y}
}
\end{equation}
such that $\iota$ is an object of $\mathscr{C'} _{\gamma, 1} ^{(m)}$,
there exists a unique morphism (resp. there exists at most one morphism)
$u \colon T \to X$ such that $u\circ \iota = u _0$ and $f \circ u = v$.

\item Replacing $\mathscr{C'} _{\gamma, 1} ^{(m)}$ 
by 
$\mathscr{C'} _{\gamma} ^{(m)}$, we get the notion of
``log $p$-étale of level $m$ compatible with $\gamma$'' 
(resp. ``log $p$-unramified of level $m$ compatible with $\gamma$'') morphism of fine log schemes.
To justify the terminology, the reader might see Proposition \ref{logp-etal=logpetaleanym}.

\item Replacing
``fine log $S$-schemes'' by ``$S$-schemes'' in the definition \ref{dfn-petale(m)}.\ref{1(m)},
we get the notion of ``formally étale of level $m$ compatible with $\gamma$'' 
(resp. ``formally unramified of level $m$ compatible with $\gamma$'')  morphism of schemes
and of 
``$p$-étale of level $m$ compatible with $\gamma$'' 
(resp. ``$p$-unramified of level $m$ compatible with $\gamma$'') morphism of schemes.

\item When $\gamma = \gamma _{\emptyset}$ (see Notation \ref{gammaempty}), 
we remove for simplicity ``compatible with $\gamma _{\emptyset}$'' in the terminology.
\end{enumerate}
\end{dfn}

\begin{lem}
\label{logpetalelevelm-stab}
The collection of formally log étale (resp. log $p$-étale) of level $m$ compatible with $\gamma$ 
morphisms of fine $S$-log schemes
is stable under base change and under composition.
Similarly replacing ``étale'' by ``unramified'' or/and removing ``log''.
\end{lem}

\begin{proof}
This is checked similarly to 
\ref{rel-parf-stab2pre} and \ref{p-etale-stab1}.
\end{proof}

\begin{rem}
\label{rem-etaleisnice}
Let $f\colon X \to Y$ be an $S$-morphism of fine log schemes. 
\begin{enumerate}
\item If the morphism $f$ is log $p$-étale then 
$f$ is  log $p$-étale of level $m$ compatible with $\gamma$ for any $m\in \N$ (recall
$\mathscr{C'} _{\gamma} ^{(m)} \subset \mathscr{T}hick _{(p)}$).
If $f$ is  log $p$-étale of level $m$ compatible with $\gamma$
then $f$ is formally log étale of level $m$ compatible with $\gamma$
(recall $\mathscr{C'} _{1,\gamma} ^{(m)} \subset \mathscr{C'} _{\gamma} ^{(m)}$).
Similarly replacing ``étale'' by ``unramified''.
Even if we do not have counterexamples, the converse seems false in general.

\item If $f$ is log étale, then $f$ is log $p$-étale (recall \ref{let-lpet})
and then $f$ is  log $p$-étale of level $m$ compatible with $\gamma$  for any $m\in \N$
and then $f$ is formally log étale of level $m$ compatible with $\gamma$  for any $m\in \N$.

\end{enumerate}

\end{rem}

\begin{lem}
\label{lem-CsThickp}
Suppose that $J _S + p \O _S$ is locally principal.
\begin{enumerate}
\item We have the inclusion 
$\mathscr{C} _{1}  \subset \mathscr{C'} _{\gamma, 1} ^{(m)} $.

\item For any $n,m \in \N$ such that $n+1 \leq p ^{m}$, we have the inclusion 
$\mathscr{C} _{n}  \subset \mathscr{C'} _{\gamma, n} ^{(m)} $ ; 
\item We have the equality 
$\cup _{m \in \N}   \mathscr{C'} _{\gamma } ^{(m)} = \mathscr{T}hick _{(p)}$.
\end{enumerate}

\end{lem}

\begin{proof}
Let us check the first two assertions. 
Let 
$u \colon U \hookrightarrow T$ 
a $S$-log thickening of order $n$, let $\I$ be the ideal defining the closed immersion $\underline{u}$. 
When $\I ^2 =0$, we get a PD-structure $\gamma$ on $\I$ defined
by putting $\gamma _n =0$ for any integer $n \geq 2$. 
Since $J _S + p \O _S$ is locally principal, then from \cite[1.3.2.(i).b)]{Be1}
$\gamma$ extends to $T$. 
Hence, $\mathscr{C} _{1}  \subset \mathscr{C'} _{\gamma, 1} ^{(0)} $, which yields the first inclusion to prove.
Suppose now $\I ^{n+1}= 0$ and $n+1 \leq p ^{m}$. In that case, 
$\I ^{(p ^{m})} =0$. Hence, $(0,\delta)$ is an $m$-PD-structure of $\I$ (where $\delta$ is the unique PD-structure on $0$). 
Let us check that 
the $m$-PD structure $(0,\delta)$ of $\I$ is compatible with $\gamma$. 
By definition, we have to check two properties (see \cite[1.3.2.(ii))]{Be1}).
Since 
$\gamma$ extends to $T$, 
then the property \cite[1.3.2.1]{Be1} is satisfied (see Definition \cite[1.2.2]{Be1}). 
The second one \cite[1.3.2.2]{Be1} is a straightforward consequence of Lemma  \cite[1.2.4.(i)]{Be1}.
Hence, $(u,\delta) \in \mathscr{C} _{\gamma} ^{(m)}$.
Since $\I ^{n+1}=0$, we have in fact $(u,\delta) \in \mathscr{C} _{\gamma, n} ^{(m)}$. 
By definition, this yields $u \in \mathscr{C'} _{\gamma, n} ^{(m)}$.

Let us check the last statement. 
The inclusion 
$\cup _{m \in \N}   \mathscr{C'} _{\gamma } ^{(m)} \subset \mathscr{T}hick _{(p)}$
is tautologic. Conversely, 
let 
$u \colon U \hookrightarrow T$ 
a $S$-log thickening of order $(p ^{m})$, let $\I$ be the ideal defining the closed immersion $\underline{u}$. 
Since $\I ^{(p ^{m})} =0$, then following the first part of the proof, 
we get that the $m$-PD structure $(0,\delta)$ is compatible with $\gamma$ of $\I$.
Hence,  $u \in \mathscr{C'} _{\gamma} ^{(m)}$, which concludes the proof of the last statement.
\end{proof}

\begin{prop}
\label{logp-etal=logpetaleanym}
Suppose that $J _S + p \O _S$ is locally principal.
Let $f\colon X \to Y$ be an $S$-morphism of fine log schemes.

\begin{enumerate}
\item The morphism $f$ is log $p$-étale if and only if for any $m\in \N$
$f$ is  log $p$-étale of level $m$ compatible with $\gamma$.
Similarly replacing ``étale'' by ``unramified''.

\item If $f$ is formally log étale of level $m$ compatible with $\gamma$
then $f$ is fine formally log étale. Similarly replacing ``étale'' by ``unramified''.

\end{enumerate}

\end{prop}

\begin{proof}
This is a consequence of  \ref{lem-CsThickp}.
\end{proof}

\begin{lem}
\label{p-etale-modplevelm}
Suppose that $J _S + p \O _S$ is locally principal.
Let $f\colon X \to Y$ be an $S$-morphism of fine log-schemes
and
 $f _0\colon X _{0} \to Y _0$ be the induced $S _0$-morphism.

The morphism $f $ is 
formally log étale of level $m$ compatible with $\gamma$ 
(log $p$-étale of level $m$ compatible with $\gamma$)
if and only if $f$ is fine formally log etale and $f _0$ is
formally log étale of level $m$ compatible with $\gamma$ 
(log $p$-étale of level $m$ compatible with $\gamma$).
Similarly replacing everywhere "étale" by "unramified".

\end{lem}

\begin{proof}
Using 
$\mathscr{C} _{1}  \subset \mathscr{C'} _{\gamma, 1} ^{(m)} $
(see \ref{lem-CsThickp}), 
we proceed similarly to \ref{p-etale-modp}.
\end{proof}

\begin{empt}
\label{mPDenv-Berthelot}
Forgetting log structures,
i.e. replacing in \ref{dfnC} and \ref{dfnCm} fine log-schemes by schemes,
we define
similarly the categories
$\mathscr{\underline{C}} $,
$\mathscr{\underline{C}} ^{(m)} _{\gamma}$ and
$\mathscr{\underline{C}} ^{(m)}_{\gamma,n} $.
Following \cite[1.4.1 and 2.1.1]{Be1},
the forgetful functor
$\underline{For} ^{(m)}\colon \mathscr{\underline{C}} ^{(m)} _{\gamma} \to \mathscr{\underline{C}}$ defined by
$(\underline{u} , \delta) \mapsto \underline{u}$ has a right adjoint
that we will denote by $\underline{P} _{(m), \gamma}$.
If $\underline{u}$ is an object of
$\mathscr{\underline{C}}$
then
$\underline{P} _{(m), \gamma} (\underline{u})$
is called the
``$m$-PD-envelope compatible with $\gamma$ of $\underline{u}$''.
Moreover, since $p$ is nilpotent then
the morphism of schemes induced by
the targets of
$\underline{P} _{(m), \gamma} (\underline{u}) \to \underline{u}$
is affine (see \cite[2.1.1]{Be1}).

\end{empt}

\begin{empt}
\label{env-strict}
Let $u \colon Z \hookrightarrow X$ be an exact $S$-immersion of fine log-schemes.
Set $(\underline{v}, \delta):= \underline{P} _{(m), \gamma} (\underline{u})$ (see \ref{mPDenv-Berthelot}).
Let $(v, \delta )$ be the object of $\mathscr{C} ^{(m)} _{\gamma}$ whose underlying object of
$\mathscr{\underline{C}} ^{(m)} _{\gamma} $ is $(\underline{v}, \delta )$ and $v$ is defined so that
the morphism $v \to u$ of  
$\mathscr{C}$ 
is strict (see the definition \ref{dfnC}). Then
$(v, \delta )$  is the $m$-PD-envelope compatible with $\gamma$ of $u$.
\end{empt}

\begin{rem}
\label{rem-flat cartesian}
Let $\alpha \colon (u', \delta ') \to (u, \delta)$ be 
a strict cartesian morphism of
$\mathscr{C} ^{(m)} _{\gamma} $.
Let $(u'', \delta '') $ be an object of $\mathscr{C} ^{(m)} _{\gamma} $
and $\beta \colon u'' \to u'$ be a morphism of $\mathscr{C}$.
We remark that if $For ^{(m)}(\alpha )\circ \beta$ is in the image of
$For ^{(m)}$ then so is
$\beta$.
Indeed, the morphism $\alpha$ is defined by a cartesian diagram of the form \ref{morpC^{(m)}}.
Since $\alpha$ is moreover strict, then we remark that 
$\underline{Z}' = \underline{Z} \times _{\underline{X}} \underline{X}'$, i.e. the diagram
\ref{morpC^{(m)}} remains cartesian after applying the forgetful functor from the category of fine log schemes to the category of schemes.
Hence, we can conclude.
\end{rem}

\begin{empt}
\label{basecgt-flat-env}
Let $u' \to u$ be a strict, flat, cartesian morphism of $\mathscr{C}$, i.e. let
\begin{equation}
\notag
\xymatrix{
{X'}
\ar[r] ^-{g}
\ar@{}[dr]|{\square}
&
{X}
\\
{Z'}
\ar@{^{(}->}[u] ^-{u'}
\ar[r] ^-{}
&
{Z}
\ar@{^{(}->}[u] ^-{u}
}
\end{equation}
be a cartesian square
whose morphism $g$ is strict and $\underline{g}$ is flat.
Suppose
that the $m$-PD-envelope compatible with $\gamma$ of $u$ exists
(in fact, this existence will be proved later in \ref{mPDenvelope}).
Let $(v ,\delta)$ be this $m$-PD-envelope.
Set $v' : =v \times _{u} u'$ and let $g' \colon v' \to v$ be the projection. 
Since $\underline{g}$ is flat and $g$ is strict,
then $g'$ is strict and $\underline{g}'$ is flat.  
From \cite[1.3.2.(i)]{Be1}, there exists a canonical $m$-PD-structure $\delta '$ compatible with $\gamma$ 
on the ideal defining $v' : =v \times _{u} u'$ such that 
the projection 
$g '\colon v'
\to
v$
induces a strict cartesian morphism of  $\mathscr{C} ^{(m)} _{\gamma}$ of the form
$(v ' , \delta ' )
\to
(v  , \delta )$.
With the remark \ref{rem-flat cartesian}, we check that
$(v ',\delta')$
is an $m$-PD-envelope compatible with $\gamma$ of $u'$.
\end{empt}

\begin{lem}
\label{mPDenv-lrp}
Let $f \colon X \to Y$ be a log $p$-étale $S$-morphism,
$u \colon Z \hookrightarrow X$ and $v \colon Z \hookrightarrow Y$ be two $S$-immersions of fine log schemes
such that $v= f\circ u$.
If the $m$-PD envelope compatible with $\gamma$ of $u$ exists then it is also an $m$-PD envelope of $v$.
\end{lem}

\begin{proof}
Let $(P(u),\delta)$ be the $m$-PD envelope compatible with $\gamma$ of $u$.
Let us check that the composition of the canonical  morphism $P(u) \to u$  with
the morphism $u \to v$ (induced by $f$)
satisfies the universal property of
the $m$-PD envelope compatible with $\gamma$ of $v$.
Let $(v', \delta ')$ be an object of $\mathscr{C} ^{(m)} _{\gamma}$
and $g \colon v' \to v$ be a morphism of $\mathscr{C} $.
Using the universal property of log $p$-étaleness of \ref{dfn-petale}
we get a unique morphism
$h\colon v' \to u$ of $\mathscr{C} $ whose composition with $u \to v$ gives $g$.
Using the universal property of the $m$-PD-envelope of $u$ compatible with $\gamma$
that there exists a unique morphism 
$(v', \delta ')\to (P(u),\delta)$ of $\mathscr{C} ^{(m)} _{\gamma}$
such that the composition of 
$v'\to P(u)$  with $P(u) \to u$ is $h$. 

\end{proof}

\begin{lem}
\label{rightadjCmCmn}
The inclusion functor  $For  _n \colon \mathscr{C} _{\gamma,n} ^{(m)} \to \mathscr{C} ^{(m)} _{\gamma}$ has a right adjoint.
We denote by $Q ^n _{(m), \gamma}
\colon \mathscr{C} _{\gamma} ^{(m)} \to \mathscr{C} ^{(m)} _{\gamma,n}$
this right adjoint functor.
The functor $Q ^n _{(m), \gamma}$ preserves the sources.
\end{lem}

\begin{proof}
Let $(u,\delta)$ be an object of $\mathscr{C} _{\gamma} ^{(m)} $ 
and $\I$ be the ideal defining the exact closed immersion $u \colon Z \hookrightarrow X$.
Let $Q ^{n}\hookrightarrow X$
be the exact closed immersion which is defined by $\I ^{\{n+1  \} _{(m)}}$.
It follows from \cite[1.3.8.(iii)]{Be1} that $Q ^n _{(m), \gamma} (u) $ exists and 
is equal to the exact closed immersion 
$Z \hookrightarrow Q ^{n}$.
\end{proof}

\begin{prop}
\label{mPDenvelope}
Let  $u\colon Z \hookrightarrow X$ be an object of
$\mathscr{C}$.

\begin{enumerate}
\item The $m$-PD-envelope compatible with $\gamma$ of $u$ exists.
In other words,
the canonical functor  $For ^{(m)}\colon \mathscr{C} ^{(m)} _{\gamma} \to \mathscr{C}$ has a right adjoint.
We denote by $P _{(m), \gamma}
\colon \mathscr{C} \to \mathscr{C} ^{(m)} _{\gamma} $
this right adjoint functor.
Similarly, we get
the right adjoint functor
$P ^{n }_{(m), \gamma}
\colon \mathscr{C} \to \mathscr{C} ^{(m)} _{\gamma,n} $
of the canonical functor $For  ^{(m)}_n\colon \mathscr{C} ^{(m)} _{\gamma,n} \to \mathscr{C}$.
We have the relation $P ^n _{(m), \gamma} = Q ^n _{(m), \gamma}  \circ P _{(m), \gamma} $.

\item If $\gamma$ extends to $Z$ then the source of
$P _{(m), \gamma} (u)$ is $Z$.

\item By denoting abusively by $P _{(m), \gamma} (u)$ (resp. $P ^{n} _{(m), \gamma} (u)$) the target of the arrow
$P _{(m), \gamma} (u)$ (resp. $P ^{n} _{(m), \gamma} (u)$),
the underlying morphism of schemes of
$P _{(m), \gamma} (u) \to X$ (resp. $P ^{n} _{(m), \gamma} (u) \to X$) is affine.
We denote by $\PP _{(m), \gamma} (u)$ (resp. $\PP ^{n} _{(m), \gamma} (u)$) the quasi-coherent $\O _X$-algebra so that
$\underline{P _{(m), \gamma} (u) } = \mathrm{Spec} (\PP _{(m), \gamma} (u))$
(resp. $\underline{P ^{n} _{(m), \gamma} (u)} = \mathrm{Spec} (\PP ^{n} _{(m), \gamma} (u))$).
The $m$-PD structure of 
$\PP _{(m), \gamma} (u)$
will be denoted by 
$(\mathscr{I} _{(m), \gamma} (u), \mathscr{J} _{(m), \gamma} (u), ^{[\ ]})$

\item Suppose that $J _S + p \O _S$ is locally principal and that
$X$ is noetherian (i.e. $\underline{X}$ is noetherian).
Then 
$\underline{P ^{n} _{(m), \gamma} (u)}$ a noetherian scheme.
\end{enumerate}

\end{prop}

\begin{proof}
1) First, let us 
check the proposition concerning the existence of $P _{(m), \gamma} (u)$ and its properties (i.e. the second part of the proposition and 
the affinity of the morphism $P _{(m), \gamma} (u) \to X$).
Using \ref{basecgt-flat-env}, the existence of $P _{(m), \gamma} (u)$ 
and its properties
are étale local on $X$.
Hence, by \ref{ex-cl-imm}, we may thus assume that
there exists a commutative diagram of the form
\begin{equation}
\notag
\xymatrix{
{\widetilde{X}}
\ar[r] ^-{f}
&
{X}
\\
&
{Z}
\ar@{_{(}->}[u] ^-{u}
\ar@{^{(}->}[lu] ^-{\widetilde{u}}
}
\end{equation}
such that $f$ is log étale, $\underline{f}$ is affine and $\widetilde{u}$ is an exact closed $S$-immersion.
In that case, following \ref{env-strict}
the $m$-PD-envelope compatible with $\gamma$ of $\widetilde{u}$ exists
and the induced object of $\mathscr{\underline{C}} ^{(m)} _{\gamma}$
is $\underline{P} _{(m), \gamma} (\widetilde{\underline{u}})$.
Following \ref{mPDenv-lrp},
the $m$-PD-envelope compatible with $\gamma$ of $u$ exists
and is isomorphic to that of $\widetilde{u}$.
Concerning the second statement,
when $\gamma$ extends to $Z$, following \cite[2.1.1]{Be1} (or \cite[1.4.5]{Be1} for the affine version), the source of the immersion
$\underline{P} _{(m), \gamma} (\widetilde{\underline{u}})$ is $\underline{Z}$.
Since $P _{(m), \gamma} (\widetilde{u})$, $\widetilde{u}$ are exact closed immersion, since
the morphism $P _{(m), \gamma} (\widetilde{u}) \to \widetilde{u}$ is strict (see \ref{env-strict}),
then so is the morphism of sources induced by 
$P _{(m), \gamma} (\widetilde{u}) \to \widetilde{u}$.
Hence, we get the second statement.
We check the third statement recalling
that
the target of $\underline{P} _{(m), \gamma} (\widetilde{\underline{u}})$ is affine over $\underline{\widetilde{X}}$ 
(see \ref{mPDenv-Berthelot})
and that $\underline{P} _{(m), \gamma} (\widetilde{\underline{u}}) \to  \widetilde{u}$ is strict.
Concerning the noetherianity, if $X$ is noetherian then so is $\widetilde{X}$.
Hence, using \cite[1.4.4]{Be1} and the description of the $m$-PD filtration given in the proof of \cite[A.2]{Be1},
we get that
$\underline{P} ^n _{(m), \gamma} (\widetilde{\underline{u}})$ is noetherian
(but not $\underline{P}  _{(m), \gamma} (\widetilde{\underline{u}})$).

2) From Lemma \ref{rightadjCmCmn}, 
we check that the functor $Q ^n _{(m), \gamma}   \circ P _{(m), \gamma} $ is a right adjoint of 
$For  ^{(m)}_n\colon \mathscr{C} ^{(m)} _{\gamma,n} \to \mathscr{C}$.
Moreover, with the description of the functor $Q ^n _{(m), \gamma} $ given in the proof of \ref{rightadjCmCmn}, 
we check the other properties concerning $P ^n _{(m), \gamma} $ from that of $P _{(m), \gamma} $. 

\end{proof}

\begin{dfn}
Let  $u$ be an object of
$\mathscr{C}$.
We say that 
$P ^{n }_{(m), \gamma} (u)$
is the ``$n$th infinitesimal neighborhood of level $m$ compatible with 
$\gamma$ of $u$''.
\end{dfn}

\begin{rem}
\label{remPmdelta}
Let $(u, \delta)$ be an object of  $\mathscr{C} ^{(m)} _{\gamma}$.
Then $P _{(m), \delta} (u) =(u, \delta)$.
But, 
beware that $P _{(m), \gamma} (u ) \not =
(u, \delta)$ in general.

\end{rem}

\begin{empt}
[The case of an exact closed immersion]
\label{levelm-level0}
Let $u\colon Z \hookrightarrow X$ be an exact closed $S$-immersion of fine log-schemes
and $\I $ be the ideal defining $u$.
We denote by $u ^{(m)}\colon Z ^{(m)}\hookrightarrow X$ the exact closed $S$-immersion of fine log-schemes
so that $\I^{(p ^{m})}$ is the ideal defining $u ^{(m)}$. 
Since the closed immersion $u$ is exact, 
in the proof of \ref{mPDenvelope}, we can skip the part concerning the exactification of $u$ (i.e. we can suppose $f = id$ or equivalently $\widetilde{u}=u$). 
Hence, we remark that,
as in the proof of \cite[1.4.1]{Be1},
we get the equality
\begin{gather}
\label{levelm-level0-description}
P _{(m), \gamma} (u)
=
P _{(0), \gamma} (u ^{(m)}).
\end{gather}
We have also the same construction as in the the proof of \cite[1.4.1]{Be1} 
(too technical to be described here in few words) 
of the $m$-PD ideal 
$(\mathscr{I} _{(m), \gamma} (u), \mathscr{J} _{(m), \gamma} (u), ^{[\ ]})$
of $\PP _{(m), \gamma} (u)$ 
directly from the level $0$ case. 
For the detailed descriptions, see the proof of \cite[1.4.1]{Be1}.  
These descriptions, in particular \ref{levelm-level0-description}, are useful 
to check the Frobenius descent for arithmetic $\D$-modules (see \cite[2.3.6]{Be2}). 
\end{empt}

\begin{lem}
\label{Be1-1432}
We have the equality
$P ^{n }_{(m), \gamma} \circ For  _n \circ P ^n 
=P ^{n }_{(m), \gamma}$, where 
 $For  _n \colon \mathscr{C} _n \to \mathscr{C}$
 is the canonical functor and $P ^n \colon \mathscr{C}  \to \mathscr{C} _n$ is its right adjoint (see \ref{PDenvelope}).
\end{lem}

\begin{proof}
Let  $u\colon Z \hookrightarrow X$ be an object of
$\mathscr{C}$.
Looking at the construction of $P ^n$ and 
$P ^{n }_{(m), \gamma}$ (see the proof of 
\ref{PDenvelope} and \ref{mPDenvelope}),
we reduce to the case where $u$ is an exact closed immersion.
In that case, the Lemma is a reformulation of \cite[1.4.3.2]{Be1}.
\end{proof}

The following proposition will not be useful later but it allows us to extend 
\ref{formally log unramified} is some particular case.
\begin{prop}
\label{formally log unramified-bis}
Suppose that $J _S + p \O _S$ is locally principal.
Let $f\colon X \to Y$ be an $S$-morphism of fine log schemes and $\Delta _{X/Y}\colon X \hookrightarrow X \times _{Y}X$ (as always the product
is taken in the category of fine log schemes) be the
diagonal $S$-immersion.
The following assertions are equivalent : 
\begin{enumerate}
\item the morphism $f$ is fine formally log unramified ; 
\item the morphism $P ^{1} (\Delta _{X/Y})$ is an isomorphism ;
\item the morphism $f$ is formally log unramified of level $m$ compatible with $\gamma$ ;
\item the morphism 
$P ^{1} _{(m), \gamma} (\Delta _{X/Y})$ is an isomorphism.
\end{enumerate}
\end{prop}

\begin{proof}
The equivalence between 1) and 2) has already been checked (see \ref{formally log unramified}). 
Following \ref{lem-CsThickp}, 
since $J _S + p \O _S$ is locally principal, then 
$\mathscr{C} _{1}  \subset \mathscr{C'} _{\gamma, 1} ^{(m)} $.
Hence, we have $3) \Rightarrow 1)$.
It follows from \ref{Be1-1432} that 
$P ^{1}_{(m), \gamma} ( P ^{1} (\Delta _{X/Y}))
=P ^{1}_{(m), \gamma}(\Delta _{X/Y})$.
If $P ^{1} (\Delta _{X/Y})$ is an isomorphism, then 
$P ^{1}_{(m), \gamma} ( P ^{1} (\Delta _{X/Y}))
=
P ^{1} (\Delta _{X/Y})$. 
Hence, we get the implication 
$2) \Rightarrow 4)$.
It remains to  check 
$4) \Rightarrow 3)$.
Suppose $P ^{1} _{(m), \gamma} (\Delta _{X/Y})$ is an isomorphism and 
let $(\iota,\delta) \in \mathscr{C} _{\gamma, 1} ^{(m)}$ and 
let
\begin{equation}
\notag
\xymatrix{
{U}
\ar[r] ^-{u _0}
\ar@{^{(}->}[d] ^-{\iota}
&
{X}
\ar[d] ^-{f}
\\
{T}
\ar[r] ^-{v}
&
{Y}
}
\end{equation}
be a commutative diagram of fine log schemes. 
Suppose there exist a morphism
$u \colon T \to X$
 such that $u\circ \iota = u _0$ and $f \circ u = v$,
and a morphism
$u '\colon T \to X$
such that $u'\circ \iota = u _0$ and $f \circ u' = v$. 
We get the morphism 
$(u,u') \colon T \to X \times _Y X$. 
We denote by 
and $\phi \colon \iota \to \Delta _{X/Y}$ be a morphism of 
$\mathscr{C}$ induced by $(u',u)$ and $u _0$.
Using the universal property of the $m$-PD envelope of order $1$, there exists a unique morphism
$\psi\colon (\iota,\delta) \to P ^{1} _{(m), \gamma} (\Delta _{X/Y})$
of  $\mathscr{C} _{\gamma, 1} ^{(m)}$ 
such that the composition of 
$For  ^{(m)}_n (\psi) $ with the canonical map
$P ^{1} _{(m), \gamma} (\Delta _{X/Y}) \to \Delta _{X/Y}$
is $\phi$.
Since $P ^{1} _{(m), \gamma} (\Delta _{X/Y})$ is an isomorphism, 
this yields that 
$(u,u') \colon T \to X \times _Y X$ is the composition of a morphism
of the form $T \to X$ with $\Delta _{X/Y}$.
Hence, $u=u'$.
\end{proof}

\begin{lem}
\label{red-Y2D}
Let $u \to v$ be a morphism of $\mathscr{C}$.
Let $\delta$ be the $m$-PD-structure of
 $P _{(m), \gamma} (v) $ and
 $w:= P _{(m), \gamma} (v) \times _{v} u$ (this is the product in $\mathscr{C}$).
 We denote by 
$P _{(m), \delta} (w)$ the $m$-PD-envelope of $w$ compatible with $\delta$. 
Then $P _{(m), \delta} (w)$
 and $P _{(m), \gamma} (u) $
are isomorphic in  $\mathscr{C} ^{(m)} _{\gamma} $.
Moreover, $P ^n _{(m), \delta} (w)$
 and $P ^n _{(m), \gamma} (u) $
are isomorphic in  $\mathscr{C} _{\gamma, n} ^{(m)} $.
\end{lem}

\begin{proof}
Since the second assertion is checked in the same way, let us prove the first one.
We check that the composition 
$P _{(m), \delta} (w)
\to
w
\to
u$
satisfies the universal property of
$P _{(m), \gamma} (u)\to u$.
Let $(u',\delta ') \in \mathscr{C} ^{(m)} _{\gamma} $
and $f \colon u' \to u$ be a morphism of $\mathscr{C}$.
First let us check the existence. 
Composing $f$ with $u \to v$ we get a morphism 
$g \colon u' \to v$. Using the universal property of the $m$-PD envelope, there exists 
a morphism 
$\phi \colon (u',\delta ') 
\to 
(P _{(m), \gamma} (v) ,\delta)$
of
$ \mathscr{C} ^{(m)} _{\gamma}$ such that the composition 
$u ' \to P _{(m), \gamma} (v) \to v$ is $g$.
Hence, we get the morphism $(\phi, f)\colon u' \to w$.
Using the universal property of $P _{(m), \delta} (w)$,
we get a morphism 
$u' \to P _{(m), \delta} (w)$ 
of $\mathscr{C} ^{(m)} _\delta$ (and then of $\mathscr{C} ^{(m)} _{\gamma} $)
whose composition with 
$P _{(m), \delta} (w)\to w \to u$ is $f$.
Let us check the unicity. 
Let $\alpha \colon u' \to P _{(m), \delta} (w)$ be a morphism of $\mathscr{C} ^{(m)} _{\gamma} $
whose composition with 
$P _{(m), \delta} (w)\to w \to u$ is $f$.
This implies that the composition of $\alpha$ with 
$P _{(m), \delta} (w)\to w \to P _{(m), \gamma} (v) \to v$ is $g$.
Since the composition $P _{(m), \delta} (w)\to w \to P _{(m), \gamma} (v) $
is a morphism of $\mathscr{C} ^{(m)} _\delta$, then so is 
the composition of $\alpha $ with $P _{(m), \delta} (w)\to w \to P _{(m), \gamma} (v) $ (in particular, 
this implies that $\alpha \in \mathscr{C} ^{(m)} _\delta$).
Using the universal property of $P _{(m), \gamma} (v) $, this latter composition morphism is uniquely determined by $g$.
Hence, the composition of $\alpha$ with $P _{(m), \delta} (w)\to w$ is a morphism of 
$\mathscr{C} $ uniquely determined by $f$. 
Since $\alpha $ is a morphism of $\mathscr{C} ^{(m)} _\delta$,
we conclude using the universal property of $P _{(m), \delta} (w)$.
\end{proof}

\subsection{Around log $p$-bases, log $p$-smoothness, local descriptions}
Let $(I _S,J _S, \gamma)$ be a quasi-coherent $m$-PD-ideal of $\O _S$.

\begin{dfn}
\label{forlogbasis}
Let $f\colon X \to Y$ be an morphism of fine $S$-log schemes.

\begin{enumerate}
\item We say that a set $(b  _{\lambda}) _{\lambda \in \Lambda}$ of elements of $\Gamma ( X, M _X)$ is a `` formal log basis of $f$''
if the corresponding $Y$-morphism
$X \to Y \times _{\Z/p ^{i+1}\Z} A _{\N ^{(\Lambda)}}$ is fine formally log étale.
When $\Lambda$ is a finite set, we say 
that $(b  _{\lambda}) _{\lambda \in \Lambda}$
is a ``finite formal log basis of $f$''.

\item We say that $f$ is ``weakly log smooth'' if, étale locally on $X$,
$f$ has finite formal log bases.
Notice that this notion of weak log smoothness is étale local on $Y$.

\end{enumerate}

\end{dfn}

\begin{dfn}
\label{logpetalelogpsmooth}
Let $f\colon X \to Y$ be an morphism of fine $S$-log schemes.

\begin{enumerate}
\item We say that a set $(b  _{\lambda}) _{\lambda \in \Lambda}$ of elements of $\Gamma ( X, M _X)$ is a ``log $p$-basis of $f$''
if the corresponding $Y$-morphism
$X \to Y \times _{\Z/p ^{i+1}\Z} A _{\N ^{(\Lambda)}}$ is log $p$-étale.
When $\Lambda$ is a finite set, we say that they 
$(b  _{\lambda}) _{\lambda \in \Lambda}$ is a ``finite log $p$-basis of $f$''

\item We say that $f$ ``has log $p$-bases locally'' if, étale locally on $X$,
$f$ has log $p$-bases.
Notice that this notion of ``having log $p$-bases locally'' is étale local on $Y$.

\item We say that $f$ is ``log $p$-smooth'' if, étale locally on $X$,
$f$ has finite log $p$-bases.
Notice that this notion of log $p$-smoothness is étale local on $Y$.

\end{enumerate}

\end{dfn}

\begin{rem}
\label{Tsuji-havingpbases}
Tsuji defined in \cite[1.4.1)]{Tsuji-nearby-log-smooth} 
(resp. \cite[1.4.2)]{Tsuji-nearby-log-smooth}) the notion of ``a morphism of log-schemes having $p$-basis''
(resp. ``a morphism of log-schemes having $p$-bases locally'').

1) Tsuji's notion of having $p$-basis is different from our notion of having log $p$-basis of \ref{logpetalelogpsmooth}
and can not be clearily compared.

2) However, we emphasize that our notion of having log $p$-bases locally (see 
\ref{logpetalelogpsmooth}) is more general that 
Tsuji's notion of having $p$-bases locally.
Indeed, this fact is a consequence of Lemmas  \ref{rlp-fle-prop}, \ref{strict-lrp} and  \cite[Lemma 1.5]{Tsuji-nearby-log-smooth}.

\end{rem}

\begin{dfn}
\label{logpetalelogpsmooth(m)}
Let $f\colon X \to Y$ be an morphism of fine $S$-log schemes.
\begin{enumerate}
\item We say that a finite set $(b  _{\lambda}) _{\lambda =1,\dots, r}$ of elements of $\Gamma ( X, M _X)$ is a 
``formal log basis of level $m$ (compatible with $\gamma$) of $f$''
if the corresponding $Y$-morphism
$X \to Y \times _{\Z/p ^{i+1}\Z} A _{\N ^{r}}$ is formally log étale of level $m$ (compatible with $\gamma$).

\item 
We say that $f$ is ``weakly log smooth of level $m$ (compatible with $\gamma$)'' if, étale locally on $X$,
$f$ has formal log bases of level $m$ (compatible with $\gamma$).
Notice that this notion is étale local on $Y$
(use the last remark of \ref{rem-etaleisnice}).

\item We say that a finite set $(b  _{\lambda}) _{\lambda =1,\dots, r}$ of elements of $\Gamma ( X, M _X)$ is a 
``log $p$-basis of level $m$ (compatible with $\gamma$) of $f$''
if the corresponding $Y$-morphism
$X \to Y \times _{\Z/p ^{i+1}\Z} A _{\N ^{r}}$ is log $p$-étale of level $m$ (compatible with $\gamma$).

\item We say that $f$ is ``log $p$-smooth of level $m$ (compatible with $\gamma$)'' if, étale locally on $X$,
$f$ has finite log $p$-bases of level $m$ (compatible with $\gamma$).
Notice that this notion of log $p$-smoothness of level $m$ (compatible with $\gamma$) is étale local on $Y$ 
(use the last remark of \ref{rem-etaleisnice}).

\end{enumerate}

\end{dfn}

The following Proposition indicates the link between 
our notions related to log smoothness.
\begin{prop}
\label{smooth=>psmooth}
Let $f\colon X \to Y$ be an $S$-morphism of fine log-schemes.
\begin{enumerate}
\item If $f$ is log smooth then $f$ is log $p$-smooth.
\item If $f$ is log $p$-smooth then $f$ is log $p$-smooth of level $m$ compatible with $\gamma$. 
\item If $f$ is log $p$-smooth of level $m$ compatible with $\gamma$ 
then $f$ is weakly log smooth of level $m$ compatible with $\gamma$. 
\item If $f$ is log $p$-smooth (resp. weakly log smooth) of level $m+1$ compatible with $\gamma$ 
then $f$ is log $p$-smooth (resp. weakly log smooth) of level $m$ compatible with $\gamma$. 
\item If $f$ is weakly log smooth of level $m$ then $f$ is weakly log smooth. 
\end{enumerate}

\end{prop}

\begin{proof}
The first statement is a straightforward consequence of Theorem  \cite[IV.2.3]{Ogus-Logbook}
and Proposition \ref{let-lpet}. 
Using Proposition \ref{logp-etal=logpetaleanym}.2, we get the last assertion. 
The other implications are consequences of Remark \ref{rem-etaleisnice}.
\end{proof}

\begin{prop}
\label{logpsmooth-stab}
The collection of formally log étale of level $m$ compatible with $\gamma$ 
(resp. log $p$-étale of level $m$ compatible with $\gamma$,
resp. log $p$-étale, resp. log smooth, resp. log $p$-smooth, resp. having log $p$-bases locally,
resp. weakly log smooth,
resp. weakly log smooth of level $m$ compatible with $\gamma$,
resp. log $p$-smooth of level $m$ compatible with $\gamma$) 
morphisms of fine $S$-log schemes
is stable under base change and under composition.
\end{prop}

\begin{proof}
The etale cases are already known (see \ref{rel-parf-stab2pre}, \ref{p-etale-stab1}, and \ref{logpetalelevelm-stab}).
Let us check the stability of the collection of morphisms having log $p$-bases locally. 
Using the étale case, the stability under base change is obvious. 
Moreover, from the non respective case, if $X \to Y \times _{\Z/p ^{i+1}\Z} A _{\N ^{(\Lambda)}}$ 
and 
$Y \to Z \times _{\Z/p ^{i+1}\Z} A _{\N ^{(\Lambda')}}$ are log $p$-étale,
then so is the composition
$X \to Z \times _{\Z/p ^{i+1}\Z} A _{\N ^{(\Lambda \coprod \Lambda')}}$.
This implies the stability under composition.
The other cases are checked similarly. 
\end{proof}

The following definition will not be useful but we add it for completeness. 
\begin{dfn}
Let $f\colon X \to Y$ be an $S$-morphism of fine log schemes.
We say that $f$ is formally log $p$-smooth if it satisfies the following property: 
for any commutative diagram of fine log schemes of the form
\begin{equation}
\label{dfn-petale-squarebis}
\xymatrix{
{U}
\ar[r] ^-{u _0}
\ar@{^{(}->}[d] ^-{\iota}
&
{X}
\ar[d] ^-{f}
\\
{T}
\ar[r] ^-{v}
&
{Y}
}
\end{equation}
such that $\iota$ is an object of $\mathscr{C} _{(p)}$,
there exists etale locally on $X$ a  morphism
$u \colon T \to X$ such that $u\circ \iota = u _0$ and $f \circ u = v$.
\end{dfn}

\begin{rem}
Let $f\colon X \to Y$ be an $S$-morphism of fine log schemes.
Using Theorem \cite[IV.18.1.2]{EGAIV4}, we check that if $f$ has log $p$-bases locally 
then $f$ is formally log $p$-smooth. 
(The converse seems to be false even if we do not have a counter example.) 
In particular, if $f$ is log $p$-smooth
then $f$ is formally log $p$-smooth. 
 
Finally, we remark  that $f$ is log $p$-étale if and only if $f$ is  
log $p$-unramified and formally log $p$-smooth. 

\end{rem}

\begin{ntn}
\label{ntn-mPDtruncated}
With the notation of proposition \ref{mPDenv-local-desc},
Let $D$ be a fine log scheme over $S$.
We denote by $ \O _{D} <T _1 ,\dots ,T _r> _{(m), n} $ the $m$-PD-polynomial ring
and by 
$ \O _{D} <T _1 ,\dots ,T _r> _{(m), n} $ the $n$-truncated $m$-PD-polynomial ring.
\end{ntn}

\begin{lem}
\label{Kerexactclosedimmer}
Let $i \colon X \hookrightarrow P$ be an exact closed $S$-immersion of fine log schemes.
Then 
$
\ker 
(
i ^{-1} \O ^* _P 
\to 
\O ^* _X
)
=
\ker 
(
i ^{-1} M ^{\mathrm{gr}}  _P
\to 
M ^{\mathrm{gr}}  _X
)$. In particular, 
$
\ker 
(
\O ^* _{P, {\overline{x}}}  
\to 
\O ^* _{X, {\overline{x}}}
)
=
\ker 
(
M ^{\mathrm{gr}}  _{P, \overline{x}}
\to 
M ^{\mathrm{gr}}  _{X, \overline{x}}
)$
for any geometric point $\overline{x}$  of $X$. 
\end{lem}

\begin{proof}
Let $\overline{x}$ be a geometric point of $X$. 
Since $i$ is an exact closed immersion,
we have $(M_P / \O _P ^{*}) _{\overline{x}}
=
(M_X  / \O _X ^{*}) _{\overline{x}}$ (use  \cite[1.4.1]{Kato-logFontaine-Illusie})
and thus
$(M_P ^\gp / \O _P ^{*}) _{\overline{x}}
=
(M_X ^\gp / \O _X ^{*}) _{\overline{x}}$ (because the functor 
$M \mapsto M ^{\mathrm{gr}}$ commutes with inductive limits).
Hence, the inclusion 
$
\ker 
(
\O ^* _{P, \overline{x}}  
\to 
\O ^* _{X, \overline{x}}
)
\subset
\ker 
(
M ^{\mathrm{gr}}  _{P, \overline{x}}
\to 
M ^{\mathrm{gr}}  _{X, \overline{x}}
)$ is in fact an equality.
Since we have the canonical inclusion
$
\ker 
(
i ^{-1} \O ^* _P 
\to 
\O ^* _X
)
\subset
\ker 
(
i ^{-1} M ^{\mathrm{gr}}  _P
\to 
M ^{\mathrm{gr}}  _X
)$, this yields that this latter inclusion is an equality.

\end{proof}

\begin{prop}
\label{mPDenv-local-desc}
Let $f \colon X \to Y$ be an $S$-morphism of fine log-schemes,
$(b  _{\lambda}) _{\lambda =1,\dots, r}$ be some elements of $\Gamma ( X, M _X)$
Let $u \colon Z \hookrightarrow X$ and $v \colon Z \hookrightarrow Y$ be two $S$-immersions of fine log schemes
such that $v= f\circ u$.
Suppose given $y _{\lambda} \in \Gamma (Y, M _{Y})$ whose images in $\Gamma (Z, M _{Z})$ coincide with the images of $b _\lambda$.

Let 
$
T' \hookrightarrow D'$
and 
$
T \hookrightarrow D$
be the $m$-PD envelopes compatible with $\gamma$ of $u$ and $v$ respectively.
Let $\alpha \colon D '\to X$ be the canonical morphism.
Using multiplicative notation, put 
$u _{\lambda} := \frac{\alpha ^* (b _{\lambda} )}{\alpha ^* ( f ^{*} (y _\lambda)) }\in 
\ker ( \Gamma (D', M ^\mathrm{gr} _{D '})\to  \Gamma (T' ,M  _{T'} ^\mathrm{gr} ))
=
\ker ( \Gamma (D', \O _{D '} ^{*} )\to  \Gamma (T' ,\O _{T'} ^{*} ))
$ 
(see \ref{Kerexactclosedimmer}). 

Let 
$
T' \hookrightarrow D' _n$
and 
$
T \hookrightarrow D _n$
be the $m$-PD envelopes of order $n$ compatible with $\gamma$ of $u$ and $v$ respectively.
Let $\alpha _n \colon D _n '\to X$ be the canonical morphism.
Using multiplicative notation, put 
$u _{\lambda,n} := \frac{\alpha _n ^* (b _{\lambda} )}{\alpha _n ^* ( f ^{*} (y _\lambda)) }\in 
\ker ( \Gamma (D', M ^\mathrm{gr} _{D _n '})\to  \Gamma (T' ,M  _{T'} ^\mathrm{gr} ))
=
\ker ( \Gamma (D', \O _{D _n '} ^{*} )\to  \Gamma (T' ,\O _{T'} ^{*} ))
$.

\begin{enumerate}
\item If $(b  _{\lambda}) _{\lambda =1,\dots, r}$ is a formal log basis 
 of level $m$ compatible with $\gamma$ of $f$,
then, we have the isomorphism of $m$-PD-$\O _{D}$-algebras
   \begin{align}
   \notag
   \O _{D _n} <T _1 ,\dots ,T _r> _{(m),n}
   &\riso \O _{D' _n}\\
   \notag
   T _\lambda &\mapsto u _{\lambda ,n} -1.
   \end{align}

\item If $(b  _{\lambda}) _{\lambda =1,\dots, r}$ is a log $p$-basis 
 of level $m$ compatible with $\gamma$ of $f$,
then, we have the isomorphism of $m$-PD-$\O _{D}$-algebras
   \begin{align}
   \notag
   \O _{D} <T _1 ,\dots ,T _r> _{(m)}
   &\riso \O _{D'}\\
   \notag
   T _\lambda &\mapsto u _{\lambda } -1.
   \end{align}

\end{enumerate}

\end{prop}

\begin{proof}
Using lemma \ref{mPDenv-lrp} (and the first remark of \ref{rem-immersion}), we may assume that
$X = Y \times _{\Z /p ^{i+1}\Z} A _{\N ^{r}}$,
$f \colon X \to Y$ is the first projection,  and that the family $(b  _{\lambda}) _{\lambda =1,\dots, r}$ are the elements of
$\Gamma(X,M _X)$ corresponding to the canonical basis $(e _\lambda)_{\lambda=1\dots r}$ of $\N ^{r}$.
Using lemma \ref{red-Y2D},
we may furthermore assume that $Y=S$,  
$Z\hookrightarrow Y$ is the exact closed immersion whose ideal of definition is $I _S$. 
In particular, we get $D=Y$ and $\gamma$ is the canonical $m$-PD structure of $D$.

Let $\overline{z}$ be a geometric point of $Z$.
From \ref{ex-cl-imm}, there exists a commutative diagram of the form
\begin{equation}
\label{mPDenv-local-desc-diag1}
\xymatrix{
{U}
\ar[r] ^-{g}
&
{X=Y\times _{\Z /p ^{i+1}\Z} A_{\mathbb{N}^r}}\ar[r]^-{f}&Y
\\
{W}
\ar@{_{(}->}[u] ^-{w}
\ar[r] ^-{h}
&
{Z}
\ar@{^{(}->}[u] ^-{u}
\ar@{^{(}->}[ur] _-{v}
} 
\end{equation}
where $g$ is log étale,
$h$ is an étale neighborhood of $\overline{z}$ in $Z$, and
$w$ is an exact closed $S$-immersion.
We set $v_\lambda:=\frac{g ^* (b_\lambda)}{(f \circ g) ^*(y_\lambda)}
\in \mathrm{Ker} (\Gamma(U,M_U^{\gp})\to \Gamma(W,M_W^{\gp}))$.
Since $w$ is an exact closed immersion,
using \ref{Kerexactclosedimmer}, shrinking $U$ if necessary we may thus assume
that $v_\lambda \in  \Gamma(U,\O _U ^{*})$.

1) In this step, we reduce to the case where 
$h=id$.
According to \cite[IV.18.1.1]{EGAIV4}, 
there exist an étale neighborhood $Y ' \to Y$ of
$\overline{z}$ in $Y$ and an open $Z$-immersion (see the definition \ref{dfn-imm}) $\rho \colon Z ' := Z \times _{Y} Y'\to W$
which is a morphism of étale neighborhoods of $\overline{z}$ in $Z$ (in particular $h \circ \rho \colon Z \times _Y Y' \to Z$ is the canonical projection).
Let us use the prime symbol to denote the base change by $Y'\to Y$ of a $Y$-log scheme or a morphism of $Y$-log schemes.
Set $j:= (\rho, v' ) \colon Z ' \to W  \times _{Y} Y' =W'$.
Since $h \circ \rho \colon Z' \to Z$ is the canonical projection, 
then we compute that $h' \circ j = id$.
Since $id$ is an immersion, then $j$ is an immersion (see the first remark of \ref{rem-immersion}).
Since $h'$ and $id$ are etale then so is $j$. Hence, $j$ is an open $Y$-immersion.
Using $h' \circ j = id$, we get the commutative diagram over $Y'$ 
\begin{equation}
\notag
\xymatrix{
{U'}
\ar[r] ^-{g'}
&
{X'= Y '\times _{\Z /p ^{i+1}\Z} A _{\N ^{r}}}
\\
{}
&
{Z'.}
\ar@{^{(}->}[ul] ^-{w' \circ j}
\ar@{^{(}->}[u] ^-{u'}}
\end{equation}
Using \ref{rem-immersion}.\ref{rem-immersion2} we may assume (shrinking $U$ if necessary) that the exact $Y$-immersion $w' \circ j$ is closed.
Since the proposition is étale local  on $Y$, we can drop the primes, i.e. we can suppose $h=id$.

2) 
Consider the $Y$-morphism $\phi \colon U \to Y \times _{\Z /p ^{i+1}\Z} A _{\N ^{r}}$ defined by
the $v _{\lambda}$'s.
Since
the $(dlog\, g ^* (b_1),\dots, dlog\, g ^* (b_r))$ forms a basis of $\Omega_{U/Y}$ (because $g$ is log étale),
then so does
$(dlog\, v_1,\dots, dlog\, v_r)$.
This implies that the canonical map
$\phi ^{*} \Omega _{X/Y} \to \Omega _{U/Y}$
induced by $\phi$
is an isomorphism. Since $U/Y$ is log smooth
we get that
$\phi$ is log étale (use \cite[3.12]{Kato-logFontaine-Illusie}).

Let $\iota \colon  Y \to Y \times _{\Z /p ^{i+1}\Z} A _{\N ^{r}}$
be the $Y$-morphism defined by $e _\lambda \mapsto 1 \in  \Gamma (Y, M _Y)$, 
and 
$i\colon Y \to Y \times _{\Z /p ^{i+1}\Z} A _{\Z ^{r}}$ be the exact closed $Y$-immersion defined by $e _\lambda \mapsto 1 \in \Gamma (Y, M _Y)$.
We compute that the diagram of morphisms of log schemes
\begin{equation}
\label{mPDenv-local-desc-diag2}
\xymatrix{
{U}
\ar[r] ^-{\phi}
&
{Y \times _{\Z /p ^{i+1}\Z} A _{\N ^{r}}}
&
{Y \times _{\Z /p ^{i+1}\Z} A _{\Z ^{r}}}
\ar[l] ^-{}
\ar[r] ^-{p _1}
&
{Y}
\\
{}
&
{Z,}
\ar@{^{(}->}[ul] ^-{w}
\ar@{^{(}->}[u] ^-{\iota \circ v} 
\ar@{^{(}->}[ur] ^-{i \circ v}
\ar@{^{(}->}[urr] _-{v}}
\end{equation}
where $p _1$ is the first projection, is commutative. 
Since $g$, $\phi$ and $Y \times _{\Z /p ^{i+1}\Z} A _{\Z ^{r}} \to Y \times _{\Z /p ^{i+1}\Z} A _{\N ^{r}}$
are log etale, then using lemma \ref{mPDenv-lrp}
we reduce to the case where $u = i\circ v$, $X= Y \times _{\Z /p ^{i+1}\Z} A _{\Z ^{r}}$, 
 $(b  _{\lambda}) _{\lambda =1,\dots, r}$ are the elements of
$\Gamma(X,M _X)$ corresponding to the canonical basis of $\Z ^{r}$,
and $(y  _{\lambda}) _{\lambda =1,\dots, r}$ are equal to $1$
(indeed, for such $b _\lambda$ and $y _\lambda$, using the commutativity of \ref{mPDenv-local-desc-diag2} and the lemma \ref{mPDenv-lrp}, we compute that the image of 
$\frac{b _{\lambda} }{b ^{*} (y _\lambda) }$ in $\Gamma (D', M ^\mathrm{gr} _{D '})$
is $u _\lambda$).

By using the remark \cite[1.4.3.(iii)]{Be1} and \ref{env-strict}, 
we can suppose that $v = id$.
Since the ideal of the exact closed immersion $i$ is generated by the regular sequence $(b _{\lambda}-1) _{\lambda=1,\dots r}$,
using \cite[1.5.3]{Be1} and \ref{env-strict}
we check that the morphism of $\O _Y$-algebras
$ \O _{Y} <T _1 ,\dots ,T _r> _{(m)} \to 
\O _{P _{(m), \gamma} (i)}$
given by $T _{\lambda}  \mapsto  b _{\lambda} -1$
is an isomorphism.

\end{proof}

\begin{rem}
With the notation \ref{mPDenv-local-desc}, when $v$ is a closed immersion then such $y_{\lambda}$'s always exist \'etale locally on $Y$.
\end{rem}

\begin{prop}
\label{nthneighb-local-desc}
Let $f \colon X \to Y$ be an $S$-morphism of fine log-schemes endowed with a formally log basis
$(b  _{\lambda}) _{\lambda =1,\dots, r}$.
Let $u \colon Z \hookrightarrow X$ and $v \colon Z \hookrightarrow Y$ be two $S$-immersions of fine log schemes
such that $v= f\circ u$.
Suppose given $y _{\lambda} \in \Gamma (Y, M _{Y})$ whose images in $\Gamma (Z, M _{Z})$ coincide with the images of $b _\lambda$.
Let 
$
T' \hookrightarrow D '$
and 
$
T \hookrightarrow D $
be the $n$th infinitesimal neighborhood of order $n$ of  $u$ and $v$ respectively.
Let $\alpha \colon D  '\to X$ be the canonical morphism.
Using multiplicative notation, put 
$u _{\lambda} := \frac{\alpha ^* (b _{\lambda} )}{\alpha ^* ( f ^{*} (y _\lambda)) }\in 
\ker ( \Gamma (D ', M ^\mathrm{gr} _{D  '})\to  \Gamma (T' ,M  _{T'} ^\mathrm{gr} ))
=
\ker ( \Gamma (D ', \O _{D  '} ^{*} )\to  \Gamma (T' ,\O _{T'} ^{*} ))
$ 
(see \ref{Kerexactclosedimmer}). 
Then, we have the isomorphism of $\O _{D }$-algebras
   \begin{align}
   \notag
   \O _{D } [T _1 ,\dots ,T _r] _n
   &\riso \O _{D '}\\
   T _\lambda &\mapsto u _{\lambda } -1,
   \end{align}
where    $\O _{D } [T _1 ,\dots ,T _r] _n :=    \O _{D } [T _1 ,\dots ,T _r] / (T _1, \dots, T _r ) ^{n+1}$.
\end{prop}

\begin{proof}
By copying the proof of \ref{mPDenv-local-desc}, 
we reduce to the case where $v = id$, 
$X= Y \times _{\Z /p ^{i+1}\Z} A _{\Z ^{r}}$, 
$u\colon Y \to Y \times _{\Z /p ^{i+1}\Z} A _{\Z ^{r}}$ is the exact closed $Y$-immersion defined by 
$e _\lambda \mapsto 1 \in \Gamma (Y, M _Y)$,
 $(b  _{\lambda}) _{\lambda =1,\dots, r}$ are the elements of
$\Gamma(X,M _X)$ corresponding to the canonical basis of $\Z ^{r}$,
and $(y  _{\lambda}) _{\lambda =1,\dots, r}$ are equal to $1$.
This case is obvious.

\end{proof}

\subsection{The case of schemes: local coordinates}

\begin{dfn}
\label{finite-p-basis}
Let $f\colon X \to Y$ be an $S$-morphism of fine log-schemes.
\begin{enumerate}
\item We say that a finite set $(t  _{\lambda}) _{\lambda =1,\dots, r}$ of elements of $\Gamma ( X, \O _X)$ are 
``log $p$-étale coordinates''
(resp. ``formal log étale coordinates'',
resp. ``formal log étale coordinates of level $m$'',
resp. ``log $p$-étale coordinates of level $m$''),
if the corresponding $Y$-morphism
$X \to Y \times _{\Z} \A  ^{r}$, where $\A ^r$ is the $r$th affine space over $\Z$ endowed 
with the trivial logarithmic structure,
is log $p$-étale
(resp. formally log étale, resp. formally log étale of level $m$, 
log $p$-étale of level $m$).

When $f$ is strict
(this is equivalent to say that the $Y$-morphism
$X \to Y \times _{\Z} \A  ^{r}$ is strict)
we remove ``log'' in the terminology, e.g. we get the notion of ``$p$-étale coordinates''.
\item We say that $f$ is ``$p$-smooth'' 
(resp. ``weakly smooth'',
resp. ``weakly smooth of level $m$'',
resp. ``$p$-smooth of level $m$''),
if $f$ is strict and if, étale locally on $X$,
$\underline{f}$ has 
$p$-étale coordinates''
(resp. ``formal étale coordinates'',
resp. ``formal étale coordinates of level $m$'',
resp. ``$p$-étale coordinates of level $m$'').
Notice that these notions 
are étale local on $Y$.
\end{enumerate}
\end{dfn}

\begin{lem}
\label{rem-lognologbasis}
Let $\star \in \{ 
\text{log $p$-étale, 
formal log étale, 
formal log étale of level $m$,
log $p$-étale of level $m$}\}$.
Let $f\colon X \to Y$ be an $S$-morphism of fine log-schemes.
\begin{enumerate}
\item 
\label{rem-lognologbasis1} 
Let $(t  _{\lambda}) _{\lambda =1,\dots, r}$ in $\Gamma ( X, \O ^* _X)$.
Then $(t  _{\lambda}) _{\lambda =1,\dots, r}$ form a $\star$-basis if and only if 
they are $\star$-coordinates.

\item 
If the morphism $f$ has 
$\star$-coordinates 
with $r$ elements then, Zariski locally on $X$,
there exist $r$ elements 
of $\Gamma ( X, \O ^* _X)$
which are 
$\star$-coordinates of $f$.

\end{enumerate}
\end{lem}

\begin{proof}
Since the other cases are treated similarly, let us consider the case where $\star = \text{log $p$-étale}$.
 Let us check the first assertion. 
Using \ref{p-etale-stab1} we check that 
in both cases the elements $(t  _{\lambda}) _{\lambda =1,\dots, r}$ induce
a  log $p$-étale morphism of the form
$X \to Y \times _{\Z} \mathbb{G} _\mathrm{m} ^{r}$.
We conclude by using \ref{p-etale-stab1}.

Let us check the second assertion.
 By hypothesis, we have a log $p$-étale morphism of the form
$\phi \colon X \to Y \times _{\Z} \A  ^{r}$.
Let  $(t  _{1},\dots, t _r)$ be the elements of $ \Gamma ( X, \O _X)$ defining $\phi$.
We have a covering of $X$ by open subsets of the form 
$D (u _1 \cdots u _r)$ where $u _\lambda$ is either $t _\lambda$ or $t _\lambda -1$.
Using a obvious change of coordinates, we check that the morphism
$D (u _1 \cdots u _r) \to Y \times _{\Z} \A  ^{r}$ given by the restriction of $(u  _{1},\dots, u _r)$ on 
$D (u _1 \cdots u _r)$ is log $p$-étale.

\end{proof}

\begin{lem}
\label{rem-lognologbasis(part2)}
Let $\star \in \{ 
\text{$p$-smooth, 
weakly smooth, 
weakly smooth of level $m$,
$p$-smooth of level $m$}\}$.
Let $f\colon X \to Y$ be an $S$-morphism of fine log-schemes.
\begin{enumerate}

\item If $f$ is $\star$ then $f$ is log $\star$. 

\item Suppose that $f\colon X \to Y$ is in fact an $S$-morphism of schemes.
The morphism $f$ is log $\star$ if and only if $f$ is $\star$.
\end{enumerate}
\end{lem}

\begin{proof}
Since the other cases are treated similarly, let us consider the case where $\star = \text{$p$-smooth}$.
Let us check the first assertion. Suppose that $f$ is $p$-smooth. 
Then $\underline{f}$ is $p$-smooth by definition. Hence, from 1) and 2), we get that 
$\underline{f}$ is log $p$-smooth. Since $f$ is strict, then $f$ is the base change of $\underline{f}$
by $Y \to \underline{Y}$. Hence, using \ref{logpsmooth-stab}, we get that $f$ is log $p$-smooth.

Let us check the last assertion. Suppose that 
$f$ is an $S$-morphism of schemes which is log $p$-smooth.
By definition, etale locally on $X$,  
there exists a log $p$-étale of the form $X \to Y \times _{\Z/p ^{i+1}\Z} A _{\N ^{r}}$ 
given by some elements $(b  _{1}, \dots, b _r) $ of $\Gamma ( X, M _X)$.
Since $X$ is a scheme, $M _X = \O _X ^*$. Hence, using 1), we get that 
$(b  _{1}, \dots, b _r) $ form also a finite $p$-basis.
The converse is already known from 3).

\end{proof}

\begin{rem}
It is not clear that a strict log $p$-smooth morphism is $p$-smooth.
\end{rem}

\begin{prop}
\label{psmooth-stab}
The collection of $p$-smooth
(resp. weakly smooth,
resp. weakly smooth of level $m$,
resp. $p$-smooth of level $m$)
is stable under base change and under composition.
\end{prop}

\begin{proof}
This is similar to \ref{logpsmooth-stab}.
\end{proof}

\begin{prop}
\label{nolong-mPDenv-local-desc}
Let $f \colon X \to Y$ be a strict $S$-morphism of fine log-schemes,
$(t  _{\lambda}) _{\lambda =1,\dots, r}$ be some elements of 
$\Gamma (X, \O _X)$.
Let $u \colon Z \hookrightarrow X$ and $v \colon Z \hookrightarrow Y$ be two $S$-immersions of fine log schemes
such that $v= f\circ u$.
Suppose that there exist $y _{\lambda} \in \Gamma (Y, \O _{Y})$ whose images in $\Gamma (Z, \O _{Z})$ coincide with the images of $t _\lambda$.

\begin{enumerate}
\item If $(t  _{\lambda}) _{\lambda =1,\dots, r}$ 
are formal log étale coordinates of level $m$,
then, we have the following isomorphism of $m$-PD-$\O _{P ^{n }_{(m), \gamma} (v)}$-algebras
   \begin{align}
   \notag
   \O _{P ^{n }_{(m), \gamma} (v)} <T _1 ,\dots ,T _r> _{(m),n}
   &\riso \O _{P ^{n }_{(m), \gamma} (u)}\\
   T _\lambda &\mapsto t _{\lambda } - f ^{*} (y _\lambda),
   \end{align}
where by abuse of notation we denote by 
$t _{\lambda } $ and 
$f ^{*} (y _\lambda)$ the canonical image in $\O _{P ^{n }_{(m), \gamma} (u)}$.

\item If $(t  _{\lambda}) _{\lambda =1,\dots, r}$ 
are log $p$-étale coordinates of level $m$,
then, we have the following isomorphism of $m$-PD-$\O _{P _{(m), \gamma} (v)}$-algebras
   \begin{align}
   \notag
   \O _{P _{(m), \gamma} (v)} <T _1 ,\dots ,T _r> _{(m)}
   &\riso \O _{P _{(m), \gamma} (u)}\\
   T _\lambda &\mapsto t _{\lambda } - f ^{*} (y _\lambda).
   \end{align}

\item  If $(t  _{\lambda}) _{\lambda =1,\dots, r}$ 
are formal log étale coordinates,
then, we have the following isomorphism of $\O _{P ^{n } (v)}$-algebras
   \begin{align}
   \notag
   \O _{P ^{n } (v)} [T _1 ,\dots ,T _r] _{n}
   &\riso \O _{P ^{n } (u)}\\
   T _\lambda &\mapsto t _{\lambda } - f ^{*} (y _\lambda).
   \end{align}
\end{enumerate}

\end{prop}

\begin{proof}
Since the other assertions are treated similarly, 
let us check the first assertion. 
This is similar to \ref{mPDenv-local-desc} and even easier:
using lemma \ref{mPDenv-lrp} (and the first remark of \ref{rem-immersion}), we may assume that
$X = Y \times _{\Z} \A ^{r}$,
$f \colon X \to Y$ is the first projection,  
 and that the family $(t   _{\lambda}) _{\lambda =1,\dots, r}$ are the elements of
$\Gamma(X,\O _X)$ corresponding to the coordinates of $\A ^{r}$.
Using lemma \ref{red-Y2D},
we may furthermore assume that $Y=S$,  
$Z\hookrightarrow Y$ is the exact closed immersion whose ideal of definition is $I _S$. 
Let $\phi \colon 
 Y \times _{\Z} \A ^{r} 
 \to 
  Y \times _{\Z} \A ^{r}$ be the $Y$-morphism given by 
$  t _{1} - f ^{*} (y _1), \dots, t _{r} - f ^{*} (y _r)$.
Let $i\colon Y \hookrightarrow Y \times _{\Z} \A ^{r}$ be the exact closed immersion defined by 
$t _\lambda \mapsto 0$.
Since $\phi$ is etale, since $\phi \circ u = i$, using lemma \ref{mPDenv-lrp}, 
we reduce to the case 
where $u$ is equal to $i$, and $(y  _{\lambda}) _{\lambda =1,\dots, r}$ are equal to $0$.
By using the remark \cite[1.4.3.(iii)]{Be1} and \ref{env-strict}, 
we can suppose that $v = id$.
Since the ideal of the exact closed immersion $i$ is generated by the regular sequence $(t _{\lambda}) _{\lambda=1,\dots r}$,
using \cite[1.5.3]{Be1} and \ref{env-strict}
we check that the morphism
$ \O _{Y} <T _1 ,\dots ,T _r> _{(m)} \to \O _{P _{(m), \gamma} (\iota)}= \O _{P _{(m), \gamma} (i)}$
given by $T _{\lambda}  \mapsto  t _{\lambda} $
is an isomorphism.
\end{proof}

\section{Differential operators of level $m$ over log $p$-smooth log schemes}
Let $i$ be an integer and
 $S$ be a fine log scheme over the scheme $\Spec (\Z / p ^{i+1}\Z)$.

\subsection{Sheaf of principal parts, sheaf of differential operators}
Let $f\colon X \to S$ be a weakly log smooth 
morphism of fine log-schemes.

\begin{ntn}
\label{notaDeltanetc}
Let $\Delta _{X/S} \colon X \to X \times _{S} X$ be the diagonal morphism,
$\Delta ^{n} _{X/S}:= P ^{n } ( \Delta _{X/S}) $,
$\PP ^{n} _{X/S} := \PP ^{n }  (\Delta _{X/S})$ (see Notation \ref{PDenvelope}).
We denote by $M ^{n} _{X/S}$ the log structure of
$\Delta ^{n} _{X/S}$.
We denote abusively the target of $\Delta ^{n} _{X/S}$ by
$\Delta ^{n} _{X/S}$.

We denote by  respectively
$p ^{n} _1,\, p ^{n} _0 \colon \Delta  ^{n } _{X/S, (m)} \to X$
 the composition of the canonical morphism
$\Delta ^n  _{X/S} \to X \times _{S} X$  with the right and left projection
$X \times _{S} X \to X$.

If $a \in M _X$, we denote by
$\mu  (a)$ the unique section of
$\ker ( \O _{\Delta ^n _{X/S}} ^{*} \to  \O _{X} ^{*} )$
such that we get in $M ^n _{X/S}$ the equality
$p _1 ^{n*} (a)= p _0 ^{n*} (a) \mu ^n  (a)$ (see \ref{Kerexactclosedimmer}).
We get 
$\mu  ^{n} \colon M  _X \to \ker ( \O _{\Delta ^{n}  _{X/S}} ^{*} \to  \O _{X} ^{*} )$
given by 
$a \mapsto \mu ^n  (a)$.
\end{ntn}

\begin{lem}
\label{pistrict}
The morphisms $p ^n _1$ and $p ^n _0$ are strict.
\end{lem}

\begin{proof}
This is similar to  \cite[2.2.1]{these_montagnon}:
let $\iota ^n  \colon X \hookrightarrow \Delta ^n  _{X/S}$ be the structural morphism. 
Since $\iota ^{-1}=id$, then 
from \cite[1.4.1]{Kato-logFontaine-Illusie} we get
the isomorphisms
$p _i ^{n*} (M _X) / \PP ^{n *} _{X/S}
\riso 
M / \O _X ^*$
and 
$M ^{n} _{X/S} / \PP ^{n *} _{X/S}
\riso 
\iota ^{n*} (M ^{n} _{X/S}) / \O _X ^*
\riso 
M _X/ \O _X ^*$ (the last isomorphism is a consequence of the exactness of 
$\iota ^n$). 
Hence, 
$p _i ^{n*} (M _X) / \PP ^{n *} _{X/S}
\riso 
M ^{n} _{X/S} / \PP ^{n *} _{X/S}$.
This implies that the canonical morphism
$p _i ^{n*} (M _X)
\to 
M ^{n} _{X/S} $
is an isomorphism.

\end{proof}

\begin{prop}
[Local description of $\PP ^n _{X/S} $]
\label{nota-etawtm}
Let $(a  _{\lambda}) _{\lambda =1,\dots, r}$ be a formal  log basis  of $f$. 
Put 
$\eta _{\lambda,n} := \mu ^n  ( a _{\lambda}) -1$.
We have  the following  isomorphism of $\O _{X}$-algebras: 
   \begin{align}
   \notag
   \O _{X} [T _1 ,\dots ,T _r] _n
   &\riso \PP  ^n  _{X/S}  \\
   \label{loc-desc-Pnf}
   T _\lambda &\mapsto
   \eta _{\lambda,n},
   \end{align}
where the structure of $\O _X$-module of 
$\PP  ^n  _{X/S}$
is given by $p ^{n} _1$ or $p ^{n} _0$.
\end{prop}

\begin{proof}
Since the case of $p ^{n} _1$ is checked symmetrically, let us compute the case where
the  $\O _X$-module of 
$\PP  ^n  _{X/S}$ is given by $p ^{n} _0$.
Consider the commutative diagram
\begin{equation}
\label{p0-p1-logbasis}
\xymatrix{
{A _{\N ^r}} 
\ar@{}[dr]|{\square}
& 
{X \times _S A _{\N ^r}} 
\ar[l] ^-{p _1}
\ar[rd] ^-{p _0}
 \\ 
 {X} 
 \ar[u] ^-{a}
&
{X \times _S X} 
\ar[l] ^-{p _1}
\ar[r] ^-{p _0}
 \ar[u] ^-{b}
  & 
 {X,} 
 }
\end{equation}
where $p _0$, $p_1$ means respectively the left and right projection, 
where $a $ is the $S$-morphism induced by $a _1,\dots, a _r$,
where $b$ is the $X$-morphism induced by $b _1,\dots, b _r$ with
$b _\lambda:= p ^* _1 ( a _\lambda)$.
Since $(b _\lambda ) _{\lambda =1,\dots, r}$ is 
a formal log basis  
of $p _0$ (because the square of the diagram \ref{p0-p1-logbasis} is cartesian),
we can apply Proposition 
\ref{nthneighb-local-desc}
in the case where $f$ is $p _0$, $u$ is $\Delta _{X/S}$,
$b _\lambda$ is as above, 
and 
$u _\lambda $ is $a _\lambda$.
\end{proof}

\begin{rem}
\label{p1p0finite}
From the local description of \ref{nota-etawtm}, 
we get that 
the morphisms $p ^n _1$ and $p ^n _0$ are finite (i.e. the underlying morphism of schemes is finite).
\end{rem}

\begin{ntn}
\label{Omega1pre}
We denote by 
$\I ^1 _{X/S}$ 
the ideal of the closed immersion $\Delta  ^1 _{X/S}$ 
and by 
$\Omega ^{1} _{X/S}:= (\Delta  ^1 _{X/S }) ^{-1} (\I ^1 _{X/S})$ the corresponding $\O _X$-module
(recall $\Delta  ^1 _{X/S *}$ is an homeomorphism).
To justify the notation, we refer to the isomorphism \cite[5.8.1]{Kato-logFontaine-Illusie}. 
\end{ntn}

\begin{rem}
\label{rem-omega-locfree}
Following the local description \ref{loc-desc-Pnf},
since $f$ has formal log bases locally, then 
$\Omega ^{1} _{X/S}$ is a locally free $\O _X$-module of finite rank
and the rank is equal to the cardinal of the formal log basis (a basis is given by
$\eta _1, \dots, \eta _r$). 
\end{rem}

\begin{empt}
The exact closed immersions
$\Delta ^{n} _{X/S} $ and $ \Delta ^{n '} _{X/S}$
induce
$\Delta ^{n,n'} _{X/S} := (\Delta ^{n} _{X/S}, \Delta ^{n'} _{X/S}) 
\colon 
X \hookrightarrow \Delta ^{n} _{X/S} \times _{X} \Delta ^{n '} _{X/S}$.
Since 
the morphisms $p ^n _1$ and $p ^n _0$ are strict 
(see \ref{pistrict}),
then $\Delta ^{n,n'} _{X/S} $ is also an exact closed immersion. 
We get $\Delta ^{n,n'} _{X/S} \in \mathscr{C} _{n+n'} $.
Using the universal property of the $n+n'$ infinitesimal neighborhood of 
$ \Delta _{X/S}$,
we get a unique
morphism
$\Delta ^{n} _{X/S} \times _{X} \Delta ^{n '} _{X/S}
\to \Delta ^{n+n'} _{X/S}$ of $\mathscr{C} _{n+n'} $
inducing the commutative diagram
\begin{equation}
\label{diag-p02pre}
\xymatrix{
{X}
\ar@{^{(}->}[r] ^-{}
\ar@{=}[d] ^-{}
&
{\Delta ^{n} _{X/S} \times _{X} \Delta ^{n '} _{X/S}}
\ar[d] ^-{}
\ar[r] ^-{}
&
{X \times _S X \times _S X}
\ar[d] ^-{p _{02}}
\\
{X}
\ar@{^{(}->}[r] ^-{}
&
{\Delta ^{n+n'} _{X/S}}
\ar[r] ^-{}
&
{X \times _S X . }
}
\end{equation}
We denote by
$\delta ^{n,n'} 
\colon
\PP ^{n +n'} _{X/S}
\to
\PP ^{n} _{X/S} \otimes _{\O _X} \PP ^{n'} _{X/S}$
the corresponding morphism.

By replacing $p _{02}$ by $p _{01}$ (resp. $ p _{12}$),
we get a unique
morphism
$\Delta ^{n} _{X/S} \times _{X} \Delta ^{n '} _{X/S}
\to \Delta ^{n+n'} _{X/S}$
making commutative the diagram \ref{diag-p02}.
We denote by
$q ^{n,n'} _{0 }
\colon
\PP ^{n +n'} _{X/S}
\to
\PP ^{n} _{X/S} \otimes _{\O _X} \PP ^{n'} _{X/S}$
(resp.
$q ^{n,n'} _{1}
\colon
\PP ^{n +n'} _{X/S}
\to
\PP ^{n} _{X/S} \otimes _{\O _X} \PP ^{n'} _{X/S}$)
the corresponding morphism (or simply
$q ^{n,n'} _{0}$ or $q _{0}$).
We notice that
$q ^{n,n'} _{0 }= \pi ^{n+n',n} _{X/S} \otimes 1$
and
$q ^{n,n'} _{1 }=1\otimes  \pi ^{n+n',n'} _{X/S} $,
where
$ \pi ^{n _1,n _2} _{X/S} $ is the projection
$\PP ^{n _1} _{X/S}
\to
\PP ^{n _2} _{X/S} $ for any integers $n _1 \geq n _2$.
\end{empt}

\begin{lem}
\label{MontagnonL231}
For any $a \in M _X$, for any integers $n,n'\in \N$, we have 
$\delta ^{n,n'} (\mu ^{n+n'} (a))
=
\mu ^n (a) \otimes \mu ^{n'} (a)$.
\end{lem}

\begin{proof}
We copy word by word the proof of Montagnon 
of Lemma \cite[2.3.1]{these_montagnon}.
\end{proof}

\begin{dfn}
\label{sheafdiffoper}
The sheaf of differential operators of order $\leq n$ of $f$
is defined by putting
$\D  _{X/S, n}:= \H om _{\O _X} ( p _{0 *} ^{n} \PP ^{n} _{X/S} , \O _X)$.
The sheaf of differential operators of $f$
is defined by putting
$\D  _{X/S}:= \cup _{n \in \N }\D  _{X/S, n}$.

Let $P \in \D  _{X/S, n}$, $P'  \in \D  _{X/S, n'}$.
We define the product
$P P' \in \D  _{X/S, n+n'}$
to be the composition
\begin{equation}
\label{dfn-prod}
P P' \colon \PP ^{n +n'} _{X/S}
\overset{\delta ^{n,n'}}{\longrightarrow}
\PP ^{n} _{X/S} \otimes _{\O _X} \PP ^{n'} _{X/S}
\overset{\mathrm{Id}\otimes P'}{\longrightarrow}
\PP ^{n} _{X/S}
\overset{P}{\longrightarrow}
\O _{X}.
\end{equation}
\end{dfn}

\begin{prop}
\label{ring-diffop}
The sheaf $\D  _{X/S} $ is a sheaf of rings with the product as defined in \ref{dfn-prod}.
\end{prop}

\begin{proof}
Using Lemma \ref{MontagnonL231} (instead of Lemma \ref{MontagnonL231(m)}), 
we check the proposition \ref{ring-diffop} similarly to the proposition \ref{ring-diffop(m)}. 
\end{proof}

\subsection{Sheaf of principal parts of level $m$}
Let $f\colon X \to S$ be a weakly log smooth of level $m$ 
morphism of fine log-schemes.
Let
$(I _S,J _S, \gamma)$
be a quasi-coherent
$m$-PD-ideal of
$\O _S$ that such that $\gamma$ extends to $X$ (e.g. from \cite[1.3.2.(i).c)]{Be1} when $J _S + p \O _S$ is locally principal).
Recall that   
$f$ is also a weakly log smooth of level $m$ compatible with $\gamma$ (this is a consequence of \ref{gammaempty-incl}).

\begin{rem}
\label{envelope of id}
Since $\gamma$ extends to $X$, then 
the $m$-PD envelope compatible with $\gamma$ (of order $n$) of the identity of $X$ is the identity of $X$.
Indeed, using the arguments given in the proof of \ref{lem-CsThickp},
we check that the ideal $0$ of $\O _X$ is endowed with a (unique) $m$-PD structure compatible with $\gamma$. 
\end{rem}

\begin{ntn}
\label{notaDeltanetc}
Let $\Delta _{X/S}\colon X \to X \times _{S} X$ be the diagonal morphism,
$\Delta  _{X/S, (m), \gamma}:= P _{(m), \gamma} ( \Delta _{X/S}) $,
$\Delta ^{n} _{X/S, (m), \gamma}:= P ^{n }_{(m), \gamma} ( \Delta _{X/S}) $,
$\PP ^{n} _{X/S, (m), \gamma} := \PP ^{n }_{(m), \gamma} (\Delta _{X/S})$ (see Notation \ref{mPDenvelope}).
We denote by $M _{X/S, (m), \gamma}$ (resp. $M ^{n} _{X/S, (m), \gamma}$) the log structure of
$\Delta  _{X/S, (m), \gamma}$ (resp. $\Delta ^{n} _{X/S, (m), \gamma}$).
We denote abusively the target of $\Delta  _{X/S, (m), \gamma}$ by
$\Delta _{X/S, (m), \gamma}$.
Since  $\gamma$ extends to $X$,
the source of $\Delta  ^{n } _{X/S, (m), \gamma}$ is $X$, i.e.
$\Delta ^{n} _{X/S, (m), \gamma}$ is a closed immersion of the form
$X \hookrightarrow (\Spec \PP ^{n} _{X/S, (m), \gamma} , M ^{n} _{X/S, (m), \gamma})$.

Let
$p _1,\, p _0 \colon \Delta  _{X/S, (m), \gamma} \to X$ be respectively the composition of the canonical morphism
$\Delta  _{X/S, (m), \gamma} \to X \times _{S} X$  with the right and left projection
$X \times _{S} X \to X$.
Similarly we get
$p ^{n} _1,\, p ^{n} _0 \colon \Delta  ^{n } _{X/S, (m), \gamma} \to X$.
As in \ref{pistrict}, 
we check that $p _1$ and $p_0$ are strict morphisms.

If $a \in M _X$, we denote by
$\mu _{(m), \gamma}  (a)$ the unique section of
$\ker ( \O _{\Delta _{X/S, (m), \gamma}} ^{*} \to  \O _{X} ^{*} )$
such that we get in $M _{X/S, (m), \gamma}$ the equality
$p _1 ^{*} (a)= p _0 ^{*} (a) \mu _{(m), \gamma} (a)$ (see \ref{Kerexactclosedimmer}).
We get 
$\mu _{(m), \gamma} \colon 
M _{X/S, (m), \gamma}
 \to \ker ( \O _{\Delta _{X/S, (m), \gamma}} ^{*} \to  \O _{X} ^{*} )$ 
given by 
$a \mapsto \mu _{(m), \gamma}  (a)$.
Similarly we define
$\mu _{(m), \gamma} ^{n} \colon 
M ^n _{X/S, (m), \gamma} \to \ker ( \O _{\Delta ^{n}  _{X/S, (m), \gamma}} ^{*} \to  \O _{X} ^{*} )$.
\end{ntn}

\begin{prop}
[Local description of $\PP  _{X/S, (m), \gamma} $]
\label{nota-eta}
Let $(a  _{\lambda}) _{\lambda =1,\dots, r}$ be a finite set of $\Gamma (X, M_X)$. 
Put $\eta _{\lambda (m), \gamma} := \mu _{(m), \gamma}  ( a _{\lambda}) -1$,
and 
$\eta _{\lambda (m), \gamma,n} := \mu ^n _{(m), \gamma}  ( a _{\lambda}) -1$.

\begin{enumerate}
\item 
If $(a  _{\lambda}) _{\lambda =1,\dots, r}$ 
is 
a formal log basis of level $m$ compatible with $\gamma$ of $f$ then
we have  he following $\O _{X}$-$m$-PD isomorphism
   \begin{align}
   \notag
   \O _{X} <T _1 ,\dots ,T _r> _{(m),n}
   &\riso \PP ^n  _{X/S, (m), \gamma}  \\
   \label{loc-desc-Pn}
   T _\lambda &\mapsto
   \eta _{\lambda,(m), \gamma,n},
   \end{align}
where the structure of $\O _X$-module of 
$\PP ^n  _{X/S, (m), \gamma}$
is given by $p ^{n} _1$ or $p ^{n} _0$.

   \item
If $(a  _{\lambda}) _{\lambda =1,\dots, r}$ 
is a  log $p$-basis  of level $m$ compatible with $\gamma$ of $f$ then
we have  he following $\O _{X}$-$m$-PD isomorphism
   \begin{align}
   \notag
   \O _{X} <T _1 ,\dots ,T _r> _{(m)}
   &\riso \PP   _{X/S, (m), \gamma}  \\
   \label{loc-desc-P}
   T _\lambda &\mapsto
   \eta _{\lambda,(m), \gamma},
   \end{align}
   where the structure of $\O _X$-module of 
$\PP ^n  _{X/S, (m), \gamma}$
is given by $p ^{n} _1$ or $p ^{n} _0$.
   \end{enumerate}

\end{prop}

\begin{proof}
By symmetry, we can focus on the case where 
the structure of $\O _X$-module of 
$\PP ^n  _{X/S, (m), \gamma}$
is given by $p ^{n} _0$.
In the first assertion,
we are in the situation of the proposition
\ref{mPDenv-local-desc}
where $u=\Delta$ and $f$ is the left projection
$p _0\colon X \times _{S} X \to X$.
Indeed, we first remark  that
$(p _1 ^{*} (a _\lambda) ) _{\lambda =1,\dots, r}$ is 
a formal log basis  of level $m$ compatible with $\gamma$ 
of $p _0$ (indeed, the log $p$-étaleness of level $m$ property is stable under base change).
Since the $m$-PD envelope compatible with $\gamma$ of order $n$ of the identity of $X$ is $X$ (see remark \ref{envelope of id}),
proposition \ref{mPDenv-local-desc} yields the result.
\end{proof}

\begin{rem}
\begin{enumerate}
\item From the local description \ref{loc-desc-Pn}, 
we get that $\PP ^n  _{X/S, (m), \gamma} $ does not depend 
on the $m$-PD-structure (satisfying the conditions of the subsection).
Hence, from now, we reduce to the case where $\gamma = \gamma _{\emptyset}$
(see Notation \ref{gammaempty}) and we remove $\gamma $ in the notation : we simply write
$\PP ^n  _{X/S, (m)} $, $\Delta ^{n} _{X/S, (m)}$, 
$M ^{n} _{X/S, (m)}$,
$\mu _{(m)} ^{n}$,
and
$\eta _{\lambda (m), n}$.

\item When $f$ is log $p$-smooth of level $m$, from \ref{loc-desc-P}, 
$\PP  _{X/S, (m), \gamma} $ does not depend 
on the $m$-PD-structure (satisfying the conditions of the subsection).
Hence, we can remove $\gamma$ in the corresponding notation. 

\end{enumerate}
\end{rem}

\begin{rem}
\label{p1p0finite(m)}
From the local description of \ref{nota-eta}, 
we get that 
the morphisms $p ^n _1$ and $p ^n _0$ are finite (i.e. the underlying morphism of schemes is finite).
\end{rem}

\begin{rem}
For any integer $m' \geq m$, 
we remark that the canonical map 
$\PP ^n  _{X/S, (m')} \to \PP ^n  _{X/S, (m)}$ 
sends 
$\eta _{\lambda (m')}$
to 
$\eta _{\lambda (m)}$.

\end{rem}

\begin{rem}
\label{rem-cohorfine?}
Noticing that the main Theorem 
\cite[IV.3.2.4]{Ogus-Logbook} 
on log smoothness 
(this is the Theorem that leads us to our definition of log $p$-smoothness) 
is valid for coherent log structures and not only fine log structures, 
one might wonder why we are focusing on fine log structures. 
The first reason we have in mind is that 
the important tool consisting of exactifying closed immersions (see \ref{ex-cl-imm})
needs fine log structures. One might refute that in the first chapter we might 
replace in the definition of $\mathscr{C}$ (see \ref{dfnC}) the word fine by coherent (but in the other categories, e.g. 
$\mathscr{C} ^{(m)} _{\gamma}$ we keep fine log structures).
But, if we replace in \ref{nota-eta} fine log structures by coherent log structures, 
the isomorphism
\ref{loc-desc-Pn} is not any more true: 
we would have in fact 
$\O _{X ^{\mathrm{int}}} <T _1 ,\dots ,T _r> _{(m)} \riso \PP  _{X/S, (m)} $.
\end{rem}

\begin{empt}

Let $g \colon S' \to S$ be a morphism of  fine log schemes over $\Z / p ^{i+1}\Z$,
let $(I _{S'}, J _{S'},\gamma')$ be a quasi-coherent $m$-PD-ideal of $\O _{S'}$ such that
$g$ becomes an $m$-PD-morphism.
Put $X ':= X \times _{S} S'$. We suppose that
$\gamma '$ extends to $X'$.
Then, the $m$-PD-morphism
$\Delta  _{X'/S', (m)} \to \Delta  _{X/S, (m)} $
induces the isomorphism
$\Delta  _{X'/S', (m)} \riso
\Delta  _{X/S, (m)} \times _{S} S'$.
Indeed, since the morphisms
$p _0 \colon \Delta  _{X/S, (m)} \to X$
and
$p _0 \colon \Delta  _{X'/S', (m)} \to X'$
are strict,
then
the morphism
$\Delta  _{X'/S', (m)} \to
\Delta  _{X/S, (m)} \times _{S} S'$
is strict.
Hence, this is sufficient to check that
the morphism
$g ^{*} \PP _{X/S, (m)}  \to \PP _{X'/S', (m)}  $
is an isomorphism. This can be checked by using the local description of
\ref{loc-desc-Pn}.
\end{empt}

\begin{empt}
\label{m2m'}
Let $m' \geq m$ be two integers. 
Since 
$\mathscr{C'} _{n} ^{(m)}
\subset 
\mathscr{C'} _{n} ^{(m')}$,
then by using the universal property defining $\Delta ^{n} _{X/S, (m')}$
we get
a morphism
$\psi ^{n} _{m,m'}
\colon
\Delta ^{n} _{X/S, (m)}
\to
\Delta ^{n} _{X/S, (m')}$
and then the homomorphism
$\psi ^{n *} _{m,m'}
\colon
\PP ^{n} _{X/S, (m')}
\to
\PP ^{n} _{X/S, (m)} $.

From \ref{Be1-1432}, we get 
$P ^{n }_{(m)} ( P ^n (\Delta _{X/S})) 
=
P ^{n }_{(m)} ( \Delta _{X/S})$.
Hence, we get a canonical map 
$\psi ^{n} _{m,m'}
\colon
\Delta ^{n} _{X/S, (m)}
\to
\Delta ^{n} _{X/S}$
and then the homomorphism
$\psi ^{n *} _{m}
\colon
\PP ^{n} _{X/S}
\to
\PP ^{n} _{X/S, (m)} $.

Now, suppose that $X \to S$ is endowed with 
a formal log basis $(b  _{\lambda}) _{\lambda =1,\dots, r}$ of level $m'$.
With the notation of \ref{nota-eta}, we have
$\psi ^{n*} _{m,m'}
(\underline{\eta} _{(m')} ^{\{\underline{k} \} _{(m')}})=
\frac{\underline{q}!}{\underline{q}'!}
\underline{\eta} _{(m)} ^{\{\underline{k} \} _{(m)}}$, 
where
$k _\lambda = p ^{m} q _\lambda +r _\lambda$ and $k '_\lambda = p ^{m'} q ' _\lambda+r' _\lambda$ is the Euclidian division of $k _{\lambda}$ 
by respectively $p ^{m}$ and $p ^{m'}$,
$\underline{\eta} _{(m)}  ^{\{\underline{k} \} _{(m)}}:= 
\prod _{\lambda = 1} ^{r} \eta _{\lambda,(m)}  ^{\{k _{\lambda} \} _{(m)}}$,
$\underline{q} := 
\prod _{\lambda = 1} ^{r} q _{\lambda}$
and similarly with some primes.
Moreover, we compute
$\psi ^{n*} _{m}
(\underline{\eta}  ^{\underline{k} })=
\underline{q}!
\underline{\eta} _{(m)} ^{\{\underline{k} \} _{(m)}}$.

\end{empt}

\begin{ntn}
\label{Omega1}
Let $\I ^1 _{X/S, (m)}$ be the ideal of the closed immersion $\Delta ^1 _{X/S, (m)}$ and
$\Omega ^{1} _{X/S,(m)}:= (\Delta ^1 _{X/S, (m)} ) ^{-1}( \I ^1 _{X/S, (m)} )$.
Thanks to the local description \ref{loc-desc-Pnf} 
(recall from \ref{smooth=>psmooth} that since $f$ is weakly log smooth of level $m$  then
$f$ is weakly log smooth) and 
\ref{loc-desc-Pn} and the local computation of \ref{m2m'},
we check that the homomorphism
$\psi ^{1 *} _{m}
\colon
\PP ^{1} _{X/S}
\to
\PP ^{1} _{X/S, (m)} $
induces the isomorphism 
$\psi ^{1 *} _{m}
\colon 
\Omega ^{1} _{X/S}
\riso 
\Omega ^{1} _{X/S,(m)}$,
and that
the homomorphism
$\psi ^{1 *} _{m,m'}
\colon
\PP ^{1} _{X/S, (m')}
\to
\PP ^{1} _{X/S, (m)} $
induces the isomorphism 
$\psi ^{1 *} _{m,m'}
\colon 
\Omega ^{1} _{X/S,(m')}
\riso 
\Omega ^{1} _{X/S,(m)}$.
Hence, we can simply write $\Omega ^{1} _{X/S}$ instead of 
$\Omega ^{1} _{X/S,(m)}$.
\end{ntn}

\begin{rem} 
\label{nota-omega}
If $X/S$ has 
a formal log basis  
$(b_\lambda)_{\lambda=1,\dots, r}$ 
of level $m$ 
of $f$
then \ref{loc-desc-Pn} implies that $\Omega^1_{X/S}$ is free of rank $r$, 
a basis being given by the images $dlog\, b_\lambda$ of the $\eta_{\lambda,(m)}$'s. 
This implies in particular that all formal log bases 
of level $m$ 
of $f$
have the same cardinality.
 We put 
 $\omega _{X/S}:= \wedge ^{r} \Omega ^{1} _{X/S}$. 
 More generally (i.e. we do not any more assume that $X/S$ has a formal log basis), 
 we define
 $\omega _{X/S}$ in the same way. 
\end{rem}

\begin{ntn}
Let $\E$ be an $\O _X$-module.
By convention,
$\PP ^{n} _{X/S, (m)} \otimes _{\O _X} \E$
means
$p _{1*} ^{n} (\PP ^{n} _{X/S, (m)} )\otimes _{\O _X} \E$
and
$\E \otimes _{\O _X} \PP ^{n} _{X/S, (m)}$
means
$\E \otimes _{\O _X} p _{0*} ^{n} (\PP ^{n} _{X/S, (m)} )$.
For instance,
$\PP ^{n} _{X/S, (m)} \otimes _{\O _X} \PP ^{n'} _{X/S, (m)}$
is
$p _{1*} ^{n} (\PP ^{n} _{X/S, (m)} ) \otimes _{\O _X} p _{0*} ^{n'} (\PP ^{n'} _{X/S, (m)} )$.

\end{ntn}

\begin{lem}
\label{delta,n,n'}
We simply denote by
$\Delta ^{n} _{X/S, (m)} \times _{X} \Delta ^{n '} _{X/S, (m)}$
the base change of
$p _0 ^{n '} \colon \Delta ^{n '} _{X/S, (m)} \to X$
by
$p _1 ^{n} \colon \Delta ^{n} _{X/S, (m)} \to X$.
The immersion
$X \hookrightarrow \Delta ^{n} _{X/S, (m)} \times _{X} \Delta ^{n '} _{X/S, (m)}$
induced by
$X \hookrightarrow \Delta ^{n} _{X/S, (m)}$
and
$X \hookrightarrow \Delta ^{n '} _{X/S, (m)}$
is an exact closed immersion endowed with a canonical $m$-PD structure of order $n+n'$. 
By abuse of notation, we denote by 
$\Delta ^{n} _{X/S, (m)} \times _{X} \Delta ^{n '} _{X/S, (m)}$ this object of 
$\mathscr{C} _{n+n'} ^{(m)}$.
This $m$-PD structure on $\Delta ^{n} _{X/S, (m)} \times _{X} \Delta ^{n '} _{X/S, (m)}$ 
is characterized by the following property : the projections 
$q ^{n,n'} _{0} \colon \Delta ^{n} _{X/S, (m)} \times _{X} \Delta ^{n '} _{X/S, (m)}
\to \Delta ^{n} _{X/S, (m)} $
and 
$q ^{n,n'} _{1} \colon \Delta ^{n} _{X/S, (m)} \times _{X} \Delta ^{n '} _{X/S, (m)}
\to \Delta ^{n'} _{X/S, (m)} $
are morphisms of 
$\mathscr{C} _{n+n'} ^{(m)}$.
\end{lem}

\begin{proof}
Since $p ^{n} _1,\, p ^{n} _0 \colon \Delta  ^{n } _{X/S, (m), \gamma} \to X$ are strict, 
we check that 
$X \hookrightarrow \Delta ^{n} _{X/S, (m)} \times _{X} \Delta ^{n '} _{X/S, (m)}$
is an exact closed immersion. 
Using the local description \ref{loc-desc-Pn}, 
to check the uniqueness and existence of the $m$-PD-structure of order $n+n'$, we proceed similarly to 
\cite[2.1.3.(i)]{Be1} or \cite[2.3.2]{these_montagnon}.
\end{proof}

\begin{empt}
Using the universal property of the $m$-PD-envelope of order $n$,
we get a unique
morphism
$\Delta ^{n} _{X/S, (m)} \times _{X} \Delta ^{n '} _{X/S, (m)}
\to \Delta ^{n+n'} _{X/S, (m)}$ of $\mathscr{C} _{n+n'} ^{(m)}$
inducing the commutative diagram
\begin{equation}
\label{diag-p02}
\xymatrix{
{X}
\ar@{^{(}->}[r] ^-{}
\ar@{=}[d] ^-{}
&
{\Delta ^{n} _{X/S, (m)} \times _{X} \Delta ^{n '} _{X/S, (m)}}
\ar[d] ^-{}
\ar[r] ^-{}
&
{X \times _S X \times _S X}
\ar[d] ^-{p _{02}}
\\
{X}
\ar@{^{(}->}[r] ^-{}
&
{\Delta ^{n+n'} _{X/S, (m)}}
\ar[r] ^-{}
&
{X \times _S X . }
}
\end{equation}
We denote by
$\delta ^{n,n'} _{(m)}
\colon
\PP ^{n +n'} _{X/S, (m)}
\to
\PP ^{n} _{X/S, (m)} \otimes _{\O _X} \PP ^{n'} _{X/S, (m)}$
the corresponding morphism.

By replacing $p _{02}$ by $p _{01}$ (resp. $ p _{12}$),
we get a unique
morphism
$\Delta ^{n} _{X/S, (m)} \times _{X} \Delta ^{n '} _{X/S, (m)}
\to \Delta ^{n+n'} _{X/S, (m)}$
making commutative the diagram \ref{diag-p02}.
We denote by
$q ^{n,n'} _{0 (m)}
\colon
\PP ^{n +n'} _{X/S, (m)}
\to
\PP ^{n} _{X/S, (m)} \otimes _{\O _X} \PP ^{n'} _{X/S, (m)}$
(resp.
$q ^{n,n'} _{1(m)}
\colon
\PP ^{n +n'} _{X/S, (m)}
\to
\PP ^{n} _{X/S, (m)} \otimes _{\O _X} \PP ^{n'} _{X/S, (m)}$)
the corresponding morphism (or simply
$q ^{n,n'} _{0}$ or $q _{0}$).
We notice that
$q ^{n,n'} _{0 (m)}= \pi ^{n+n',n} _{X/S,(m)} \otimes 1$
and
$q ^{n,n'} _{1 (m)}=1\otimes  \pi ^{n+n',n'} _{X/S,(m)} $,
where
$ \pi ^{n _1,n _2} _{X/S,(m)} $ is the projection
$\PP ^{n _1} _{X/S, (m)}
\to
\PP ^{n _2} _{X/S, (m)} $ for any integers $n _1 \geq n _2$.
\end{empt}

The following Lemma will be useful to check the associativity of the product law of the sheaf of differential operator: 
\begin{lem}
\label{delta,n,n',n''}
We denote by 
$\Delta ^{n} _{X/S, (m)} \times _{X} \Delta ^{n '} _{X/S, (m)}\times _{X} \Delta ^{n ''} _{X/S, (m)}$
the base change of
$p _0 ^{n '} \circ q ^{n',n''} _{0}  \colon  \Delta ^{n '} _{X/S, (m)}\times _{X} \Delta ^{n ''} _{X/S, (m)}  \to X$
by
$p _1 ^{n} \colon \Delta ^{n} _{X/S, (m)} \to X$.
The exact closed immersion
$X \hookrightarrow \Delta ^{n} _{X/S, (m)} \times _{X} \Delta ^{n '} _{X/S, (m)}\times _{X} \Delta ^{n ''} _{X/S, (m)}$
induced by
$X \hookrightarrow \Delta ^{n} _{X/S, (m)}$,
$X \hookrightarrow \Delta ^{n'} _{X/S, (m)}$
and
$X \hookrightarrow \Delta ^{n ''} _{X/S, (m)}$
is endowed with a canonical $m$-PD structure. 
By abuse of notation, we denote by 
$\Delta ^{n} _{X/S, (m)} \times _{X} \Delta ^{n '} _{X/S, (m)}\times _{X} \Delta ^{n ''} _{X/S, (m)}$ this object of 
$\mathscr{C} _{n+n'+n''} ^{(m)}$.
This $m$-PD structure on $\Delta ^{n} _{X/S, (m)} \times _{X} \Delta ^{n '} _{X/S, (m)}\times _{X} \Delta ^{n ''} _{X/S, (m)}$
is characterized by the following property : the projections 
$
\Delta ^{n} _{X/S, (m)} \times _{X} \Delta ^{n '} _{X/S, (m)}\times _{X} \Delta ^{n ''} _{X/S, (m)}
\to \Delta ^{n} _{X/S, (m)} $,
$\Delta ^{n} _{X/S, (m)} \times _{X} \Delta ^{n '} _{X/S, (m)}\times _{X} \Delta ^{n ''} _{X/S, (m)}
\to \Delta ^{n'} _{X/S, (m)} $,
and
$\Delta ^{n} _{X/S, (m)} \times _{X} \Delta ^{n '} _{X/S, (m)}\times _{X} \Delta ^{n ''} _{X/S, (m)}
\to \Delta ^{n''} _{X/S, (m)} $
are morphisms of 
$\mathscr{C} _{n+n'+n''} ^{(m)}$.

\end{lem}

\begin{proof}
This is checked similarly to \ref{delta,n,n'}.
\end{proof}

\subsection{Sheaf of differential operators of level $m$}

\begin{dfn}
\label{sheafdiffoper(m)}
The sheaf of differential operators of level $m$ and order $\leq n$ of $f$
is defined by putting
$\D ^{(m)} _{X/S, n}:= \H om _{\O _X} ( p _{0, (m) *} ^{n} \PP ^{n} _{X/S, (m)} , \O _X)$.
The sheaf of differential operators of level $m$ of $f$
is defined by putting
$\D ^{(m)} _{X/S}:= \cup _{n \in \N }\D ^{(m)} _{X/S, n}$.

Let $P \in \D ^{(m)} _{X/S, n}$, $P'  \in \D ^{(m)} _{X/S, n'}$.
We define the product
$P P' \in \D ^{(m)} _{X/S, n+n'}$
to be the composition
\begin{equation}
\label{dfn-prod(m)}
P P' \colon \PP ^{n +n'} _{X/S, (m)}
\overset{\delta ^{n,n'} _{(m)}}{\longrightarrow}
\PP ^{n} _{X/S, (m)} \otimes _{\O _X} \PP ^{n'} _{X/S, (m)}
\overset{\mathrm{Id}\otimes P'}{\longrightarrow}
\PP ^{n} _{X/S, (m)}
\overset{P}{\longrightarrow}
\O _{X}.
\end{equation}
\end{dfn}

\begin{lem}
\label{MontagnonL231(m)}
For any $a \in M _X$, for any integers $n,n'\in \N$, we have 
$\delta ^{n,n'} _{(m)} (\mu ^{n+n'} _{(m)} (a))
=
\mu ^n (a) _{(m)}\otimes \mu ^{n'} _{(m)}(a)$.
\end{lem}
\begin{proof}
We copy word by word the proof of Montagnon 
of Lemma \cite[2.3.1]{these_montagnon}.
\end{proof}

\begin{prop}
\label{ring-diffop(m)}
The sheaf $\D ^{(m)} _{X/S} $ is a sheaf of rings with the product as defined in \ref{dfn-prod(m)}.
\end{prop}

\begin{proof}
We have to check the product as defined in \ref{dfn-prod(m)} is associative.
One checks the commutativity of the diagram
\begin{equation}
\label{trans-deltann'}
\xymatrix @R=0,3cm @C=0,3cm{
{\PP ^{n +n'+n''} _{X/S, (m)} }
\ar@{=}[r]
\ar[ddd] ^-{(PP') P''}
&
{\PP ^{n +n'+n''} _{X/S, (m)} }
\ar[r] ^-{\delta ^{n,n'+n''} _{(m)}}
\ar[d] ^-{\delta ^{n+n',n''} _{(m)}}
&
{\PP ^{n} _{X/S, (m)} \otimes _{\O _X} \PP ^{n'+n''} _{X/S, (m)}}
\ar[d] ^-{\mathrm{id} \otimes \delta ^{n',n''} _{(m)}}
\ar@{=}[r]
&
{\PP ^{n} _{X/S, (m)} \otimes _{\O _X} \PP ^{n'+n''} _{X/S, (m)}}
\ar[ddd] ^-{id \otimes P' P''}
\\
&{\PP ^{n+n'} _{X/S, (m)} \otimes _{\O _X} \PP ^{n''} _{X/S, (m)}}
\ar[r] ^-{\delta ^{n,n'} _{(m)}\otimes \mathrm{id}}
\ar[d] ^-{id \otimes P''}
&
{\PP ^{n} _{X/S, (m)} \otimes _{\O _X} \PP ^{n'} _{X/S, (m)} \otimes _{\O _X} \PP ^{n''} _{X/S, (m)}}
\ar[d] ^-{id \otimes id \otimes P''}
\\
&
{\PP ^{n+n'} _{X/S, (m)}}
\ar[r] ^-{\delta ^{n,n'} _{(m)}}
\ar[d] ^-{P P'}
&
{{\PP ^{n} _{X/S, (m)} \otimes _{\O _X} \PP ^{n'} _{X/S, (m)}} }
\ar[d] ^-{id \otimes P'}
\\
{\O _X}
\ar@{=}[r]
&
{\O _X}
&
{{\PP ^{n} _{X/S, (m)} \otimes _{\O _X} \PP ^{n'} _{X/S, (m)}}}
\ar[l] ^-{P}
\ar@{=}[r]
&
{{\PP ^{n} _{X/S, (m)} \otimes _{\O _X} \PP ^{n'} _{X/S, (m)}}.}
}
\end{equation}
Indeed, let us check the commutativity of the top square of the middle.
Since this is local, we can suppose that 
$f$ has a formal log basis $(a  _{\lambda}) _{\lambda =1,\dots, r}$ of level $m$. 
 Using Lemma \ref{MontagnonL231(m)} and the notation of \ref{nota-eta},
 we compute that the images of 
$ \eta _{1,(m),n+n'},\dots,  \eta _{r,(m),n+n'}$ by both maps 
$\PP ^{n +n'+n''} _{X/S, (m)} \to 
\PP ^{n} _{X/S, (m)} \otimes _{\O _X} \PP ^{n'} _{X/S, (m)} \otimes _{\O _X} \PP ^{n''} _{X/S, (m)}$
are the same. 
Using the \ref{nota-eta}, since both maps are $m$-PD-morphisms (see \ref{delta,n,n',n''} for the $m$-PD-structure), we get the desired commutativity. 
Since the commutativity of the other squares are obvious, we conclude the proof. 

\end{proof}

\begin{empt}
[Description in local coordinates]
\label{locallogp-basis}
Suppose that $X \to S$ is endowed with 
a formal log basis $(b  _{\lambda}) _{\lambda =1,\dots, r}$ of level $m$. 
With the notion of \ref{m2m'}, 
the elements
$\{\underline{\eta} _{(m)} ^{\{\underline{k} \} _{(m)}} \} _{| \underline{k}| \leq n}$
form a basis of
$\PP ^{n} _{X/S, (m)}$.
The corresponding dual basis of
$\D ^{(m)} _{X/S, n}$
will be denoted by
$\{ \underline{\partial} ^{< \underline{k}> _{(m)}} \} _{| \underline{k}| \leq n}$.
Let $\epsilon _1, \dots, \epsilon _r $  be the canonical basis of $\N ^{r}$,
 i.e. the coordinates of $\epsilon _\lambda $ are $0$ except for the $i$th term which is $1$.
 We put $\partial _\lambda := \smash{\underline{\partial} }^{< \epsilon _\lambda> _{(m)}}$.
We have the same formulas than in \cite[2.3.3]{these_montagnon}.
For instance, for any section $a \in \O _X$, for any $\underline{k}, \underline{k'}, \underline{k''}\in \N ^n$, 
\begin{gather}
\label{Mon2.3.3}
\underline{\partial} ^{< \underline{k}> _{(m)}} a
=
\sum _{\underline{i} \leq \underline{k}}
\left \{
\begin{smallmatrix}
 \underline{k}\\
 \underline{i}
\end{smallmatrix}
\right\}
\underline{\partial} ^{< \underline{k}-  \underline{i}> _{(m)}} (a)
\underline{\partial} ^{<\underline{i}> _{(m)}} ; 
\\
\label{Mon2.3.3bis}
\underline{\partial} ^{< \underline{k'}> _{(m)}}
\underline{\partial} ^{< \underline{k''}> _{(m)}}
=
\sum _{\underline{k} = \max\{\underline{k}',  \underline{k}''\} }
^{\underline{k}' + \underline{k}''}
\frac{\underline{k}!}{(\underline{k}'+\underline{k}''-\underline{k})!(\underline{k} -\underline{k}')!(\underline{k} -\underline{k}'')!}
\frac{\underline{q} _{\underline{k}'}!\underline{q} _{\underline{k}''}!}{\underline{q} _{\underline{k}}!}
 \underline{\partial} ^{< \underline{k}> _{(m)}},
\end{gather}
where $\underline{q} _{\underline{k}}$ means 
the quotient of the Euclidian division of $\underline{k}$ by  $p ^{m}$ and similarly with some primes.
The left (or right) structure of $\O _X$-algebra of $\D ^{(m)} _{X/S}$ is generated by the operators
$\partial _\lambda ^{< p ^{j}> _{(m)}}$ with $1\leq \lambda \leq r$, $0\leq j\leq m$.
These formulas yield that
$\gr \D ^{(m)} _{X/S}$ is commutative and that,
when $\underline X$ is affine and noetherian, the ring
$\Gamma (X, \D ^{(m)} _{X/S})$ is left and right noetherian.

\end{empt}

\begin{empt}
[Comparison of the local description
of differential operators with or without logarithmic structure]
\label{non logarithmic partial}
Suppose given formal log étale coordinates
$(t _\lambda)_{\lambda=1,\dots, r}$ of level $m$ of $X/S$ (see definition \ref{finite-p-basis}).
\begin{enumerate}
\item By \ref{nolong-mPDenv-local-desc},
we get the following  isomorphism of $m$-PD-$\O _{X}$-algebras
   \begin{align}
   \notag
   \O _{X} <T _1 ,\dots ,T _r> _{(m),n}
   &\riso \PP ^n _{X/S, (m)}  \\
   \label{loc-desc-Pn2}
   T _\lambda &\mapsto
   \tau  _{\lambda},
   \end{align}
where $\tau  _{\lambda}:= p ^{*} _1 (t _\lambda ) -p ^{*} _0 (t _\lambda )$.
The elements
$\{\underline{\tau} ^{\{\underline{k} \} _{(m)}} \} _{| \underline{k}| \leq n}$
form a basis of
$\PP ^{n} _{X/S, (m)}$.
The corresponding dual basis of
$\D ^{(m)} _{X/S, n}$
will be denoted by
$\{ \smash{\underline{\partial} _{\flat}} ^{< \underline{k}> _{(m)}} \} _{| \underline{k}| \leq n}$.

\item Suppose now that the $t_\lambda$'s lie in $\Gamma(X,\O_X^*)$.
Then, from the Lemma \ref{rem-lognologbasis}.\ref{rem-lognologbasis1} they are also a formal log basis of level $m$.
We have
\begin{equation}
\label{log-nonlog}
\underline \tau^{\{\underline{k}\}_{(m)}}
=
p ^{*} _0 ( \underline t^{\underline k} )
\underline \eta _{(m)} ^{\{\underline{k}\}_{(m)}} 
\hspace{.5cm} \hbox{and} \hspace{.5cm}
\smash{\underline{\partial}} ^{< \underline{k}> _{(m)}}
=
 \underline{t} ^{\underline{k}}\smash{\underline{\partial}}  ^{< \underline{k}> _{(m)}} _\flat ,
\end{equation}
where $\underline \eta _{(m)}$ (resp. $\smash{\underline{\partial}}  ^{< \underline{k}> _{(m)}}$) is defined in \ref{nota-eta} (resp. \ref{locallogp-basis}).

\item Suppose now that the $t_\lambda$'s lie in $\Gamma(X,\O_X^*)$ and that  $X\to S$ is strict.
Since $X \to S$ is strict, then
$\PP ^{n} _{\underline{X}/\underline{S}, (m)}=\PP ^{n} _{X/S, (m)}$.
This yields
$\D ^{(m)} _{\underline{X}/\underline{S}} = \D ^{(m)} _{X/S}$,
where $\D ^{(m)} _{\underline{X}/\underline{S}}$ is the sheaf of differential operators defined by Berthelot in \cite{Be1}.
Then the local description of \ref{loc-desc-Pn2} extends
that given in \cite{Be1}  when $\underline{X}/\underline{S}$ has étale coordinates.

\item Consider the following diagram
\begin{equation}
\xymatrix{
{Y}
\ar[d] ^-{f}
\ar[r]  ^-{b}
&
{A _{\N ^r} \times T}
\ar[d] ^-{}
\\
{X}
\ar[r] ^-{t}&
{ \A ^{r} \times S}
}
\end{equation}
where the right arrow is induced by
a morphism of fine log schemes of the form  $T \to S$ and by the canonical morphism
$A _{\N ^r} \to  \A ^{r} $,
the bottom arrow is induced by some formal log étale coordinates
$(t _\lambda)_{\lambda=1,\dots, r}$ of level $m$
and where the top arrow is induced by a formal log basis
$(b _\lambda)_{\lambda=1,\dots, r}$ of level $m$.
Let $\underline \eta _{(m)}$ (resp. $\smash{\underline{\partial}}  ^{< \underline{k}> _{(m)}}$)
be the element constructed from
$(b _\lambda)_{\lambda=1,\dots, r}$
as defined in \ref{nota-eta} (resp. \ref{locallogp-basis}).
Then the functorial morphisms $f ^{*} \PP ^{n} _{X/S, (m)}\to \PP ^{n} _{Y/T, (m)}$
and
$\D ^{(m)} _{Y/T} \to f ^{*}\D ^{(m)} _{X/S}$  (see \ref{chgtbasis})
are explicitly described by
\begin{equation}
\label{log-nonlog2}
\underline \tau^{\{\underline{k}\}_{(m)}}\mapsto\underline t^{\underline k} \underline \eta _{(m)} ^{\{\underline{k}\}_{(m)}} \hspace{.5cm} \hbox{and} \hspace{.5cm}
\smash{\underline{\partial}}  ^{< \underline{k}> _{(m)}}
\mapsto
\underline{t} ^{\underline{k}}\smash{\underline{\partial}}  ^{< \underline{k}> _{(m)}} _\flat.
\end{equation}
\end{enumerate}

\end{empt}

\begin{empt}
For any $m' \geq m$,
from the homomorphisms
$\psi ^{n *} _{m,m'}
\colon
\PP ^{n} _{X/S, (m')}
\to
\PP ^{n} _{X/S, (m)} $ of \ref{m2m'}, we get
by duality, the homomorphisms
$\rho _{m',m}
\colon
\D ^{(m)} _{X/S}
\to
\D ^{(m')} _{X/S}$.
Let $k _\lambda = p ^{m} q _\lambda +r _\lambda$ and $k '_\lambda = p ^{m'} q ' _\lambda+r' _\lambda$ be the Euclidian division of $k _\lambda$ and 
$k '_\lambda$ by respectively $p ^{m}$ and $p ^{m'}$.
Now, suppose that $X \to S$ is endowed with a log basis
$(b  _{\lambda}) _{\lambda =1,\dots, r}$ of level $m'$.
With its notation, we get from  \ref{m2m'} the equality
 $\rho _{m',m} (\underline{\partial} ^{< \underline{k}> _{(m)}}) =
 \frac{\underline{q}!}{\underline{q}'!}\underline{\partial} ^{< \underline{k}> _{(m')}}$.
\end{empt}

\begin{empt}
\label{chgtbasis}
Let $g \colon S' \to S$ be a morphism of fine log schemes over $\Z / p ^{i+1}\Z$.
Consider the commutative diagram
\begin{equation}
\label{funct-diag}
\xymatrix{
{X'}
\ar[r] ^-{f}
\ar[d] ^-{\pi _{X'}}
&
{X}
\ar[d] ^-{\pi _X}
\\
{S'}
\ar[r] ^-{g}
& {S}
}
\end{equation}
such that $\pi _X$ and $\pi _{X'}$ are formall log smooth of level $m$.
Using the universal property of the $m$-PD envelope, we get the
$m$-PD-morphism
$f ^{*} \PP ^{n} _{X/S, (m)} \to \PP ^{n} _{X'/S', (m)}$.
This yields the morphism
$\D ^{(m)} _{X'/S', n}
\to
f ^{*}  \D ^{(m)} _{X/S, n}$
and then
$\D ^{(m)} _{X'/S'}
\to
f ^{*}  \D ^{(m)} _{X/S}$.

When the diagram \ref{funct-diag} is cartesian (in the category of fine log schemes),
the morphism $f ^{*} \PP ^{n} _{X/S, (m)} \to \PP ^{n} _{X'/S', (m)}$
is in fact an isomorphism of rings and so is
$\D ^{(m)} _{X'/S'}
\to
f ^{*}  \D ^{(m)} _{X/S}$.

When $g=id$
and $f$ is formally log étale of level $m$, then
the morphism $f ^{*} \PP ^{n} _{X/S, (m)} \to \PP ^{n} _{X'/S, (m)}$
is in fact an isomorphism and so is
$\D ^{(m)} _{X'/S}
\to
f ^{*}  \D ^{(m)} _{X/S}$.

\end{empt}

\subsection{Logarithmic PD stratification of level $m$}
One can follow Berthelot's construction of PD stratifications of level $m$ and check properties analogous to those of
\cite{Beintro2} or \cite[2.3]{Be1} (or Montagnon logarithmic version in \cite[2.6]{these_montagnon}). Let us give a quick exposition.
Even if one might consider the étale topology, an $\O _X$-module will 
mean an $\O _X$-module for the Zariski topology.

\begin{dfn}
Let $\E$ be an $\O _X$-module. An $m$-PD-stratification (or a PD-stratification of level $m$)
is the data of a family of compatible (with respect to the projections
$\pi ^{n+1,n} _{X/S,(m)} $)
$\PP ^{n} _{X/S, (m)}$-linear isomorphisms
$$\epsilon ^{\E} _n
\colon
\PP ^{n} _{X/S, (m)} \otimes _{\O _X} \E
\riso
\E \otimes _{\O _X} \PP ^{n} _{X/S, (m)}$$
satisfying the following conditions:
\begin{enumerate}
\item $\epsilon ^{\E} _0 = \mathrm{Id} _\E$ ;
\item for any $n, n'$, the diagram
\begin{equation}
\notag
\xymatrix{
{\PP ^{n} _{X/S, (m)} \otimes _{\O _X} \PP ^{n'} _{X/S, (m)} \otimes _{\O _X} \E}
\ar[rr] ^-{\delta _{(m)} ^{n,n' *} (\epsilon ^{\E} _{n+n'})} _-{\sim}
\ar[rd] _-{q _{1 (m)} ^{n,n' *} (\epsilon ^{\E} _{n+n'})} ^-{\sim}
&&
{ \E \otimes _{\O _X} \PP ^{n} _{X/S, (m)} \otimes _{\O _X} \PP ^{n'} _{X/S, (m)}}
\\
&
{\PP ^{n} _{X/S, (m)}  \otimes _{\O _X} \E \otimes _{\O _X} \PP ^{n'} _{X/S, (m)}}
\ar[ur] _-{q _{0 (m)} ^{n,n' *} (\epsilon ^{\E} _{n+n'})} ^-{\sim}
&
}
\end{equation}
is commutative
\end{enumerate}

\end{dfn}

\begin{prop}
\label{strat-prop}
Let $\E$ be an $\O _X$-module.
The following datas are equivalent :
\begin{enumerate}
\item A structure of left $\D ^{(m)} _{X/S}$-module on $\E$ extending its structure of
$\O _X$-module.
\item A family of compatible $\O _X$-linear homomorphisms
$\theta ^{\E} _n \colon \E \to \E \otimes _{\O _X} \PP ^{n} _{X/S, (m)} $
 such that $\theta ^{\E} _0= \mathrm{Id} _\E$ and for any integers $n,n'$ the diagram
 \begin{equation}
 \label{diag-theta}
 \xymatrix{
 {\E \otimes _{\O _X} \PP ^{n} _{X/S, (m)} }
\ar[r] ^-{\mathrm{Id}\otimes \delta _{(m)} ^{n,n'} }
 &
 { \E \otimes _{\O _X} \PP ^{n} _{X/S, (m)}   \otimes _{\O _X} \PP ^{n'} _{X/S, (m)}   }
 \\
 {\E}
 \ar[u] ^-{\theta ^{\E} _{n+n'}}
  \ar[r] ^-{\theta ^{\E} _{n'}}
 &
 {\E \otimes _{\O _X} \PP ^{n'} _{X/S, (m)}  }
  \ar[u] ^-{\theta ^{\E} _{n} \otimes \mathrm{Id}}
 }
 \end{equation}
is commutative.

\item An $m$-PD-stratification on $\E$.
\end{enumerate}

An $\O _X$-linear morphism
$\phi \colon \E \to \FF$ between two left $\D ^{(m)} _{X/S}$-modules is
$\D ^{(m)} _{X/S}$-linear if and only if it commutes with the homomorphisms
$\theta _n$ (resp. $\epsilon  _n$).
\end{prop}

\begin{proof}
The proof is identical to that of
\cite[2.6.1]{these_montagnon} or
\cite[2.3.2]{Be1}.
\end{proof}

\begin{empt}
If $X \to S$ is endowed with a formal log basis $(b  _\lambda) _{\lambda =1,\dots, n}$ of level $m$ then
for any $x\in \E$ we have the Taylor development
\begin{equation}
\label{Taylordev}
\theta ^{\E} _n (x) =
\sum _{|\underline{k}|\leq n}
\underline{\partial} ^{< \underline{k}> _{(m)}} \cdot x \otimes \underline{\eta} _{(m)} ^{\{ \underline{k}\}}.
\end{equation}
\end{empt}

In order to define overconvergent isocrystals in our context (see \ref{isoc}), 
we will need the following definition and proposition.
\begin{dfn}
\label{dfn-algcomp-mod}
Let $\B$ be a commutative $\O _X$-algebra endowed with a  structure of left $\D ^{(m)} _{X/S}$-module.
We say that the structure of left $\D ^{(m)} _{X/S}$-module on $\B$ is compatible with its
structure of $\O _X$-algebra if the isomorphisms
$\epsilon _n ^{\B}$ are isomorphisms of $\PP ^{n} _{X/S, (m)} $-algebras.
This compatibility is equivalent to the following condition : for any $f, g \in \B$ and
$\underline{k}  \in \N ^{d}$,
$$\underline{\partial} ^{< \underline{k}> _{(m)}} (fg )
=
\sum _{\underline{i} \leq \underline{k}}
\left \{
\begin{smallmatrix}
 \underline{k}\\
 \underline{i}
\end{smallmatrix}
\right\}
\underline{\partial} ^{< \underline{i}> _{(m)}} (f)
\underline{\partial} ^{< \underline{k}- \underline{i}> _{(m)}} (g).
$$
\end{dfn}

\begin{prop}
\label{prop-algBotimes}
Let $\B$ be a commutative $\O _X$-algebra endowed with a compatible structure of left $\D ^{(m)} _{X/S}$-module.
Then there exists on the tensor product
$\B \otimes _{\O _X} \D ^{(m)} _{X/S}$ a unique structure of rings satisfying the following conditions
\begin{enumerate}
\item the canonical morphisms
$\B \to \B \otimes _{\O _X} \D ^{(m)} _{X/S}$
and
$\D ^{(m)} _{X/S} \to \B \otimes _{\O _X} \D ^{(m)} _{X/S}$
are homomorphisms of sheaf of rings,
\item if $X \to S$ is endowed with a formal log basis $(b  _\lambda) _{\lambda =1,\dots, n}$ of level $m$, then, for any $b \in \B$ and $\underline{k}\in \N ^{n}$, we have
$(b \otimes 1) ( 1 \otimes P) = b \otimes P$
and
$$(1 \otimes \underline{\partial} ^{< \underline{k}> _{(m)}})
(b \otimes 1)
=
\sum _{\underline{i} \leq \underline{k}}
\left \{
\begin{smallmatrix}
 \underline{k}\\
 \underline{i}
\end{smallmatrix}
\right\}
\underline{\partial} ^{< \underline{i}> _{(m)}}(b)
\otimes
\underline{\partial} ^{< \underline{k}- \underline{i}> _{(m)}}.$$

\end{enumerate}

If $\B \to \B'$ is a morphism of $\O _X$-algebras with compatible structure of  left $\D ^{(m)} _{X/S}$-modules,
then the induced
morphism
$\B \otimes _{\O _X} \D ^{(m)} _{X/S}
\to
\B '\otimes _{\O _X} \D ^{(m)} _{X/S}$
is a homomorphism of rings.
\end{prop}

\begin{proof}
We copy \cite[2.3.5]{Be1}.
\end{proof}

\subsection{Logarithmic transposition}
\label{subsection2.4}
 We suppose that $X \to S$ is endowed with a formal log basis $(b  _\lambda) _{\lambda =1,\dots, n}$ of level $m$.

\begin{ntn}
\label{ntn-transp}
Put $\underline{1} := (1,1,\dots,1) \in \N ^{n}$. 
For any $\underline{k}  \in \N ^{n}$, we set
\begin{equation}
\label{tildepartialdef}
\smash{\underline{\widetilde{\partial}} }^{< \underline{k}> _{(m)}}
:=
(-1) ^{|\underline{k}|}
\sum _{\underline{1}\leq \underline{i} \leq \underline{k}}
\smash{\underline{\partial} ^{< \underline{i}> _{(m)} }}
\left \{
\begin{smallmatrix}
\underline{k}
\\
\underline{i}
\end{smallmatrix}
\right \}
\underline{q} _{\underline{k}-\underline{i}} !
\left (
\begin{smallmatrix}
\underline{k} -\underline{1}
\\
\underline{k} -\underline{i}
\end{smallmatrix}
\right ),
\end{equation}
where $q _k$ means the quotient of the Euclidian division of $k$ by $p ^{m}$.
For any differential operator $P$ of the form
$P = \sum _{\underline{k}} a _{\underline{k}} \smash{\underline{\partial} ^{< \underline{k}> _{(m)} }}$,
we set
$\widetilde{P}:=
\sum _{\underline{k}} \smash{\underline{\widetilde{\partial}} }^{< \underline{k}> _{(m)}}  a _{\underline{k}} $.
We say that
$\widetilde{P}$
is the logarithmic transposition of $P$.

\end{ntn}

\begin{rem}
[Comparison between transposition with or without logarithmic structure]
\label{rem-transp-log-nonlog}
We suppose that $f$ is weakly smooth of level $m$ and that $b_1 ,\dots, b_n\in \O _{X} ^{*}$ 
(and following \ref{rem-lognologbasis} then they form also some formal log-étale coordinates of level $m$).
In that case, we prefer to write $t _\lambda := b _\lambda$.
With the notation of \ref{non logarithmic partial},
any differential operator $P$ of $\D ^{(m)} _{X /S }$ can be written of the form
$P = \sum _{\underline{k}} a _{\underline{k}} \smash{\underline{\partial} _{\flat}^{< \underline{k}> _{(m)} }}$.
We can extend the non logarithmic transposition as defined by Berthelot (see \cite[1.3]{Be2})
to our context by putting
$${}^{t} P :=
\sum _{\underline{k}}
(-1) ^{| \underline{k}|}\smash{\underline{\partial} _{\flat}^{< \underline{k}> _{(m)} }}a _{\underline{k}} .$$
Then we have
$$\widetilde{P}= \underline{t} \  {}^{t} P \ \frac{1}{\underline{t}},$$
where $\underline{t} = \underline{t} ^{\underline{1}} = t _1\cdots t _n$.
Indeed, it is enough to check it when
$P=\smash{\underline{\partial} }^{< \underline{k}> _{(m)}}$.
The definition of \ref{tildepartialdef} was precisely introduced  to get
$\smash{\underline{\widetilde{\partial}} }^{< \underline{k}> _{(m)}} = \underline{t} \  {}^{t}
\smash{\underline{\partial} }^{< \underline{k}> _{(m)}}
\ \frac{1}{\underline{t}}$.

One reason to introduce  the logarithmic transposition is the formula \ref{rightD-mod-omegaformal}.

\end{rem}

\begin{prop}
\label{prop-tildePQ}
For any differential operators $P$ and $Q$,
we have
$\widetilde{PQ}= \widetilde{Q}\widetilde{P}$.

\end{prop}

\begin{proof}
0) When $P \in \O _X$, the proposition is obvious.

1) Suppose that
$P= \smash{\underline{\partial} ^{< \underline{k}> _{(m)} }}$
and $Q = a \in \O _X$.
For any tuples 
$\underline{i},  \underline{j} \in \N ^{n}$ so that
$\underline{i} \leq \underline{j}$, we put
$\alpha _{\underline{i}, \underline{j}}  :=
(-1) ^{|\underline{j}|}
\left \{
\begin{smallmatrix}
\underline{j}
\\
\underline{i}
\end{smallmatrix}
\right \}
\underline{q} _{\underline{j}-\underline{i}} !
\left (
\begin{smallmatrix}
\underline{j} -\underline{1}
\\
\underline{j} -\underline{i}
\end{smallmatrix}
\right )$
if $\underline{1}\leq \underline{i} $
and
$\alpha _{\underline{i}, \underline{j}} := 0 $ otherwise.
For any 
$\underline{h},  \underline{i} ,
\underline{j},  \underline{k} 
\in \N ^{n}$ so that
$\underline{h} \leq\underline{i} \leq \underline{j} \leq \underline{k} $,
we put
$P _{\underline{h} ,\underline{i} , \underline{j} , \underline{k} }:=
\alpha _{\underline{i}, \underline{j}}
\left <
\begin{smallmatrix}
\underline{i}
\\
\underline{h}
\end{smallmatrix}
\right >
\left <
\begin{smallmatrix}
\underline{k}
\\
\underline{j}
\end{smallmatrix}
\right >
\smash{\underline{\partial} ^{< \underline{i} -\underline{h}> _{(m)} }}
\smash{\underline{\partial} ^{< \underline{k} -\underline{j}> _{(m)} }}$.
On one side we have
$a\smash{\underline{\widetilde{\partial}} }^{< \underline{k}> _{(m)}}
=
\sum _{\underline{h} \leq \underline{k} }
\alpha _{\underline{h}, \underline{k}} a
\smash{\underline{\partial} ^{< \underline{h}> _{(m)} }}$
and on the other side using twice the formula \ref{Mon2.3.3} we compute
$(\smash{\underline{\partial} ^{< \underline{k}> _{(m)} }} a) ^\sim
=
\sum _{\underline{h} \leq\underline{i} \leq \underline{j} \leq \underline{k} }
P _{\underline{h} ,\underline{i} , \underline{j} , \underline{k} } (a)
\smash{\underline{\partial} ^{< \underline{h}> _{(m)} }}$.
Hence, when $\underline{h}$ and $\underline{k}$ are fixed,
this is sufficient to check
$\alpha _{\underline{h}, \underline{k}} =
\sum _{\underline{h} \leq\underline{i} \leq \underline{j} \leq \underline{k} }
P _{\underline{h} ,\underline{i} , \underline{j} , \underline{k} }$.
Since the coefficients of the differential operators $P _{\underline{h} ,\underline{i} , \underline{j} , \underline{k} }$ are integers,
shrinking $X$ if necessary,
we reduce to check the equality when the log structures are trivial.
With the remark \ref{rem-transp-log-nonlog}, we obtained the desired equality.

2) When $P= \smash{\underline{\partial}} ^{< \underline{k}> _{(m)}} $
and $Q = \smash{\underline{\partial} ^{< \underline{k'}> _{(m)}} }$,
using the formula \ref{Mon2.3.3},
we check the equality
$\widetilde{PQ}
=
\smash{\underline{\widetilde{\partial}} }^{< \underline{k'}> _{(m)}}
\smash{\underline{\widetilde{\partial}} }^{< \underline{k}> _{(m)}} $.
Indeed, with the remark \ref{rem-transp-log-nonlog},
we notice that the formula we have to check is the same as that obtained
using the analogy $\widetilde{P}= \underline{t} \  {}^{t} P \ \frac{1}{\underline{t}}$.

3) Suppose $P= \smash{\underline{\partial}} ^{< \underline{k'}> _{(m)}} $
and
$Q=  \smash{\underline{\partial}} ^{< \underline{k''}> _{(m)}} a$, with $a \in \O _X$.
From \ref{Mon2.3.3}, we have the equality
$\underline{\partial} ^{< \underline{k'}> _{(m)}}
\underline{\partial} ^{< \underline{k''}> _{(m)}} a
=
\sum _{\underline{k} = \max\{\underline{k}',  \underline{k}''\} }
^{\underline{k}' + \underline{k}''}
\beta _{\underline{k}, \underline{k}', \underline{k}''}  \underline{\partial} ^{< \underline{k}> _{(m)}}a,$
 where
$\beta _{\underline{k}, \underline{k}', \underline{k}''}:=
\frac{\underline{k}!}{(\underline{k}'+\underline{k}''-\underline{k})!(\underline{k} -\underline{k}')!(\underline{k} -\underline{k}'')!}
\frac{\underline{q} _{\underline{k}'}!\underline{q} _{\underline{k}''}!}{\underline{q} _{\underline{k}}!}\in \Z$.
Hence, from the step 1), we get
$(\underline{\partial} ^{< \underline{k'}> _{(m)}}
\underline{\partial} ^{< \underline{k''}> _{(m)}} a ) ^{\sim}
=
a \sum _{\underline{k} = \max\{\underline{k}',  \underline{k}''\} }
^{\underline{k}' + \underline{k}''}
\beta _{\underline{k}, \underline{k}', \underline{k}''}  \smash{\underline{\widetilde{\partial}}} ^{< \underline{k}> _{(m)}}$
Applying $^\sim$ to the equality \ref{Mon2.3.3} and using the step 1),
we obtain
$\sum _{\underline{k} = \max\{\underline{k}',  \underline{k}''\} }
^{\underline{k}' + \underline{k}''}
\beta _{\underline{k}, \underline{k}', \underline{k}''}  \smash{\underline{\widetilde{\partial}}} ^{< \underline{k}> _{(m)}}
=
  \smash{\underline{\widetilde{\partial}}} ^{< \underline{k''}> _{(m)}}
   \smash{\underline{\widetilde{\partial}}} ^{< \underline{k'}> _{(m)}}$.
   Again using the step 1), this yields
$  (\underline{\partial} ^{< \underline{k'}> _{(m)}}
\underline{\partial} ^{< \underline{k''}> _{(m)}} a ) ^{\sim}
=
(\underline{\partial} ^{< \underline{k''}> _{(m)}} a ) ^{\sim}
\smash{\underline{\widetilde{\partial}}} ^{< \underline{k'}> _{(m)}} $.

4) By additivity, using part 0) and part 4), we check the proposition.
\end{proof}

\begin{rem}
[Logarithmic transposition at the level $0$]
At the level $0$, any differential operator of
$\D ^{(0)} _{X /S }$ can be written uniquely in the form
$P = \sum _{\underline{k}} a _{\underline{k}} \smash{\underline{\partial}} ^{\underline{k}}$,
where $\smash{\underline{\partial}} ^{\underline{k}}$ has not to be confounded with
$\smash{\underline{\partial} }^{< \underline{k}> _{(0)}}$.
Since $\smash{\underline{\widetilde{\partial}}} = - \smash{\underline{\partial}}$,
we get
$\widetilde{P} =
\sum _{\underline{k}} (-1) ^{| \underline{k}|} \smash{\underline{\partial}} ^{\underline{k}} a _{\underline{k}} $.
\end{rem}

\begin{prop}
\label{tildetilde}
For any differential operator $P$, we have
$\widetilde{\widetilde{P}} = P$.
\end{prop}

\begin{proof}
Using \ref{prop-tildePQ}, we reduce to the case where
$P= \underline{\partial} ^{< \underline{k}> _{(m)}}$, which is left to the reader.
\end{proof}

\begin{rem}
Recall that from \ref{Mon2.3.3}, 
we have
$\underline{\partial} ^{< \underline{k}> _{(m)}} a
=
\sum _{\underline{i} \leq \underline{k}}
\left \{
\begin{smallmatrix}
 \underline{k}\\
 \underline{i}
\end{smallmatrix}
\right\}
\underline{\partial} ^{< \underline{k}-  \underline{i}> _{(m)}} (a)
\underline{\partial} ^{<\underline{i}> _{(m)}}$.
Using \ref{prop-tildePQ},
this yields
\begin{equation}
\label{2.5Montagnonsim}
a\smash{\underline{\widetilde{\partial}}} ^{< \underline{k}> _{(m)}}
=
\sum _{\underline{i} \leq \underline{k}}
\smash{\underline{\widetilde{\partial}}} ^{<\underline{i}> _{(m)}}
\left \{
\begin{smallmatrix}
 \underline{k}\\
 \underline{i}
\end{smallmatrix}
\right\}
\underline{\partial} ^{< \underline{k}-  \underline{i}> _{(m)}} (a).
\end{equation}
In the formula \ref{2.5Montagnonsim}, beware that we can not
replace $\smash{\underline{\widetilde{\partial}}}$ by $\smash{\underline{\partial}}$.
\end{rem}

\begin{empt}
The logarithmic transposition commutes with
the canonical morphism
$\D ^{(m)} _{X/S} \to \D ^{(m+1)} _{X/S}$.
\end{empt}

\section{Differential operators over fine log formal schemes}

We recall that Shiho introduced the notion of log formal $\V$-schemes (see \cite[2.1.1.(4)]{Shiho-log-isocI}) as follows:
A log formal $\V$-scheme $\X$ is a formal $\V$-scheme $\underline{\X}$ endowed with a logarithmic structure $\alpha \colon M _{\X } \to \O _\X$,
where $\O _{\X}:= \O _{\underline{\X}}$
(this means that $\alpha$ is a logarithmic morphism of sheaves of monoids for the étale topology over $\X$,
i.e. $\alpha$ is such that $\alpha ^{-1} (\O _{\X} ^*)\to  \O _{\X} ^*$ is an isomorphism).
When $M _{\X }$ is fine as sheave for the étale topology over the special fiber of $\X$
(i.e. when $M _{\X }$ is integral and $M _{\X }$ is coherent in the sense defined at the end of the remark
\cite[II.2.1.2]{Ogus-Logbook}), we say that 
the logarithmic structure $M _{\X }$ is fine. 
We say that 
$\X$ is a fine log formal $\V$-scheme if $M _{\X }$ is fine.
Let $\S:= \Spf \, \V$ be the formal $\V$-scheme (endowed with the trivial log structure).
A fine $\S$-log formal scheme $\X$ will be a morphism of
fine log formal $\V$-schemes of the form $\X \to \S$.

If $\X$ is a fine log formal $\V$-scheme and $i\in\N$,
then we denote by $X _{i} $ the fine log $\V/ \pi ^{i+1} \V$-scheme
so that $\underline{X} _{i}:= \underline{\X} \times _{\Spf (\V)} \Spec ( \V/ \pi ^{i+1} \V)$
and the morphism
$X _{i}\to \X$ is strict.
For $i=0$, we can simply denote $X _0$ by  $X$.
If $f\colon \X \to \Y$ is a morphism of fine log formal $\V$-schemes, then
we denote by
 $f _i\colon X _{i} \to Y _i$ the induced morphism of fine log-schemes over $\V/ \pi ^{i+1} \V$.
We remark that if $f\colon \X \to \Y$ is a morphism of fine 
$\S$-log formal schemes,
then  $f _i\colon X _{i} \to Y _i$ is a morphism of fine 
log $S _i$-schemes.

\subsection{From log schemes to formal log schemes}

\begin{empt}
[Charts for log formal $\V$-schemes]
\label{formalcharts}
Let $P$ be a fine monoid and 
$\V \{ P \}$ be the $p$-adic completion of $\V [ P ]$.
Since $\V$ is fixed, we denote by
$\mathfrak{A} _P$ the fine log formal $\V$-scheme
whose underlying formal $\V$-scheme is $\Spf (\V \{ P \})$
and whose log structure is the log structure associated with the pre-log structure
induced canonically by $P \to \V \{ P \}$.

Let $\X$ be a fine $\S$-log formal scheme.
We denote by $P _\X$ 
the sheaf associated to the constant presheaf of $P$ over $\X$.
Following Shiho's definition of \cite[2.1.7]{Shiho-log-isocI},
a chart of $\X$ is a morphism of monoids $\alpha \colon P _\X \to \O _{\X}$
whose associated  log structure  is isomorphic to $M _\X \to \O _{\X}$.
A chart of $\X$ is equivalent to the data of
a strict morphism of the form
$\X \to \mathfrak{A} _P$.

\end{empt}

\begin{lem}
\label{inductivelimites-presurj}
Let $\X$ be a fine $\S$-log formal scheme.
Let $i\geq 0$ be an integer.  
Then, the morphisms
$\O _{\X} ^* \to \O _{X _i} ^*$
and 
$M _\X  \to M _{X _i}$ are surjective. 
\end{lem}

\begin{proof}
The fact that $\O _{\X} ^* \to \O _{X _i} ^*$ is surjective
comes from the fact that $\O _{\X}$ is complete for the $p$-adic topology.
The fact that
$M _\X  \to M _{X _i}$ is surjective
is étale local on $\X$. 
Hence, we can suppose 
there exists a fine monoid $P$  and a morphism of 
sheaves of monoids $\alpha \colon P _\X \to \O _{\X}$ (here $P _\X$ means the sheaf associated to 
the constant presheaf of $P$ over $\X$)
which induces the isomorphism of sheaves of monoids
$P _{\X} \oplus _{\alpha ^{-1} (  \O ^{*} _{\X})} \O ^{*} _{\X} \riso M _{\X}$
and the isomorphism
$P _{X _i} \oplus _{\alpha _i ^{-1} (  \O ^{*} _{X _i})} \O ^{*} _{X _i} \riso M _{X _i}$. 
Since 
$\O ^{*} _{\X}  \to  \O ^{*} _{X _i}$ is surjective, we conclude.

\end{proof}

\begin{prop}
\label{inductivelimites}
Let $\X$ be a fine $\S$-log formal scheme.
Then, in the category of fine $\S$-log formal schemes,
$\X$ is the inductive limit of the system $(X _i) _i$.
\end{prop}

\begin{proof}
From \cite[I.10.6.1]{EGAI}, $\underline{\X}$ is the inductive limit of the system
$(\underline{X} _i ) _i$. It remains to check that the canonical morphism of sheaves of monoids
$M _{\X} 
\to 
\underleftarrow{\lim} _i \, 
M _{X _i} $
is an isomorphism.
Since this is étale local on $\X$ 
and since $\X$ is fine then we can suppose 
there exists a fine monoid $P$  and a morphism of 
sheaves of monoids $\alpha \colon P _\X \to \O _{\X}$ 
which induces the isomorphism of sheaves of monoids
$P _{\X} \oplus _{\alpha ^{-1} (  \O ^{*} _{\X})} \O ^{*} _{\X} \riso M _{\X}$.
Let $i\geq 0$ be an integer.  We get the morphism of sheaves of monoids $\alpha _i \colon P _{X _i} \to \O _{X _i}$
which induces the isomorphism
$P _{X _i} \oplus _{\alpha _i ^{-1} (  \O ^{*} _{X _i})} \O ^{*} _{X _i} \riso M _{X _i}$. 
Hence, we reduce to prove that 
the canonical map
$P _{\X} \oplus _{\alpha ^{-1} (  \O ^{*} _{\X})} \O ^{*} _{\X} 
\to 
\underleftarrow{\lim} _i \, 
P _{X _i} \oplus _{\alpha _i ^{-1} (  \O ^{*} _{X _i})} \O ^{*} _{X _i} $
is an isomorphism.
 We put  
 $\FF _i := P _{X _i} ``\oplus" _{\alpha  _i ^{-1} (  \O ^{*} _{X _i})} \O ^{*} _{X _i}$
 where 
 $``\oplus"$ means that the amalgamated sum is computed in the category of 
 presheaves. 
We put $\E _i := P _{X _i} \oplus \O ^{*} _{X _i} $, 
 $\theta _{i} \colon \E _i \to \FF _i$ the canonical surjective morphism, 
  $\G _i := P _{X _i} \oplus _{\alpha  _i ^{-1} (  \O ^{*} _{X _i})} \O ^{*} _{X _i}$
  and
 $\epsilon _i \colon \FF _i \to \G _i$ the canonical morphism from a presheaf to its
 associated sheaf. 
 We put $\phi _i := \epsilon _i \circ \theta _i$. 
 We denote by $\pi _i \colon \O ^{*} _{X _{i+1}} \to \O ^{*} _{X _{i}}$,
 $\pi _i \colon \E _{i+1} \to \E _{i}$
 $\pi _i \colon \FF _{i+1} \to \FF _{i}$,
 $\pi _i \colon \G _{i+1} \to \G _{i}$ the canonical projections.
 Let $\U\to \X$ be an étale map such that $\U$ is connected.

\noindent \ 1) Let $ s _{i+1} \in \FF _{i+1} ( U _{i+1})$ and $s _{i}:= \pi _i ( s _{i+1}) \in \FF _i ( U _{i})$.
Then the canonical map 
$\pi _i \colon \theta _{i+1} ^{-1} (s _{i+1}) \to \theta _{i} ^{-1} (s _{i}) $
induced by  $\pi _i \colon \E _{i+1} (U _{i+1}) \to \E _{i} (U _i) $
is a bijection. 

a) We check the injectivity. Let $(x, a), (x', a') \in \theta _{i+1}^{-1} (s _{i+1}) $ such that 
$\pi _i (x, a) = \pi _i (x', a')$ (where $x,x' \in P$ and $a,a'\in \O ^{*} _{X _{i+1}} (U _{i+1}) $). 
The latter equality yields $x = x'$. Since $P$ is integral, 
$\theta _{i+1}(x, a)=\theta _{i+1} (x, a')$ implies $a = a'$
(for the computation, use the remark of \cite[1.3]{Kato-logFontaine-Illusie}).

b) We check the surjectivity. Let
$ (y, b) \in \theta _i ^{-1} (s _{i}) $. 
 We remark that 
 $\alpha ^{-1} (  \O ^{*} _{\X}) (\U) 
 =
 \alpha _i ^{-1} (  \O ^{*} _{X _i}) (U _i) $ and we denote it by $Q$. 
Since $\theta _{i+1}$ is an epimorphism (in the category of presheaves)
then there exists $(x, a)\in \theta _{i+1}^{-1} (s _{i+1}) $.
Since $\pi _i (x, a) = (x, \pi _i (a)) \in \theta _i ^{-1} (s _{i}) $,
there exists $q, q ' \in Q (U _{i})$ such that $\pi _i (a) \alpha _i (q) = b \alpha _i (q')$ and $x q '= y q$ (see the remark of \cite[1.3]{Kato-logFontaine-Illusie}). 
Set $a ':= a \alpha _{i+1} (q)   \alpha _{i+1} (q') ^{-1}$. Then 
$\pi _i (a') = b$ and $\theta _{i+1} ( x, a)  = \theta _{i+1} ( y, a ') $, i.e. 
$\pi _i ( y, a ') =  (y, b) $ and
$( y, a ') \in \theta _{i+1} ^{-1} (s _{i+1})$.

\noindent \ 2)  Let $ t _{i+1} \in \G _{i+1} ( U _{i+1})$ and $t _{i}:= \pi _i ( t _{i+1}) \in \G _i ( U _{i})$.
Then the canonical map 
 $\pi _i \colon \phi _{i+1} ^{-1} (t _{i+1}) \to \phi _i ^{-1} (t _{i}) $
is a bijection. 

a)  We check the injectivity. Let $r,r' \in \phi _{i+1} ^{-1} (t _{i+1})$ such that $\pi _i (r ) = \pi _i (r')$.
There exists an étale covering $(\U _\lambda \to \U) _\lambda$ of $\U$ such that 
$\theta _{i+1}(r) |U _\lambda
=
\theta (r') _{i+1} |U _\lambda$.
From 1) (applied for $\U _\lambda$ instead of $\U$), this yields 
$r |U _\lambda
=
r' |U _\lambda$.
Hence, $r =r'$.

b)  We check the surjectivity. Let $r \in \phi _i ^{-1} (t _{i})$.
Put $s := \theta _i (r)$.
There exist an étale covering $(\U _\lambda \to \U )_\lambda $ of $\U$
and sections $ s _{\lambda} \in \FF _{i+1} (U _{\lambda})$
such that 
$\epsilon _{i+1} (s _{\lambda}) = t _{i+1} | U _{\lambda}$
and 
$\pi _i ( s _{\lambda}) = s | U _{\lambda}$.
From 1.b), there exists $r _{\lambda}\in \E _{i+1} (U _{\lambda})$ such that 
$\pi _i ( r _{\lambda}) = r | U _{\lambda}$ and 
$\theta _{i+1} (r _{\lambda}) = s _{\lambda} $.
Hence, 
$\pi _i ( r _{\lambda}) = r | U _{\lambda}$ and 
$\phi _{i+1} (r _{\lambda}) = t _{i+1} | U _{\lambda}$.
From 2.a), this yields that 
$(r _{\lambda}) _{\lambda}$ come from a section 
of $\E _{i+1} (U _{i+1})$.

3) 
Now, let us check that the canonical map
$P _{\X} \oplus _{\alpha ^{-1} (  \O ^{*} _{\X})} \O ^{*} _{\X} 
\to 
\underleftarrow{\lim} _i \, 
P _{X _i} \oplus _{\alpha _i ^{-1} (  \O ^{*} _{X _i})} \O ^{*} _{X _i} $
is an isomorphism.
First, we start with the injectivity.
As above, put
$\E :=P _{\X} \oplus \O ^{*} _{\X}  $,
$\FF: = P _{\X} ``\oplus" _{\alpha ^{-1} (  \O ^{*} _{\X})} \O ^{*} _{\X}$,
$\G:= P _{\X} \oplus _{\alpha ^{-1} (  \O ^{*} _{\X})} \O ^{*} _{\X} $,  
$\theta \colon \E \to \FF$, 
$\epsilon \colon \FF \to \G$, 
$\phi:= \epsilon \circ \theta$.
Let $(x,a),(y,b)\in \E (\U)$ such that
the image of 
$\phi  (x, a)$ and $\phi  (y, b )$ in 
$\underleftarrow{\lim} _i \, 
P _{X _i} \oplus _{\alpha _i ^{-1} (  \O ^{*} _{X _i})} \O ^{*} _{X _i} 
(\U)$
are equal (where $x, y \in P$, and $a,b \in  \O ^{*} _{\X} (\U)$). 
Since the injectivity is locally etale, 
we reduce to check that $\phi  (x, a)=\phi  (y, b )$.
Denote by $(x,a _i),(y,b _i)\in \E_i$ the image of $(x,a),(y,b)$.
Shrinking $\U$ if necessary, we can suppose that 
$\theta _0 (x, a _0) =\theta _0 (y, b _0) $.
Doing the same computation as in 1.b), 
we check there exists $c \in \O ^{*} _{\X}$ 
such that 
$\theta  (x, a ) =\theta  (y, c) $.
Moreover, since $P$ is integral, 
we check that $\phi _i (y, c _i) =\phi _i (y, b _i)$ if and only if 
$c _i = b _i$ (since this is etale local, we reduce to check 
$\theta _i (y, c _i) =\theta _i (y, b _i)$ if and only if 
$c _i = b _i$).
Hence, $b=c$, which implies 
$\phi  (x, a)=\phi  (y, b )$.
Hence, we have checked the injectivity. 
The surjectivity is an easy consequence of 2).
\end{proof}

\begin{dfn}
We define the category of strict inductive systems of noetherian fine log schemes over $(S _i) _{i\in \N}$ as follows. 
A strict inductive system of noetherian fine log schemes over $(S _i) _{i\in \N}$ is the data, 
for any integer $i\in \N$, of 
a noetherian fine $S _i$-log scheme $X _i$, 
of  an exact closed $S _i$ immersion  $X _i \hookrightarrow X _{i+1}$ 
 such that the induced morphism $X _i \to X _{i+1} \times _{S _{i+1}} S _{i}$ is an isomorphism. 
 A morphism 
 $(X _i) _{i\in \N} 
 \to 
 (Y _i) _{i\in \N}$
 of strict inductive systems of noetherian fine log schemes over $(S _i) _{i\in \N}$
 is a family of 
$S _i$-morphism 
$X _i \to Y _i$ 
making commutative the diagram
\begin{equation}
\notag
\xymatrix{
{ X _i} 
\ar[r] ^-{}
\ar[d] ^-{}
& 
{X _{i+1} } 
\ar[d] ^-{}
\\ 
{Y _i} 
\ar[r] ^-{}
& 
{Y _{i+1.} } 
}
\end{equation}
\end{dfn}

\begin{lem}
\label{int2fine-lem}
Let $X$ be a scheme. Let $M\to N$ be a local  (see Definition \cite[II.1.1.2]{Ogus-Logbook}) 
morphism of sheaves (for the étale topology)  of monoids over $X$. 
Suppose $M$ integral, $N$ fine (i.e. $N$ is integral and $N$ is coherent in the sense defined at the end of the remark
\cite[II.2.1.2]{Ogus-Logbook})
and $\overline{M}=\overline{N}$. 
Then $M$ is fine. 
\end{lem}

\begin{proof}
Let us fix some notation. Let $\overline{x} $ be a geometric point of $X$.
Since $N$ is fine, 
using \cite[I.1.2.1 and II.1.1.11]{Ogus-Logbook}, we check that 
$\overline{N}  _{\overline{x}}$ is fine.
Since $\overline{M} 
=
\overline{N} $, 
we get that
$\overline{M}  _{\overline{x}}$ is fine.
Hence, there exist a free $\Z$-module of finite type 
$L$ endowed with a morphism 
$\alpha \colon L \to M _{\overline{x}} ^\mathrm{gr}$ 
such that 
the composition of $\alpha$ with 
the projection
$M _{\overline{x}} ^\mathrm{gr}
 \to 
\overline{M} _{\overline{x}} ^\mathrm{gr}$
is surjective
(following Ogus's terminology appearing in \cite[II.2.2.10]{Ogus-Logbook}, 
this means $\alpha \colon L \to M _{\overline{x}} ^\mathrm{gr}$ is a markup of $M _{\overline{x}}$).
We put 
$P := L \times _{M ^\mathrm{gr} _{\overline{x}}} M _{\overline{x}}$.
Since $M _{\overline{x}}$ is integral, then 
$M _{\overline{x}}
\to 
\overline{M} _{\overline{x}}$
is exact 
(see \cite[I.4.1.3.1]{Ogus-Logbook}).
Hence, we get the equality
$P:=L \times _{M ^\mathrm{gr} _{\overline{x}}} M _{\overline{x}}
=
L \times _{\overline{M} ^\mathrm{gr} _{\overline{x}}} \overline{M} _{\overline{x}}$.
From 
\cite[I.2.1.9.6]{Ogus-Logbook},
since $L$ and $\overline{M} _{\overline{x}}$ are fine and since $\overline{M} ^\mathrm{gr} _{\overline{x}}$ is integral, then 
$P=L \times _{\overline{M} ^\mathrm{gr} _{\overline{x}}} \overline{M} _{\overline{x}}$ is fine.

Let $P\to M _{\overline{x}}$ be the projection. 
Using \cite[II.2.2.5]{Ogus-Logbook}, there exist an étale neighborhood $u \colon U\to X$ of ${\overline{x}}$ 
and a morphism of monoids 
$P\to M (U)$ inducing $P \to M _{\overline{x}}$.
Let $\beta \colon P _U \to u ^* M$ be the corresponding morphism.
We get the factorization of $\beta$ of the form
$P _U \to P _U ^{\beta} \overset{\beta ^\mathrm{a}}{\longrightarrow}u ^* M$,
where $\beta ^\mathrm{a}$ is the sharp localisation of $\beta$.

We prove that, shrinking $U$ is necessary, 
the morphism $\beta ^\mathrm{a}$ is an isomorphism (and then $M$ is coherent).
Following \cite[I.4.1.2]{Ogus-Logbook}, 
for any geometric point $\overline{y}$ of $U$, 
since $\beta ^\mathrm{a} _{\overline{y}}$ is sharp and since 
$M _{\overline{y}}$ is quasi-integral, 
we check that the morphism
$\beta ^\mathrm{a} _{\overline{y}}
\colon 
(P _U ^{\beta} ) _{\overline{y}} \to M _{\overline{y}}$
is an isomorphism if and only if 
$\overline{\beta ^\mathrm{a}} _{\overline{y}}
\colon 
\colon 
\overline{(P _U ^{\beta} )} _{\overline{y}} 
\to \overline{M} _{\overline{y}}$ is an isomorphism
(recall $\overline{M} _{\overline{y}} =\overline{M _{\overline{y}} }$).
Using \cite[II.1.1.11.1]{Ogus-Logbook}, we check that the canonical morphism
$P / \beta _{\overline{y}} ^{-1} (M ^* _{\overline{y}})
\to
\overline{(P _U ^{\beta} )} _{\overline{y}} $
is an isomorphism.
Hence, $\beta ^\mathrm{a} _{\overline{y}}$ is an isomorphism if and only if the canonical morphism
$P / \beta _{\overline{y}} ^{-1} (M ^* _{\overline{y}})
\to \overline{M} _{\overline{y}}$
is an isomorphism.

Let $\beta _0$ be the composition of $\beta$ with 
$u ^* M \to u ^* N$.
We get the factorization of $\beta _0$ of the form
$P _U \to P _U ^{\beta _0} \overset{\beta _0 ^\mathrm{a}}{\longrightarrow}u ^* N$,
where $\beta _0 ^\mathrm{a}$ is the sharp localisation of $\beta _0$.
Since
$P:=L \times _{M ^\mathrm{gr} _{\overline{x}}} M _{\overline{x}}
=
L \times _{\overline{M} ^\mathrm{gr} _{\overline{x}}} \overline{M} _{\overline{x}}
=
L \times _{\overline{N} ^\mathrm{gr} _{\overline{x}}} \overline{N} _{{\overline{x}}}
=
L \times _{N _{{\overline{x}}} ^\mathrm{gr}} N _{{\overline{x}}}$,
since $N$ is fine, 
following \cite[II.2.2.11]{Ogus-Logbook}
(which is checked similarly than \cite[2.10]{Kato-logFontaine-Illusie}), 
replacing $U$ if necessary, 
we can suppose that $\beta ^\mathrm{a} _{0}$ is an isomorphism.
For any geometric point $\overline{y}$ of $U$, 
this yields that the morphism
$\beta ^\mathrm{a} _{0,\overline{y}}
\colon 
(P _U ^{\beta _0} ) _{\overline{y}} \to N _{\overline{y}}$
is an isomorphism.
Hence so is 
$\overline{\beta _0 ^\mathrm{a}} _{\overline{y}}
\colon 
\overline{(P _U ^{\beta _0} )} _{\overline{y}} 
\to 
\overline{N} _{\overline{y}}$, i.e., 
$P / \beta _{0,\overline{y}} ^{-1} (N ^* _{\overline{y}})
\to \overline{N} _{\overline{y}}$
is an isomorphism.

Since the morphism 
$M 
\to 
N$
is local, 
the induced morphism 
$M ^* \to M \times _{N} N ^*$
is an isomorphism. 
Hence, we get 
$M ^* _{\overline{y}}\to M _{\overline{y}} \times _{N _{\overline{y}} } N ^* _{\overline{y}} $, i.e. 
the morphism
$M _{\overline{y}}
\to 
N _{\overline{y}}$
is local.
Hence, we get 
$\beta _{0,\overline{y}} ^{-1} (N ^*_{\overline{y}})
=
\beta _{\overline{y}} ^{-1} (M ^* _{\overline{y}})$. 
Recalling that 
$\overline{M} _{\overline{y}}
=
\overline{N} _{\overline{y}}$, 
this implies that 
$P / \beta _{\overline{y}} ^{-1} (M ^* _{\overline{y}})
\to \overline{M} _{\overline{y}}$
is an isomorphism.
Hence, we are done.

\end{proof}

\begin{prop}
\label{inductivelimites-lem}
Let 
$(X _i) _{i\in \N}$
be a strict inductive systems of noetherian fine log schemes over $(S _i) _{i\in \N}$.
Then 
$\underrightarrow{\lim} _i \, X _i $ 
is a fine $\S$-log formal scheme.
Moreover, the canonical morphism
$X _i \to (\underrightarrow{\lim} _i \, X _i ) \times _{\S} S _i$
is an isomorphism of fine log schemes. 
\end{prop}

\begin{proof}
We already know that 
$\underline{\X}:= \underrightarrow{\lim} _i \, 
\underline{X} _i $ 
is a formal $\V$-scheme such that 
$\underline{X} _i \riso \underline{\X} \times _{\underline{\S} } \underline{S} _i$.
We have 
$\underrightarrow{\lim} _i \, X _i = (\underrightarrow{\lim} _i \, 
\underline{X} _i , \underleftarrow{\lim} _i \, 
M _{X _i} )$. 
Put $M :=
\underleftarrow{\lim} _i \, 
M _{X _i} $,
$\X:= \underrightarrow{\lim} _i \, X _i $. 
It remains to check that $M$ is fine log structure of $\underline{\X}$. 
This is checked in the step I). 

I )1) The canonical map 
$\theta \colon M \to \O _{\X}$,
canonically induced by the structural morphisms
$\theta _i \colon M _{X _i} \to \O _{X _i}$, is a log structure. 
Indeed, we compute $M^* (\U)= 
(\underleftarrow{\lim} _i \, 
M _{X _i} (U _i) )  ^*
=
\underleftarrow{\lim} _i \, 
(M _{X _i} (U _i) )  ^*
=
\underleftarrow{\lim} _i \, 
M ^* _{X _i} (U _i) 
=
\underleftarrow{\lim} _i \, 
\O ^* _{X _i} (U _i) 
=
\O ^* _{\X} (\U)$,
for any etale morphism $\U \to \X$.
Hence, $M ^* = \O ^* _{\X}$.
It remains to check that the 
morphism
$M ^* \to M \times _{\O _\X} \O _\X ^*$ is an isomorphism,
i.e. $M ^* (\U) \to M (\U) \times _{\O _\X (\U) } \O _\X (\U) ^*$ is an isomorphism.
Since
$M (\U) \times _{\O _\X (\U) } \O _\X (\U) ^* \subset M (\U)$, the injectivity is obvious. 
Let us check the surjectivity. 
Let $(a _i ) _{i\in\N} \in \underleftarrow{\lim} _i \, 
M _{X _i} (U _i)$ such that 
$(\theta _i (a _i ))_{i\in\N} \in \underleftarrow{\lim} _i \, 
\O ^* _{X _i} (U _i) $.
Since $M _{X _i}$ is a log structure of $X _i$, we get 
$a _i \in \O ^* _{X _i} (U _i) $, hence 
$(a _i ) _{i\in\N} \in M ^* (\U)$.

2) Let $\U \to \X$ be an etale morphism. Suppose $\U$ affine. 
We prove in this step that the canonical morphisms
$M ^\mathrm{gr} _{X _{i+1}} (U _{i+1})  
/
\O ^* _{X _{i+1}} (U _{i+1})
\to 
M ^\mathrm{gr} _{X _i} (U _i)  
/
\O ^* _{X _i} (U _i)$ 
and
$M _{X _{i+1}} (U _{i+1})  
/
\O ^* _{X _{i+1}} (U _{i+1})
\to 
M _{X _i} (U _i)  
/
\O ^* _{X _i} (U _i)$ 
are isomorphisms.
First, we remark that since $(\pi ^{i} \O _{X _{i+1}}) ^2=0$, then 
we have the canonical isomorphism of groups
$(1 + \pi ^{i} \O _{X _{i+1}}, \times) \riso
(\pi ^{i} \O _{X _{i+1}},+)$. 
Since $U _{i+1}$ is affine and $\pi ^{i} \O _{X _{i+1}}$ is quasi-coherent, this yields
$H ^{1} ( U _{i+1}, 1 + \pi ^{i} \O _{X _{i+1}})=0$.
Hence, we get the commutative diagram (we use in the proof multiplicative notation)
\begin{equation}
\label{inductivelimites-lem-diag1}
\xymatrix{
{}
&
{1}
&
{1}
\\
{1} 
\ar[r] ^-{}
& 
\ar[r] ^-{}
{\O ^* _{X _i} (U _i)}
\ar[u] ^-{} 
& 
{M ^\mathrm{gr} _{X _i} (U _i) } 
\ar[u] ^-{}
\ar[r] ^-{}
& 
{M ^\mathrm{gr} _{X _i} (U _i)  
/
\O ^* _{X _i} (U _i)} 
\ar[r] ^-{}
& 
{1} 
\\ 
{1} 
\ar[r] ^-{}
& 
\ar[r] ^-{}
{\O ^* _{X _{i+1}} (U _{i+1})} 
\ar[u] ^-{}
& 
{M ^\mathrm{gr} _{X _{i+1}} (U _{i+1}) } 
\ar[r] ^-{}
\ar[u] ^-{}
& 
{M ^\mathrm{gr} _{X _{i+1}} (U _{i+1})  
/
\O ^* _{X _{i+1}} (U _{i+1})} 
\ar[r] ^-{}
\ar[u] ^-{}
& 
{1} 
\\
{}
&
{1 + \pi ^{i} \O _{X _{i+1}} (U _{i+1})}
\ar@{=}[r] ^-{}
\ar[u] ^-{}
&
{1 + \pi ^{i} \O _{X _{i+1}} (U _{i+1})}
\ar[u] ^-{}
\\
{}
&
{1}
\ar[u] ^-{}
&
{1}
\ar[u] ^-{}
}
\end{equation}
whose two rows and two columns are exact. 
Hence, the morphism
$M ^\mathrm{gr} _{X _{i+1}} (U _{i+1})  
/
\O ^* _{X _{i+1}} (U _{i+1})
\to 
M ^\mathrm{gr} _{X _i} (U _i)  
/
\O ^* _{X _i} (U _i)$ 
is an isomorphism.
Since 
$M _{X _{i+1}} \to M _{X _{i}}$ is exact 
(see \cite[IV.2.1.2.4]{Ogus-Logbook}),
we have the surjective projection 
$M _{X _{i+1}} (U _{i+1})
=
M ^\mathrm{gr} _{X _{i+1}} (U _{i+1})
\times _{M ^\mathrm{gr} _{X _i} (U _i)}
M _{X _i} (U _i)
\to 
M _{X _i} (U _i)$.
Hence, 
$M _{X _{i+1}} (U _{i+1})  
/
\O ^* _{X _{i+1}} (U _{i+1})
\to 
M _{X _i} (U _i)  
/
\O ^* _{X _i} (U _i)$ 
is surjective. 
Since $M  _{X _i} (U _i)  $ and 
$M  _{X _{i+1}} (U _{i+1})  $
is integral, from 
\cite[I.1.2.1]{Ogus-Logbook}, the horizontal morphisms of the commutative diagram
\begin{equation}
\notag
\xymatrix{
{M  _{X _{i+1}} (U _{i+1})  
/
\O ^* _{X _{i+1}} (U _{i+1})}
\ar@{^{(}->}[r] ^-{} 
\ar[d] ^-{}
& 
{M ^\mathrm{gr} _{X _{i+1}} (U _{i+1})  
/
\O ^* _{X _{i+1}} (U _{i+1})} 
\ar[d] ^-{\sim}
\\ 
{M  _{X _i} (U _i)  
/
\O ^* _{X _i} (U _i)} 
\ar@{^{(}->}[r] ^-{} 
& 
{M ^\mathrm{gr} _{X _i} (U _i)  
/
\O ^* _{X _i} (U _i)} 
}
\end{equation}
are injective. 
Hence, 
$M _{X _{i+1}} (U _{i+1})  
/
\O ^* _{X _{i+1}} (U _{i+1})
\to 
M _{X _i} (U _i)  
/
\O ^* _{X _i} (U _i)$ 
is an isomorphism.

3) Using Mittag-Leffler condition, we get  the exact sequence 
$$1 \to 
\underleftarrow{\lim} _i \, 
\O ^* _{X _i} (U _i) 
\to 
\underleftarrow{\lim} _i \, 
M ^\mathrm{gr}_{X _i} (U _i) 
\to 
\underleftarrow{\lim} _i \, 
M _{X _i} ^\mathrm{gr} (U _i) /\O ^* _{X _i} (U _i) 
\to 1.$$
Since 
$\O _{\X} ^* (\U) 
= 
\underleftarrow{\lim} _i \, 
\O ^* _{X _i} (U _i) $,
using the step 2) we obtain 
$(\underleftarrow{\lim} _i \, 
M ^\mathrm{gr}_{X _i} (U _i)  )/\O _{\X} ^* (\U) 
\riso
M ^\mathrm{gr} _{X _0} (U _0) /\O ^* _{X _0} (U _0) $.
By considering the commutative diagram
\begin{equation}
\notag
\xymatrix{
{(\underleftarrow{\lim} _i \, 
M ^\mathrm{gr}_{X _i} (U _i)  )/\O _{\X} ^* (\U) }
\ar[r] ^-{\sim} 
& 
{M ^\mathrm{gr} _{X _0} (U _0) /\O ^* _{X _0} (U _0)  } 
\\ 
{M (\U)/\O _{\X} ^* (\U)
= 
(\underleftarrow{\lim} _i \, 
M  _{X _i} (U _i) )/\O _{\X} ^* (\U) } 
\ar@{^{(}->}[u] ^-{}
\ar[r] ^-{}
& 
{M_{X _0} (U _0) /\O ^* _{X _0} (U _0)  } 
\ar@{^{(}->}[u] ^-{}
}
\end{equation}
we get the injectivity of the map
$M (\U)/\O _{\X} ^* (\U) 
\to 
M_{X _0} (U _0) /\O ^* _{X _0} (U _0)  $.
Since 
the maps 
$M _{X _{i+1}} (U _{i+1})
\to 
M _{X _i} (U _i)$ are surjective (this is checked in the step 2), 
we get that 
$M (\U)\to M_{X _0} (U _0)$ is surjective and then 
so is $M (\U)/\O _{\X} ^* (\U) 
\to 
M_{X _0} (U _0) /\O ^* _{X _0} (U _0)  $.
Hence, the canonical morphism
$M (\U) /\O _{\X} ^* (\U) 
\to 
M _{X _0} (U _0) /\O ^* _{X _0} (U _0) $
is an isomorphism.
This yields
$\overline{M} = M / M ^* =
M _{X _0} / M ^* _{X _0}
=
\overline{M_{X _0}} $.

4) We compute that
the induced morphism 
$M ^* \to M \times _{M _{X _0}} M _{X _0} ^*$
is an isomorphism, 
i.e. that the morphism 
$M 
\to 
M _{X _0}$
is local. 
Hence, since 
$\overline{M} =
\overline{M_{X _0}} $ (see  step 3), 
since $M _{X _0}$ is fine and $M$ is integral, 
using Lemma \ref{int2fine-lem}, 
we get that $M$ is fine.

II) In this last step, we establish 
that $X _i \to \X \times _{\S} S _i$
is an isomorphism of fine log schemes.
Let $u _i \colon X _i \to \X$ be the canonical morphism. 
We already know that $\underline{X} _i \riso \underline{\X} \times _{\underline{\S} } \underline{S} _i$.
It remains to check that 
the morphism 
$u ^* _i M \to M _{X _i}$ 
is an isomorphism.
Since this is a morphism of fine log structures, then using 
\cite[I.4.1.2]{Ogus-Logbook},
this is equivalent to 
check the isomorphism 
$\overline{u ^* _i M} \riso \overline{M _{X _i}}$.
Following \cite[1.4.1]{Kato-logFontaine-Illusie},
$\overline{u ^* _i M} = \overline{M}$.
From the step I).3), we have $\overline{M}\riso\overline{M _{X _i}}$ and we are done.
\end{proof}

\begin{thm}
\label{inductivelimitesthm}
The functors $\X \mapsto (X _i) _{i\in\N}$
and 
$(X _i) _{i\in\N} \mapsto  \underrightarrow{\lim} _i \, X _i $ 
are quasi-inverse
equivalences of categories between the category of 
fine $\S$-log formal schemes to that of strict inductive systems of noetherian fine log schemes over $(S _i) _{i\in \N}$.
\end{thm}

\begin{proof}
This is a consequence of Propositions \ref{inductivelimites} and \ref{inductivelimites-lem}.
\end{proof}

\begin{lem}
\label{lem-stri-formal}
Let $f \colon \X \to \Y$ be a morphism of fine $\S$-log formal schemes.
Then $f$ is strict if and only if, for any $i\in \N$, 
$f _i$ is strict.
\end{lem}

\begin{proof}
If $f$ is strict then $f _i$, the base change of $f$ by $S _i \hookrightarrow \S$ is strict. 
Conversely, suppose that for any $i\in \N$, 
$f _i$ is strict.
Let $\ZZ$ be the fine $\S$-log formal scheme whose underlying fine formal $\S$-scheme is $\underline{\X}$ and whose log structure is $f ^{*} (M _\Y)$. Then $Z _i \to Y _i$ is strict and $\underline{Z _i} = \underline{X _i}$.
Hence, $Z _i= X _i$.
Using \ref{inductivelimites}, this yields that $\X = \ZZ$, i.e. $f$ is strict. 
\end{proof}

\subsection{Around log etaleness}
\begin{dfn}
\label{dfn-petaleformal}
Let $f\colon \X \to \Y$ be a morphism of fine $\S$-log formal schemes.
We say that $f$ is ``log étale'' (resp. ``log smooth, 
resp. ``formally log étale'', 
resp. ``log $p$-étale'', resp. ``formally log étale of level $m$'',
resp. ``log $p$-étale of level $m$'') 
if
for any integer $i\in \N$ the morphism
$f _i$ is
log étale (resp. log smooth, 
resp. fine formally log étale, 
resp. log $p$-étale, resp. formally log étale of level $m$,
resp. log $p$-étale of level $m$).
\end{dfn}

\begin{rem}
\label{comp-Shiho-formlogétale}
\begin{itemize}
\item We remark that our definition of log étaleness
was named by Shiho formal log étaleness (see \cite[2.2.2]{Shiho-log-isocI}).
We hope there will be no confusion.
\item It is also possible define some ``fine saturated'' definition similar to 
\ref{dfn-petaleformal} but we leave it to the interested reader. 
Since we only consider the ``fine'' case,  we have removed the word ``fine'' in the terminology of \ref{dfn-petaleformal}. 
We might also consider the notion of ``fine log relatively perfect'' morphism of fine $\S$-log formal schemes, but it seems useless for us. 
\end{itemize}

\end{rem}

\begin{empt}
The following diagram summarizes the relations between our definitions:
\begin{equation}
\label{linksbis}
\xymatrix{
{\text{log étale}}
\ar@{=>}[r] ^-{\ref{let-lpet}}
&
{\text{log $p$-étale}}
\ar@{=>}[d] ^-{\ref{logp-etal=logpetaleanym}}
\ar@{=}[r] ^-{\ref{logp-etal=logpetaleanym}}
&
{\text{$\forall m$, log $p$-étale of level $m$}}
\\
&
{\text{log $p$-étale of level $m$}}
\ar@{=>}[r] ^-{\ref{rem-etaleisnice}}
&
{\text{formally log étale of level $m$}}
\ar@{=>}[r] ^-{\ref{logp-etal=logpetaleanym}}
&
{\text{formally log étale}.}
}
\end{equation}
\end{empt}

\begin{prop}
Let $\Y$ be a  fine $\S$-log formal schemes.
Let $f _0\colon X _0 \to Y _0$ be a log smooth morphism of fine log $S _0$-schemes
such that $\underline{X _0}$ is affine. 
Then there exists a log smooth morphism of fine $\S$-log formal schemes of the form
$f\colon \X \to \Y$ whose reduction modulo $\pi$ is $f _0$.
We say that such morphism $f$ is a log smooth lifting of $f _0$. 
\end{prop}

\begin{proof}
From \cite[3.14.(1)]{Kato-logFontaine-Illusie}, 
there exists a unique up to isomorphism log smooth morphism of fine log $S_i$-schemes
$f _i \colon X _i \to Y _i$ endowed with an isomorphism
$X _0 \riso X _i \times _{Y _i} Y _0$.
Put $\Y := \underrightarrow{\lim} _i \, Y _i$.
Let $f \colon \Y \to \X$ be the induced morphism. 
Following Theorem
\ref{inductivelimitesthm}, 
$\Y$ is a fine $\S$-log formal schemes. 
By construction, $f$ is log smooth since $f _i$ is log smooth for any $i\in\N$. 
\end{proof}

\begin{prop}
Let $f\colon \X \to \Y$ be a morphism of fine $\S$-log formal schemes.
The morphism $f$ is
log étale 
(resp. log $p$-étale, resp. formally log étale of level $m$,
resp. log $p$-étale of level $m$)
if and only if 
$f$ is formally log étale and 
$f _0$ is
log étale 
(resp. log $p$-étale, resp. formally log étale of level $m$,
resp. log $p$-étale of level $m$).

\end{prop}

\begin{proof}
If $f _0$ is log étale then $\underline{f_0} $ is of finite type. This yields that
$\underline{f _i} $ is of finite type, which proves the non respective case. 
The respective cases are consequences of 
\ref{p-etale-modp} or \ref{p-etale-modplevelm}.
\end{proof}

\begin{empt}
\label{flatnessformalogmooth}
Let $f \colon \X \to \Y$ be a morphism of fine $\S$-log formal schemes.
We set 
$\Omega _{\X/\Y} ^1
:=
\underleftarrow{\lim} _i \, 
\Omega _{X _i/Y _i} ^1$.
When $f$ is log smooth, 
then from \cite[3.10]{Kato-logFontaine-Illusie}
the $\O _{\X}$-module $\Omega _{\X/\S} ^1$ is locally free of finite type. 
When $f$ log smooth and $\Y =\S$, 
then $f$ is flat.
Indeed, in that case, $f _i \colon X _i \to S _i$ is integral (use \cite[4.3]{Kato-logFontaine-Illusie})  and smooth
 and then 
flat (see \cite[4.4]{Kato-logFontaine-Illusie}).
Since $f$ is of finite type, then $f$ is flat.

\end{empt}

\begin{lem}
\label{lem-log-etale-modpi}
Let $f \colon \X \to \Y$, $g \colon \Y \to \ZZ$ be two morphisms of fine $\S$-log formal schemes
such that $\O _{\X}$ has no $p$-torsion, the structural morphism $g \circ f$ is log smooth
and $f _0 \colon X _0 \to Y _0 $ is log étale. 
Then $f$ is log étale. 
\end{lem}

\begin{proof}
We construct by $p$-adic completion the morphism 
$\phi \colon 
f ^* \Omega _{\Y/\ZZ} ^1
\to 
\Omega _{\X/\ZZ} ^1$,
where we put
$f ^* \Omega _{\Y/\ZZ} ^1:= 
\underleftarrow{\lim} _i \, 
f _i ^* \Omega _{Y _i/Z _i} ^1$.
Since $g \circ f\colon \X \to \ZZ$ is log smooth, 
the $\O _{\X}$-module $\Omega _{\X/\ZZ} ^1$ is locally free of finite type
(see \ref{flatnessformalogmooth}). In particular, 
$\Omega _{\X/\ZZ} ^1$ has no $p$-torsion.
The reduction of $\phi$ modulo $\pi$ is canonically isomorphic to 
$f _0 ^* \Omega _{Y _0/Z _0} ^1
\to 
\Omega _{X_0/Z _0} ^1$.
Since $f _0$ is log étale, 
this latter homomorphism is an isomorphism.
Since $\Omega _{\X/\ZZ} ^1$ has no $p$-torsion, 
this yields 
that $\phi$ is an isomorphism (e.g. use Lemma \cite[2.2.15]{caro_courbe-nouveau}).
This implies that the canonical morphism
$f _i ^* \Omega _{Y _i/Z _i} ^1
\to 
\Omega _{X_i/Z _i} ^1$
is an isomorphism.
Since $X _i \to Z _i$ is log smooth, 
from \cite[3.12]{Kato-logFontaine-Illusie}, 
we conclude that $f _i$ is log-étale.

\end{proof}

\begin{prop}
\label{prop-smoothiswsmooth}
Let $f \colon \X \to \Y$ be a morphism of fine $\S$-log formal schemes such that $\O _{\X}$ has no $p$-torsion.
The morphism $f$ is log smooth if and only if, 
étale locally on $\X$ there exists a log étale $\Y$-morphism 
of the form 
$\X \to \Y \times _{\V} \mathfrak{A} _{\N ^{r}}$.
\end{prop}

\begin{proof}
Suppose $f$ is log smooth.
Since $f _0$ is log smooth, when can suppose 
there exists a morphism 
$X _0 \to A _{\N ^{r}}$ such that the induced 
$Y _0$-morphism
$X _0 \to Y _0 \times A _{\N ^{r}}$
is log-étale. 
Using \ref{inductivelimites-presurj}, 
we can suppose that
$X _0 \to A _{\N ^{r}}$
has the lifting of the form 
$\X \to  \mathfrak{A} _{\N ^{r}}$.
We get the $\Y$-morphism
$\X \to \Y \times _{\V} \mathfrak{A} _{\N ^{r}}$.
We conclude by applying Lemma \ref{lem-log-etale-modpi}
that this latter morphism is log-étale.
\end{proof}

\subsection{Sheaf of differential operators over weakly log smooth $\S$-log formal scheme}

\begin{dfn}
As in \ref{dfnC}
we define the category $\mathfrak{C}$
of $\S$-immersions of fine $\S$-log formal schemes.
For any integer $n$,  we denote by
$\mathscr{C} _{n}$ the full subcategory
of $\mathscr{C}$ whose objects are exact closed immersions 
of order $n$.
\end{dfn}

\begin{lem}
\label{PDenvelopeformal}
The inclusion functor 
 $For  _n \colon \mathfrak{C} _n \to \mathfrak{C}$
 has a right adjoint functor which we will denote by
$P ^{n }
\colon \mathfrak{C} \to \mathfrak{C}  _n $.
Let  $u\colon \ZZ \hookrightarrow \X$ be an object of
$\mathfrak{C}$.
Then $\ZZ$ is also the source of $P ^{n} (u) $.
\end{lem}

\begin{proof}
Let  $u\colon \ZZ \hookrightarrow \X$ be an object of
$\mathfrak{C}$.
Since $u _i \colon Z _i \hookrightarrow X _i$ is an object of 
$\mathscr{C} $, 
from \ref{PDenvelope},
we get the object
$P ^{n} (u _i)\colon Z _i \hookrightarrow P ^n (u _i)$ of 
$\mathscr{C} ^n$ such that 
$P ^{n} (u _i) \to X _i$ is affine and $P ^{n} (u _i)$ is noetherian.
Hence, using Theorem \ref{inductivelimitesthm}, 
we get
that 
$ \underrightarrow{\lim} _i \, P ^{n} (u _i)$ satisfies the universal property of 
$P ^{n} (u)$. 
\end{proof}

\begin{dfn}
\label{forlogbasisformal}
Let $f\colon \X \to \Y$ be a morphism of fine $\S$-log formal schemes.

\begin{enumerate}
\item 
We say that a finite set $(b  _{\lambda}) _{\lambda =1,\dots, r}$ of elements of $\Gamma ( \X, M _\X)$ is a 
`` formal log basis of $f$''
if the induced  $\Y$-morphism
$\X \to \Y \times _{\V} \mathfrak{A} _{\N ^{r}}$ is formally log étale
(concerning $\mathfrak{A} _{\N ^{r}}$, see the notation of \ref{formalcharts}).

\item We say that $f$ is ``weakly log smooth'' if, étale locally on $\X$,
$f$ has formal log bases.
Notice that this notion of weak log smoothness is étale local on $\Y$.
When $\Y= \S$, we say that $\X$ is a ``weakly log smooth $\S$-log formal scheme''
(following the terminology, the log structure of such $\X$ is understood to be fine).

\end{enumerate}

\end{dfn}

\begin{rem}
\label{rem-smoothiswsmooth}
Following Proposition \ref{prop-smoothiswsmooth},
a log smooth morphism is weakly log smooth which justifies the terminology.
\end{rem}

\begin{empt}
[$n$th infinitesimal neighborhood]
Let $\X$ be a weakly log smooth $\S$-log formal scheme.
Let $\Delta _{\X/\S}\colon \X \hookrightarrow \X \times _{\S} \X$ be the diagonal immersion.
Since $\Delta _{\X/\S}$ is not necessarily an object of $\mathfrak{C}$ 
(because $\X \times _\S \X$ is not noetherian in general), 
we can not use \ref{PDenvelopeformal} and we can not put 
$\Delta ^{n} _{\X/\S}:= P ^{n } ( \Delta _{\X/\S}) $. 
But this is possible to define $\Delta ^{n} _{\X/\S}$ by taking inductive limits as follows. 
From \ref{prebasecgt-flat-env}, we have
$\Delta _{X _i/S _i} ^n 
= \Delta _{X _{i+1}/S _{i+1}} ^n \times _{S _{i+1}} S _{i}$.
From \ref{PDenvelope}, since 
$\underline{\Delta} _{X _i/S _i} ^n$ are noetherian schemes, 
using Theorem \ref{inductivelimitesthm}, 
we get the fine $\S$-log formal schemes
$\Delta _{\X /\S} ^n$ by putting 
$\Delta _{\X /\S} ^n:= \underrightarrow{\lim} _i \, \Delta _{X _i/S _i} ^n$.
Taking the inductive limites to 
the strict morphisms of fine log schemes
$p ^n _0 \colon 
\Delta _{X _i/S _i} ^n
\to 
X _i$ 
(resp. $p ^n _1 \colon 
\Delta _{X _i/S _i} ^n
\to 
X _i$), 
using Lemma \ref{lem-stri-formal}
we get  
the strict morphism of fine log formal $\V$-schemes 
$p ^n _0\colon \Delta _{\X /\S} ^n \to \X$
(resp. $p ^n _1\colon \Delta _{\X /\S} ^n \to \X$).
Using the remark \ref{p1p0finite},
we check that the underlying morphism of formal $\V$-schemes of 
$p ^n _0\colon \Delta _{\X /\S} ^n \to \X$
and $p ^n _1\colon \Delta _{\X /\S} ^n \to \X$ are finite (more precisely, we can check 
the local description \ref{loc-desc-Pnfformal}).
Hence, we denote by 
$\PP ^{n} _{\X/\S}$ the coherent $\O _\X$-algebra such that 
$\Spf \PP ^{n} _{\X/\S} =  \underline{\Delta} ^n _{\X/\S}  $.

If $a \in M _\X$, we denote by
$\mu _{(m)}  (a)$ the unique section of
$\ker ( \O _{\Delta ^n _{\X/\S}} ^{*} \to  \O _{\X} ^{*} )$
such that we get in $M ^n _{\X/\S, (m)}$ the equality
$p _1 ^{n*} (a)= p _0 ^{n*} (a) \mu ^n  (a)$ (see \ref{Kerexactclosedimmer}).
We get 
$\mu  ^{n} \colon M  _\X \to \ker ( \O _{\Delta ^{n}  _{\X/\S}} ^{*} \to  \O _{\X} ^{*} )$
given by 
$a \mapsto \mu ^n  (a)$.

\end{empt}

\begin{prop}
[Local description of $\PP  ^n _{\X/\S} $]
\label{nota-etawtmformal}
Let $(a  _{\lambda}) _{\lambda =1,\dots, r}$ be a formal  log basis  of $f$. 
Put 
$\eta _{\lambda,n} := \mu ^n  ( a _{\lambda}) -1$.
We have  the following  isomorphism of $\O _{\X}$-algebras: 
   \begin{align}
   \notag
   \O _{\X} [T _1 ,\dots ,T _r] _n
   &\riso \PP  ^n  _{\X/\S}  \\
   \label{loc-desc-Pnfformal}
   T _\lambda &\mapsto
   \eta _{\lambda,n}.
   \end{align}

\end{prop}

\begin{proof}
This is a consequence of \ref{nota-etawtm}. 
\end{proof}

\begin{dfn}
\label{sheafdiffoperformal}
The sheaf of differential operators of order $\leq n$ of $f$
is defined by putting
$\D  _{\X/\S, n}:= \H om _{\O _\X} ( p _{0 *} ^{n} \PP ^{n} _{\X/\S} , \O _\X)$.
The sheaf of differential operators of $f$
is defined by putting
$\D  _{\X/\S}:= \cup _{n \in \N }\D  _{\X/\S, n}$.

Let $P \in \D  _{\X/\S, n}$, $P'  \in \D  _{\X/\S, n'}$.
We define the product
$P P' \in \D  _{\X/\S, n+n'}$
to be the composition
\begin{equation}
\label{dfn-prodformal}
P P' \colon \PP ^{n +n'} _{\X/\S}
\overset{\delta ^{n,n'}}{\longrightarrow}
\PP ^{n} _{\X/\S} \otimes _{\O _\X} \PP ^{n'} _{\X/\S}
\overset{\mathrm{Id}\otimes P'}{\longrightarrow}
\PP ^{n} _{\X/\S}
\overset{P}{\longrightarrow}
\O _{\X}.
\end{equation}
Similarly to \ref{ring-diffop}, we check that 
the sheaf $\D  _{\X/\S} $ is a sheaf of rings with the product as defined in \ref{dfn-prodformal}
\end{dfn}

\subsection{Sheaf of differential operators of level $m$ over weakly log smooth of level $m$ fine $\S$-log formal schemes}

Let $m\geq 0$ be an integer. 
The principal ideal $(p)$ of $\V$ is endowed with a canonical $m$-PD-structure, 
which we will denote by  $\gamma _\emptyset$.

\begin{dfn}
\label{dfnCmformal}
As in \ref{dfnCm},
we define the categories
$\mathfrak{C} _n ^{(m)}$
whose objects are pairs
$(u, \delta)$ where $u$ is an exact closed $\S$-immersion
of fine log $\S$-schemes and $\delta$ is an $m$-PD-structure on the ideal $\I$
defining $u$
(which is compatible with $\gamma _\emptyset$)
and such that $\I ^{\{n+1\} _{(m)}}=0$ and whose morphisms
$(u', \delta ') \to (u, \delta)$ are morphisms $u' \to u$ of $\mathfrak{C} $ 
which are compatible with the $m$-PD-structures $\delta$ and
$\delta '$.
\end{dfn}

\begin{prop}
\begin{enumerate}
\item The canonical functor  
$\mathfrak{C} _n ^{(m)} \to \mathfrak{C}$ has a right adjoint, which we
will denote by 
$P ^n _{(m)}
\colon \mathfrak{C} \to \mathfrak{C} _n ^{(m)} $.

\item Let $u$ be an object of $\mathfrak{C}$.
The source of
$P ^n _{(m)} (u)$ is the source of $u$.

\end{enumerate}
\end{prop}

\begin{proof}
The first assertion is a consequence of \ref{inductivelimitesthm} and \ref{mPDenvelope} 
(we need in particular the \ref{mPDenvelope}.4).
Since $\gamma _\emptyset$ extends to any $\S$-log formal schemes 
(because the ideal of the $m$-PD-structure $\gamma _\emptyset$ is locally principal: see  \cite[1.3.2.c)]{Be1}),
we get the second assertion.
\end{proof}

\begin{empt}
Let $u$ be an object of $\mathfrak{C}$. We call $P ^n _{(m)}  (u)$ the
$m$-PD-envelope compatible 
of order $n$ of $u$.
We sometimes denote abusively by 
$P ^n _{(m)} (u)$ the target of the arrow
$P ^n _{(m)} (u)$.

\end{empt}

\begin{dfn}
Let $f\colon \X \to \Y$ be a morphism of fine $\S$-log formal schemes.
\begin{enumerate}
\item We say that a finite set $(b  _{\lambda}) _{\lambda =1,\dots, r}$ of elements of $\Gamma ( \X, M _\X)$ is 
a ``log $p$-basis of $f$'' (resp. ``formal log basis of level $m$ of $f$'', resp. 
``log $p$-basis of level $m$ of $f$'')
if the induced  $\Y$-morphism
$\X \to \Y \times _{\V} \mathfrak{A} _{\N ^{r}}$ is log $p$-étale
(resp. formally log étale of level $m$, resp. log $p$-étale of level $m$).

\item We say that $f$ is ``log $p$-smooth'' 
(resp. ``weakly log smooth of level $m$'', 
resp. ``log $p$-smooth of level $m$'')
 if, étale locally on $\X$,
$f$ has log $p$-bases
(resp. formal log bases of level $m$, 
resp. log $p$-bases of level $m$).
When $\Y = \S$, we say that
$\X$ is  a log $p$-smooth log-formal $\S$-scheme
(resp. a weakly log smooth of level $m$ log-formal $\S$-scheme, 
resp. a log $p$-smooth of level $m$ log-formal $\S$-scheme).
Remark that the log structure of such $\X$ is always fine following our terminology. 

\end{enumerate}

\end{dfn}

\begin{empt}
\label{linkster}
Using \ref{linksbis} and \ref{prop-smoothiswsmooth},
we get the following diagram 
summarizing the relations between our definitions:
\small
\begin{equation}
\notag
\xymatrix{
{\text{log smooth}}
\ar@{=>}[r] ^-{}
&
{\text{log $p$-smooth}}
\ar@{=>}[r] ^-{}
&
{\text{log $p$-smooth of level $m$}}
\ar@{=>}[r] ^-{}
&
{\text{weakly log smooth of level $m$}}
\ar@{=>}[r] ^-{}
&
{\text{weakly log smooth}}
}
\end{equation}
\normalsize
\end{empt}

\begin{empt}
Let $\X$ be a weakly log smooth of level $m$ log-formal $\S$-scheme.
Using  \ref{p1p0finite(m)}, we check the underlying scheme of $\Delta _{X _i/S _i, (m)} ^n$ is noetherian.
Moreover, from the local description \ref{loc-desc-Pn}, we get 
$\Delta _{X _i/S _i, (m)} ^n 
\riso \Delta _{X _{i+1}/S _{i+1}, (m)} ^n \times _{S _{i+1}} S _{i}$ (recall 
also that $p ^n _0\colon \Delta _{X _i/S _i, (m)} ^n \to X _i$ is strict).
Using Theorem \ref{inductivelimitesthm}, 
we get the fine $\S$-log formal schemes
$\Delta _{\X /\S, (m)} ^n$ by putting 
$\Delta _{\X /\S, (m)} ^n:= \underrightarrow{\lim} _i \, \Delta _{X _i/S _i, (m)} ^n$.
Let
$p  ^{n} _1,\, p ^{n} _0 \colon \Delta  ^{n} _{\X/\S, (m)} \to \X$ be the morphisms induced
respectively by 
$p  ^{n} _1,\, p ^n _0 \colon \Delta _{X _i/S _i, (m)} ^n \to X _i$.
From \ref{notaDeltanetc}, \ref{lem-stri-formal} and \ref{p1p0finite(m)}, the morphisms
$p ^{n} _1,\, p ^{n} _0 \colon \Delta  ^n _{\X/\S, (m)} \to \X$
are strict and finite (more precisely concerning the finiteness, we have the local description 
\ref{loc-desc-Pn-form}).

We denote by $M ^{n} _{\X/\S, (m)}$ the log structure of
$\Delta ^{n} _{\X/\S, (m)}$.
We denote by $\PP ^{n} _{\X/\S, (m)}$ the coherent $\O _\X$-algebra corresponding to the underlying formal $\V$-scheme of
$\Delta ^{n} _{\X/\S, (m)} $.
Hence, $\Delta ^{n} _{\X/\S, (m)}$ is an exact closed immersion of the form
$\Delta ^{n} _{\X/\S, (m)} \colon \X \hookrightarrow (\Spf \PP ^{n} _{\X/\S, (m)} , M ^{n} _{\X/\S, (m)})$.
We sometimes denote abusively by 
$\Delta ^{n} _{\X/\S, (m)}$ 
the target of the arrow
$\Delta ^{n} _{\X/\S, (m)}$.

As in paragraph \ref{sheafdiffoper(m)}, we can define
$\D ^{(m)} _{\X/\S}$, the sheaf of differential operator on $\X$ of level $m$.
\end{empt}

\begin{empt}
[Local description]
\label{nota-eta-formal}
Suppose in this paragraph that $\X \to \S$ is endowed with 
a formal log basis of level $m$ 
$(b  _{\lambda}) _{\lambda =1,\dots, r}$ of $f$.
Put
$\eta _{\lambda (m)} := \mu _{(m)} ^{n} ( b _{\lambda}) -1$
(or simply $\eta  _{\lambda}$),
where
$\mu _{(m)} ^{n} (a)$ is the unique section of
$\ker ( \O _{\Delta ^{n} _{\X/\S, (m)}} ^{*} \to  \O _{\X} ^{*} )$
such that we get in $M ^{n} _{\X/\S, (m)}$ the equality
$p _1 ^{n*} (a)= p _0 ^{n*} (a) \mu ^{n} _{(m)} (a)$.
Taking the limits to \ref{nota-eta},
we get the isomorphism of $m$-PD-$\O _{\X}$-algebras
   \begin{align}
   \notag
   \O _{\X} <T _1 ,\dots ,T _r> _{(m),n}
   &\riso \PP ^{n} _{\X/\S, (m)}  \\
   \label{loc-desc-Pn-form}
   T _\lambda &\mapsto
   \eta _{\lambda, (m)} ,
   \end{align}
where the first term is defined as in \ref{ntn-mPDtruncated}.
In particular, the elements
$\{\underline{\eta} ^{\{\underline{k} \} _{(m)}} \} _{| \underline{k}| \leq n}$
form an $\O _{\X}$-basis of
$\PP ^{n} _{\X/\S, (m)}$.
The corresponding dual basis of
$\D ^{(m)} _{\X/\S, n}$
will be denoted by
$\{ \underline{\partial} ^{< \underline{k}> _{(m)}} \} _{| \underline{k}| \leq n}$.
Let $\epsilon _1, \dots, \epsilon _r $  be the canonical basis of $\N ^{r}$,
 i.e. the coordinates of $\epsilon _i $ are $0$ except for the $i$th term which is $1$.
 We put $\partial _i := \smash{\underline{\partial} }^{< \epsilon _i> _{(m)}}$.
We can define the logarithmic transposition as in \ref{ntn-transp}
and we can check that the properties analogous to the subsection 
\ref{subsection2.4} are still satisfied in the formal context.
\end{empt}

We finish the subsection by the formal version of the definition appearing in \ref{finite-p-basis}. 
\begin{dfn}
\label{finite-p-basisformal}
Let $f\colon \X \to \Y$ be a morphism of fine $\S$-log formal schemes.
\begin{enumerate}
\item We say that a finite set $(t  _{\lambda}) _{\lambda =1,\dots, r}$ of elements of $\Gamma ( \X, \O _\X)$ are 
``log $p$-étale coordinates''
(resp. ``formal log étale coordinates'',
resp. ``formal log étale coordinates of level $m$'',
resp. ``log $p$-étale coordinates of level $m$''),
if the corresponding $\Y$-morphism
$\X \to \Y \times _{\S} \widehat{\A} _\S ^{r}$, where $\widehat{\A} _\S ^{r}$ is the 
$p$-adic completion of the $r$th affine space over $\V$ endowed 
with the trivial logarithmic structure,
is log $p$-étale
(resp. formally log étale, resp. formally log étale of level $m$, 
log $p$-étale of level $m$).

When $f$ is strict
we remove ``log'' in the terminology, e.g. we get the notion of ``$p$-étale coordinates''.
\item We say that $f$ is ``$p$-smooth'' 
(resp. ``weakly smooth'',
resp. ``weakly smooth of level $m$'',
resp. ``$p$-smooth of level $m$''),
if $f$ is strict and if, étale locally on $\X$,
$\underline{f}$ has 
$p$-étale coordinates''
(resp. ``formal étale coordinates'',
resp. ``formal étale coordinates of level $m$'',
resp. ``$p$-étale coordinates of level $m$'').
Notice that these notions 
are étale local on $Y$.
\end{enumerate}
\end{dfn}

\subsection{Sheaf of differential operators of finite level over log $p$-smooth $\S$-log formal schemes}

\begin{empt}
Let $\X$ be a log $p$-smooth $\S$-log formal schemes.
We denote by
$\widehat{\D} ^{(m)} _{\X/\S} $ the $p$-adic completion of
$\D ^{(m)} _{\X/\S}$.
As in the paragraph \ref{chgtbasis},
we check that
$\D ^{(m)} _{X _i /S _i}   \riso  \D ^{(m)} _{\X/\S} \otimes _{\V } \V / \pi ^{i+1}$.
Hence $\widehat{\D} ^{(m)} _{\X/\S}  \riso
\underleftarrow{\lim} _i \,
\D ^{(m)} _{X _i /S _i}$.
\end{empt}

\begin{empt}
Let $\X$ be a log $p$-smooth $\S$-log-formal schemes.
We put
$\D ^{\dag} _{\X/\S} :=
\underrightarrow{\lim} _{m}\,
\widehat{\D} ^{(m)} _{\X/\S} $.
This is the ``sheaf of differential operators of finite level of $\X/\S$''.
When $\X \to \S$ is endowed with a log $p$-basis $(b  _\lambda) _{\lambda =1,\dots, n}$,
we get the usual description
(\cite[2.4.4]{Be1}):
an operator $P$ of $\Gamma (\X, \D ^{\dag} _{\X/\S}) $ is of the form
$$ P = \sum _{\underline{k}\in \N ^{n}} a _{\underline{k}} \underline{\partial} ^{[\underline{k}]}$$
where $a _{\underline{k}} \in \Gamma (\X, \O _{\X})$ satisfy the condition :
there exist some constants $c, \eta \in \R$, with $\eta <1$, such that for any $\underline{k}\in \N ^{n}$ we have
$$\parallel a _{\underline{k}} \parallel \leq c \, \eta ^{| \underline{k}|},$$
where $\parallel \ \parallel$ is the $p$-adic norm i.e. whose basis of open neighbourhoods of $0$ is given by
$(p ^{n} \O _\X) _{n\in \N}$.
\end{empt}

\begin{empt}
Let $\X$ be a log $p$-smooth $\S$-log-formal scheme.
As in \cite[2.2.5]{Be1}, we check that if $\U$ is a Zariski open set of $\X$ having a log $p$-basis then
$\Gamma ( U _i, \D ^{(m)} _{X _i /S _i})$ is right and left noetherian.
As in  \cite[3.1.2]{Be1},  we check that the sheaf $\D ^{(m)} _{X _i /S _i}$ is coherent on the right and on the left.
As in  \cite[3.3.4]{Be1}, this yields that
$\widehat{\D} ^{(m)} _{\X /\S} $ is coherent on the right and on the left.
As in  \cite[3.4.2]{Be1}, this implies that $\widehat{\D} ^{(m)} _{\X /\S, \Q} $ is
coherent on the right and on the left.
By following the proof of  \cite[3.5.3]{Be1},
we prove that the extension
$\widehat{\D} ^{(m)} _{\X /\S, \Q}  \to \widehat{\D} ^{(m+1)} _{\X /\S, \Q} $ is flat on the right and on the left.
Hence, taking the inductive limits, we obtain the coherence on the right and on the left of
$\D ^{\dag} _{\X/\S} $.
Similarly, we have theorem of type $B$ as in \cite[3]{Be1}:
if $\underline{\X}$ is affine, for any integer $q \geq 1$ we have the vanishing
$H ^{q} (\X,  \D ^{(m)} _{X _i /S _i})= 0$,
$H ^{q} (\X,  \widehat{\D} ^{(m)} _{\X /\S} )= 0$,
$H ^{q} (\X,  \widehat{\D} ^{(m)} _{\X /\S,\Q} )= 0$,
$H ^{q} (\X,  \D ^{\dag} _{\X/\S} )= 0$.

\end{empt}

\begin{empt}
[Frobenius descent]
Let $\X$ be a $p$-smooth $\S$-log-formal scheme (in particular $\X /\S$ is strict). 
Using \ref{levelm-level0}, 
we remark that the results concerning Frobenius of \cite{Be2} (e.g. the fact that the functor $F _X ^*$ induces an equivalence 
of categories between coherent $\widehat{\D} ^{(m)} _{\X /\S}$-modules and 
coherent $\widehat{\D} ^{(m+1)} _{\X /\S}$-modules)
are still valid (indeed, we have only to copy 
word by word the arguments or computations) in the context of a $p$-smooth $\S$-log-formal scheme. 
Recall that when $\X/\S$ is not strict, 
these equivalence of categories involving Frobenius are false in general  (see \cite{these_montagnon}).

As in \cite[4.4]{Be2}, since $\D ^{(0)} _{\X/\S}$ has finite cohomological dimension (the proof is standard), 
this implies (by completion and Frobenius descent) that so are $\widehat{\D} ^{(m)} _{\X /\S}$ and 
then $\D ^\dag _{\X /\S, \Q}$ by passing to the limit.
Since $\D ^\dag _{\X /\S, \Q}$ is also coherent,
this yields that the derived category of perfect complexes of left 
$\D ^\dag _{\X /\S, \Q}$-modules and
that of bounded coherent complexes 
of left $\D ^\dag _{\X /\S, \Q}$-modules are the same.
\end{empt}

\begin{empt}
[Overconvergent isocrystals]
\label{isoc}
Let $\X$ be a log $p$-smooth $\S$-log-formal schemes 
and $Z$ be a Cartier divisor of $\underline{X _0} $.
As in \cite[4.2.3]{Be1} and with its notation, 
the commutative $\O _{X _i}$-algebra $\B ^{(m)} _{X _i} (Z) $ can be endowed with a (canonical) compatible 
structure of left $\D ^{(m)} _{X _i/S _i}$-module (see the definition \ref{dfn-algcomp-mod})
such that $\B ^{(m)} _{X _i} (Z) \to \B ^{(m+1)} _{X _i} (Z)$ is $\D ^{(m)} _{X _i/S _i}$-linear.
We get a structure of 
$\widehat{\D} ^{(m)} _{\X/\S}  $-module on 
$\widehat{\B} ^{(m)} _{\X} (Z)= 
\underleftarrow{\lim} _i \,
\B ^{(m)} _{X _i} (Z)$.
From 
\ref{prop-algBotimes}, we get 
the $\O _{\X}$-algebra
$\widehat{\B} ^{(m)} _{\X} (Z) \widehat{\otimes}\widehat{\D} ^{(m)} _{\X/\S}  $
such that the canonical map 
$\widehat{\D} ^{(m)} _{\X/\S}  
\to 
\widehat{\B} ^{(m)} _{\X} (Z) \widehat{\otimes}\widehat{\D} ^{(m)} _{\X/\S}  $
is a morphism of $\O _{\X}$-algebras.
We set 
$\O _{\X} (\hdag Z) 
: = 
\underrightarrow{\lim} \, _m 
\widehat{\B} ^{(m)} _{\X} (Z)$
and
$\D ^{\dag} _{\X /\S} (\hdag Z) 
: = 
\underrightarrow{\lim} \, _m 
\widehat{\B} ^{(m)} _{\X} (Z) \widehat{\otimes}\widehat{\D} ^{(m)} _{\X/\S}  $.
We define an isocrystal on $\X$ overconvergent along $Z$ 
to be a coherent 
$\D ^{\dag} _{\X /\S} (\hdag Z)  _{\Q}$-module 
which is also 
$\O _{\X} (\hdag Z) _{\Q}$-coherent (for the structure induced 
by the canonical morphism $\O _{\X} (\hdag Z) _{\Q} \to 
\D ^{\dag} _{\X /\S} (\hdag Z)  _{\Q}$).

\end{empt}

\subsection{Structure of right $\D ^{\dag} _{\X/\S} $-module on $\omega _{\X/\S}$}

\begin{prop}
[Structure of right $\D ^{(0)} _{X/S} $-module on $\omega _{X/S}$]
\label{rightD-mod}
Let $S$ be a fine log scheme  over $\Z / p ^{i+1} \Z$ and let $(I _S,J _S, \gamma)$  be a quasi-coherent $m$-PD-ideal of $\O _S$.
Let $f\colon   X \to     S$ be a weakly log smooth of level $m$ compatible with $\gamma$ morphism 
of fine log-schemes such that $\gamma$ extends to $X$.

We have a canonical structure of right $\D ^{(0)} _{X /S }$-module on
$\omega _{X/S}$ (see notation \ref{nota-omega}).
Locally, this structure is characterized by the following description.
Suppose that $X \to S$ is endowed with a formal log basis $(b  _i) _{i=1,\dots, n}$ of level $m$ compatible with $\gamma$.
Let $dlog\, b_i$ denotes the image of $\eta _i$ in $\Gamma(X,\Omega ^{1} _{X/S})$.
The action of $P \in \D ^{(0)} _{X/S} $ on the section $a\, dlog\, b_1 \wedge \cdots \wedge dlog\, b_n$,
where
$a$ is section of $\O _X$
is given by the formula
\begin{equation}
\label{rightD-modloc}
 (a\  dlog\, b_1 \wedge \cdots \wedge dlog\, b_n) \cdot P = \widetilde{P} (a) \ dlog\, b_1 \wedge \cdots \wedge dlog\, b_n .
 \end{equation}

\end{prop}

 \begin{proof}
0) It is sufficient to check the independence of the formula \ref{rightD-modloc} with respect to the chosen 
formal log basis  of level $m$ compatible with $\gamma$.
Suppose that $X \to S$ is endowed with two formal log bases $(b  _{i}) _{i =1,\dots, n}$
and $(b  '_{i}) _{i =1,\dots, n}$  of level $m$ compatible with $\gamma$.

1) Let $A = (a _{i j}) \in M _n ( \O _{X}) $ (resp. $A '= (a '_{i j}) \in M _n ( \O _{X}) $) be the matrix such that
$\left (  \begin{smallmatrix}
\partial _1
\\
\vdots
\\
\partial _n
\end{smallmatrix}
\right )
= A
\left (
\begin{smallmatrix}
\partial _1 ^{\prime}
\\
\vdots
\\
\partial _n ^{\prime}
\end{smallmatrix}
\right ) $
(resp. $\left (
\begin{smallmatrix}
\partial _1 ^{\prime}
\\
\vdots
\\
\partial _n ^{\prime}
\end{smallmatrix}
\right )
= A ^{\prime}
\left (  \begin{smallmatrix}
\partial _1
\\
\vdots
\\
\partial _n
\end{smallmatrix}
\right ) $).
Hence, we get $A ' = A ^{-1}$,
$\left (
\begin{smallmatrix}
dlog\, b_1 ^{\prime}
\\
\vdots
\\
dlog\, b_n ^{\prime}
\end{smallmatrix}
\right )  ={} ^{t} A\left (
\begin{smallmatrix}
dlog\, b_1
\\
\vdots
\\
dlog\, b_n
\end{smallmatrix}
\right ) $
and then
$dlog\, b_1 \wedge \cdots \wedge dlog\, b_n
=
| A '| \, dlog\, b'_1 \wedge \cdots \wedge dlog\, b'_n $ .
We compute
$\partial  _{i '} \partial  _{i} =
\sum _{j = 1} ^{n} \partial  _{i '} a _{ij} \partial ^{\prime} _j
=
\sum _{j = 1} ^{n} a _{ij} \partial  _{i '}  \partial ^{\prime } _j
+
\sum _{j = 1} ^{n} \partial  _{i '} (a _{ij} )\partial ^{\prime} _j
=
\sum _{j , j'= 1} ^{n} a _{ij} a _{i'j'} \partial  ^{\prime }_{j '}  \partial ^{\prime} _j
+
\sum _{j = 1} ^{n} \partial  _{i '} (a _{ij} )\partial ^{\prime} _j $.
Since
$\partial  _{i '} \partial  _{i} =
\partial  _{i } \partial  _{i '} $
and
$\partial  ^{\prime }_{j '}  \partial ^{\prime} _j
 =
  \partial  ^{\prime }_{j}  \partial ^{\prime} _{j'} $,
 exchaging $i$ with $i'$ yields
$\sum _{j = 1} ^{n} \partial  _{i '} (a _{ij} )\partial ^{\prime} _j
=
\sum _{j = 1} ^{n} \partial  _{i } (a _{i'j} )\partial ^{\prime} _j $.
Hence, for any $i,i', j$, we have
$ \partial  _{i '} (a _{ij} )
=
\partial  _{i } (a _{i'j} )$
and then (by symmetry)
$ \partial ' _{i '} (a '_{ij} )
=
\partial ' _{i } (a '_{i'j} )$.

2) By symmetry and $\O _X$-linearity, it is sufficient to check
that both actions of $\partial _1$ on $dlog\, b_1 \wedge \cdots \wedge dlog\, b_n $
coincides.
With the first choice, this is straightforward that we get $0$.
Now, we consider the action $\partial _1$ on $dlog\, b_1 \wedge \cdots \wedge dlog\, b_n $ for the second choice of log $p$-basis.
Since $\partial _{1} = \sum _{j =1} ^{n} a _{1j} \partial ^{\prime} _{j}$,
we get
$dlog\, b_1 \wedge \cdots \wedge dlog\, b_n  \cdot \partial _{1}
=
(| A '| \, dlog\, b'_1 \wedge \cdots \wedge dlog\, b'_n )  \cdot \partial _{1}
=
 -\sum _{j =1} ^{n}  \partial ^{\prime} _{j} ( a _{1j} |A'|) dlog\, b'_1 \wedge \cdots \wedge dlog\, b'_n$.
 Hence, we have to check
 $\sum _{j =1} ^{n}   \partial ^{\prime} _{j} ( a _{1j} |A'|) =0$.

a) We compute
$a _{1j} |A'| = \sum _{\sigma \in S _n, \sigma (1)=j} (-1) ^{\epsilon (\sigma)} \prod _{i=2} ^{n} a ' _{\sigma (i) i}.$
Indeed, let $L ' _1, \dots, L ' _n$ be the rows of $A'$.
We remark that $a _{1j} |A'|$ is equal to the determinant of
the matrix $A'$ whose row $L' _j$ is replaced by
$a _{1 j} L' _j$ and then by
$\sum _{l = 1} ^{n}a _{1l } L ' _l$.
Since $A A ' = I _n$, we get
$\sum _{l = 1} ^{n}a _{1l } L ' _l
= ( 1, 0, \dots , 0)$.
This yields the desired formula.

b) We have
$\sum _{j =1} ^{n}   \partial ^{\prime} _{j} ( a _{1j} |A'|)
=
\sum _{\sigma \in S _n, l \in \{2,\dots, n\}} (-1) ^{\epsilon (\sigma)}\partial ' _{\sigma (1)} (a ' _{\sigma (l) l}) \prod _{i=2, i\not =l} ^{n} a ' _{\sigma (i) i}$.
Indeed, this is a consequence of the formula  $\partial ' _{\sigma (1)} (\prod _{i=2} ^{n} a ' _{\sigma (i) i})
=
\sum _{l=2} ^{n} \partial ' _{\sigma (1)} (a ' _{\sigma (l) l}) \prod _{i=2, i\not =l} ^{n} a ' _{\sigma (i) i}$,
and of that of part a).

c) We define on $S _n \times \{2,\dots, n \}$ the following equivalence relation.
Two elements $(\sigma, l)$ and $ ( \sigma ',l')$ of
$S _n \times \{2,\dots, n \}$ are equivalent if either $( \sigma ',l')= (\sigma, l)$
or
$( \sigma ',l') = (\sigma \circ (1,l), l)$.
Let 
$ (\sigma , l)$
and
$( \sigma ',l')= (\sigma \circ (1,l), l)$
be a class of 
$S _n \times \{2,\dots, n \}$.
Using the formula checked in the part 1) of the proof,
we get
$(-1) ^{\epsilon (\sigma ')}\partial ' _{\sigma '(1)} (a ' _{\sigma ' (l) l}) \prod _{i=2, i\not =l} ^{n} a ' _{\sigma '(i) i}
+
(-1) ^{\epsilon (\sigma)}\partial ' _{\sigma (1)} (a ' _{\sigma (l) l}) \prod _{i=2, i\not =l} ^{n} a ' _{\sigma (i) i}
=
0$.
Since we have a partition of $S _n \times \{2,\dots, n \}$
by its classes, this implies
$\sum _{\sigma \in S _n, l \in \{2,\dots, n\}} (-1) ^{\epsilon (\sigma)}\partial ' _{\sigma (1)} (a ' _{\sigma (l) l}) \prod _{i=2, i\not =l} ^{n} a ' _{\sigma (i) i}
=0$.
We conclude using 2.b).

 \end{proof}

\begin{prop}
\label{rightDdag-moduleOmega}
Let $\X$ be a log $p$-smooth $\S$-log-formal scheme.
We suppose that
$\underline \X$  
has no $p$-torsion.
We put $\omega _{\X/\S} := \underleftarrow{\lim} _i \,
\omega _{X _i /S _i}$.
There exists a canonical structure of right $\D ^{\dag} _{\X/\S} $-module on $\omega _{\X/\S}$.
It is characterized by the following local formula:
suppose that $\X$ is endowed with a log $p$-basis $(b  _{\lambda}) _{\lambda =1,\dots, n}$.
Let $d log b _\lambda$ be the image of $\eta _\lambda$ in $\Omega ^{1} _{\X/\S}$.
Then, for any integer $m$, for any differential operator $P \in \D ^{(m)} _{\X/\S}$ and $a \in \O _\X$ we have
\begin{equation}
\label{rightD-mod-omegaformal}
(a \ dlog\, b_1 \wedge \cdots \wedge dlog\, b_n) \cdot P := \widetilde{P} (a) \ dlog\, b_1 \wedge \cdots \wedge dlog\, b_n.
\end{equation}

\end{prop}

\begin{proof}
Using \ref{rightD-mod}, we get a canonical structure of
right $\widehat{\D} ^{(0)} _{\X} $-module on $\omega _{\X/\S} = \underleftarrow{\lim} _i \,
\omega _{X _i /S _i}$.
Hence, we get a structure of right
$\widehat{\D} ^{(0)} _{\X,\Q} $-module on $\omega _{\X/\S,\Q}$.
Since $\D ^{(m)} _{\X/\S}  \subset \widehat{\D} ^{(0)} _{\X/\S,\Q} $,
we get a structure of right $\D ^{(m)} _{\X/\S}$-module on
$\omega _{\X/\S,\Q}$.
Let us check that $\omega _{\X/\S}$ is a sub
$\D ^{(m)} _{\X/\S}$-module of
$\omega _{\X/\S,\Q}$.
Since this is local,
we can suppose that $\X$ is endowed with a log $p$-basis $(b  _{\lambda}) _{\lambda =1,\dots, n}$.
We compute that the right $\D ^{(m)} _{\X/\S}$-action on
$\omega _{\X/\S,\Q}$ is given by the formula \ref{rightD-mod-omegaformal}.
This implies that $\omega _{\X/\S}$ is a sub $\D ^{(m)} _{\X/\S}$-module of
$\omega _{\X/\S,\Q}$.
Using (a right log version of) \cite[3.1.3]{Be0},
this yields that
$\omega _{\X/\S}$ is endowed with a canonical structure of
$\widehat{\D} ^{(m)} _{\X}$-module.
Since these structures are compatible with
$\widehat{\D} ^{(m)} _{\X} \to \widehat{\D} ^{(m+1)} _{\X}$, we are done.
\end{proof}

\begin{coro}
Let $\X$ be a log $p$-smooth $\S$-log-formal scheme such that
$\underline \X$  
has no $p$-torsion.
Let $i $ be an integer.
There exists a canonical structure of right $\D ^{(m)} _{X _i/S _i} $-module on $\omega _{X _i /S _i}$.
It is characterized by the following local formula:
suppose that $\X$ is endowed with a log $p$-basis $(b  _{\lambda}) _{\lambda =1,\dots, n}$.
Let $d log b _\lambda$ be the image of $\eta _\lambda$ in $\Omega ^{1} _{X _i/S _i}$.
Then, for any integer $m$, for any differential operator $P \in \D ^{(m)} _{X _i/S _i}$ and $a \in \O _{X _i}$ we have
\begin{equation}
\label{coro-rightD-mod-omegaformal}
(a \ dlog\, b_1 \wedge \cdots \wedge dlog\, b_n) \cdot P := \widetilde{P} (a) \ dlog\, b_1 \wedge \cdots \wedge dlog\, b_n.
\end{equation}
\end{coro}

\begin{proof}
This is a consequence of  \ref{rightDdag-moduleOmega}.
\end{proof}

\begin{coro}
\label{twisting}
Let $\X$ be a log $p$-smooth $\S$-log-formal scheme such that
$\underline \X$  
has no $p$-torsion.
The functor
$-\otimes _{\O _{X _i}} \omega _{X _i}$
(resp. $-\otimes _{\O _\X} \omega _\X$) is an equivalence of categories
between that of
left $\D ^{(m)} _{X _i/S _i} $-modules
(resp. left $\D ^{\dag} _{\X/\S} $-modules)
and that of right $\D ^{(m)} _{X _i/S _i} $-modules
(resp. right $\D ^{\dag} _{\X/\S} $-modules).
The functor
$-\otimes _{\O _{X _i}} \omega _{X _i} ^{-1}$
(resp. $-\otimes _{\O _\X} \omega _\X ^{-1}$) is a quasi-inverse functor.
Both functors preserve the coherence.
These functors are the "twisted structures" of $\D$-module.
\end{coro}

\subsection{Pushforwards, extraordinary pull-backs and duality}

Let $f \colon \X \to \Y$ be a morphism of log $p$-smooth $\S$-log-formal schemes.
We suppose that the formal $\V$-schemes
$\underline \X$ and $\underline \Y$ 
have no $p$-torsion.
We can follow Berthelot's construction of pushforwards, extraordinary pull-backs and duality
as explained in \cite{Beintro2} or \cite{Be2}.
For the reader, let's briefly sketch the construction.

\begin{empt}
By functoriality, the left $\D ^{(m)} _{X _i /S _i}$-module
$f ^{*} \D ^{(m)} _{Y _i /S _i}$ is in fact endowed with a structure of
$(\D ^{(m)} _{X _i /S _i}, f ^{-1}\D ^{(m)} _{Y _i /S _i}) $-bimodule
which we will denote by
$ \D ^{(m)} _{X _i \to Y _i /S _i}$.
By twisting (see \ref{twisting}), we get a
$(f ^{-1}\D ^{(m)} _{Y _i /S _i}, \D ^{(m)} _{X _i /S _i}) $-bimodule
 by setting
$  \D ^{(m)} _{Y _i \leftarrow X _i /S _i}
  :=
\omega _{X _i /S _i}    \otimes _{\O _{X _i}}  f ^{*} _\mathrm{g} (\D ^{(m)} _{Y _i /S _i}
  \otimes _{\O _{Y _i}} \omega _{Y _i /S _i} ^{-1})$,
  where the index $g$ means that we choose the left structure of the
  left $  \D ^{(m)} _{Y _i /S _i}$-bimodule
 $ \D ^{(m)} _{Y _i /S _i}
  \otimes _{\O _{Y _i}} \omega _{Y _i /S _i} ^{-1}$.
Next, we put
$ \widehat{\D} ^{(m)} _{\X\to \Y /\S}
:=
\underleftarrow{\lim} _i \,
\D ^{(m)} _{X _i \to Y _i/S _i}$ and
$ \widehat{\D} ^{(m)} _{\Y\leftarrow \X /\S}
:=
\underleftarrow{\lim} _i \,
\D ^{(m)} _{Y _i \leftarrow X _i/S _i}$.
Finally,
$\D ^{\dag} _{\X\to \Y /\S}
:=
\underrightarrow{\lim} _{m}\,
 \widehat{\D} ^{(m)} _{\X\to \Y /\S}$
and
$\D ^{\dag} _{\Y\leftarrow \X /\S}
:=
\underrightarrow{\lim} _{m}\,
\widehat{\D} ^{(m)} _{\Y\leftarrow \X /\S} $.
Let 
$d _{\X}$ (resp. $d _{\Y}$) be the Krull dimension of
$\underline{\X}$ (resp. $\underline{\Y}$).
We denote by 
$D ^\mathrm{b}
(\D ^{\dag} _{\X /\S,\Q} )$,
(resp. $D ^\mathrm{b} _\mathrm{coh}
(\D ^{\dag} _{\X/\S,\Q} )$,
resp. $D _\mathrm{parf}
(\D ^{\dag} _{\X /\S,\Q} )$)
the derived category of bounded (resp. and bounded and coherent, resp. perfect) complexes of left 
$\D ^{\dag} _{\X /\S,\Q}$-modules.
\begin{enumerate}
\item As in \cite[4.3.2.2]{Beintro2},
we get the functor
$f ^!  \colon
D ^\mathrm{b} _\mathrm{coh}
(\D ^{\dag} _{\Y/\S,\Q} )
\to
D ^\mathrm{b}
(\D ^{\dag} _{\X /\S,\Q} )$
by setting, for any object
$\FF$ of
$D ^\mathrm{b} _\mathrm{coh}
(\D ^{\dag} _{\Y/\S,\Q} )$,
$$ f ^{!} (\FF) :=
\D ^{\dag} _{\X\to \Y /\S,\Q}
\otimes ^{\L}  _{f ^{-1} \D ^{\dag} _{\Y/\S,\Q}  }
f ^{-1} \FF [ d _{\X} -d _\Y].$$

\item As in \cite[4.3.7.1]{Beintro2},
we get the functor
$f _+  \colon
D ^\mathrm{b} _\mathrm{coh}
(\D ^{\dag} _{\X/\S,\Q} )
\to
D ^\mathrm{b}
(\D ^{\dag} _{\Y /\S,\Q} )$
by setting, for any object
$\E$ of
$D ^\mathrm{b} _\mathrm{coh}
(\D ^{\dag} _{\X/\S,\Q} )$,
$$
 f _{+} (\E) := \R f _{*} (
\D ^{\dag} _{\Y\leftarrow  \X /\S,\Q}
\otimes ^{\L} _{\D ^{\dag} _{\X/\S,\Q}  }
\E ).$$

\item As in \cite[4.3.10]{Beintro2},
we get the functor
$\DD _{\X}  \colon
D _\mathrm{parf}
(\D ^{\dag} _{\X/\S,\Q} )
\to
D _\mathrm{parf}
(\D ^{\dag} _{\X /\S,\Q} )$
by posing, for any object
$\E$ of
$D _\mathrm{parf}
(\D ^{\dag} _{\X/\S,\Q} )$,
$$
 \DD _{\X} (\E) := \R \mathcal{H}om _{\D ^{\dag} _{\X/\S,\Q}}
(\E , \D ^{\dag} _{\X/\S,\Q} \otimes _{\O _{\X,\Q}} \omega ^{-1} _{\X/\S,\Q}) [ d _{\X}].$$

\end{enumerate}

\end{empt}

\small
\def\cprime{$'$}
\providecommand{\bysame}{\leavevmode ---\ }
\providecommand{\og}{``}
\providecommand{\fg}{''}
\providecommand{\smfandname}{et}
\providecommand{\smfedsname}{\'eds.}
\providecommand{\smfedname}{\'ed.}
\providecommand{\smfmastersthesisname}{M\'emoire}
\providecommand{\smfphdthesisname}{Th\`ese}

\bigskip
\noindent Daniel Caro\\
Laboratoire de Mathématiques Nicolas Oresme\\
Université de Caen
Campus 2\\
14032 Caen Cedex\\
France.\\
email: daniel.caro@unicaen.fr

\bigskip

\noindent David Vauclair\\
Laboratoire de Mathématiques Nicolas Oresme\\
Université de Caen
Campus 2\\
14032 Caen Cedex\\
France.\\
email: david.vauclair@unicaen.fr

\end{document}